\documentclass[11pt,a4paper]{article}

\usepackage{pdfsync}

\usepackage{graphicx}
\usepackage{xcolor}

\usepackage{fancyhdr}

\usepackage{float}   % For inserting image options

\usepackage[margin=1in]{geometry}  % set the margins to 1in on all sides
\usepackage{graphicx}              % to include figures
\usepackage{amsmath}               % great math stuff
\usepackage{amsfonts}              % for blackboard bold, etc
\usepackage{bbm}                   % for identity function
\usepackage{amsthm}                % better theorem environments
\usepackage{xcolor}
\usepackage[pdftex,%
a4paper=true,%
bookmarks=true,%
bookmarksnumbered=true,%
colorlinks=true,%
urlcolor=blue,%
pdfstartview=FitH%
]{hyperref}

\usepackage{lipsum}

\newcommand\blfootnote[1]{%
  \begingroup
  \renewcommand\thefootnote{}\footnote{#1}%
  \addtocounter{footnote}{-1}%
  \endgroup
}
  
\newtheorem{thm}{Theorem}[section]
\newtheorem{lem}[thm]{Lemma}
\newtheorem{prop}[thm]{Proposition}
\newtheorem{defi}[thm]{Definition}
\newtheorem{cor}[thm]{Corollary}
\newtheorem{rmk}[thm]{Remark}
\newtheorem{exm}[thm]{Example}

\newcommand{\HH}{\mathcal{H}}     % for the set of Halfspaces or the functor
\newcommand{\MM}{\mathcal{M}}     % for the ultrafilter functor
\newcommand{\BB}{\mathcal{B}}     % for the bridge
\newcommand{\hh}{\mathfrak{h}}    % for denoting halfspaces
\newcommand{\uu}{\mathfrak{u}}    % for denoting ultrafilters
\newcommand{\PP}{\mathcal{P}}     % for the set of subsets 
\newcommand{\RR}{\mathbb{R}}      % for Real numbers
\newcommand{\ZZ}{\mathbb{Z}}      % for Integers
\newcommand{\NN}{\mathbb{N}}      % for natural numbers

\usepackage{hyperref}
\usepackage{xcolor}
\usepackage{multicol}
\hypersetup{
colorlinks=true, %colorise les liens
breaklinks=true, %permet le retour a la ligne dans les liens trop longs
urlcolor= blue, %couleur des hyperliens
linkcolor= blue %couleur des liens internes
}

\begin{document}

%\mu(\Delta (\pi_{C_1}(x),c_1))+\mu(\Delta (c_1,c_2))+\mu(\Delta (c_2,\pi_{C_2}(x)))    

%\renewcommand{\refname}{References}
%\bibliographystyle{alpha}

% fancy header pour la page de garde : rien
\pagestyle{fancy}
\fancyhead[L]{ }
\fancyhead[R]{}
\fancyfoot[C]{}
\fancyfoot[L]{ }
\fancyfoot[R]{}
\renewcommand{\headrulewidth}{0pt} 
\renewcommand{\footrulewidth}{0pt}

%----------- TITRE --------------%
\newcommand{\montitre}{Isometric actions on locally compact finite rank median spaces}

\newcommand{\auteur}{\textsc{  Lamine Messaci
}}
\newcommand{\affiliation}{Laboratoire J.-A. Dieudonn\'e \\ Universit\'e C\^{o}te d'Azur\\ Parc Valrose, 06108 Nice Cedex 02,   France\\
\url{messaci@unice.fr
} }

 \begin{center}
{\bf  {\LARGE \montitre}}\\ \bigskip \bigskip
 {\large\auteur}\\ \bigskip \smallskip
 \affiliation \\ \bigskip
\today
 \end{center}
\begin{abstract}
We prove that a connected locally compact median space of finite rank which admits a transitive action is isometric to $\RR^n$ endowed with the $\ell^1$-metric. In the other side, replacing the transitivity assumption on the group of isometries by a certain regularity of the action on the compactification of the space, we show that all orbits are discrete. In our way to prove these results, we give a characterization of the compactness in complete median spaces of finite rank by the combinatorics of their halfspaces.
\end{abstract}

%\medskip

 %{\bf AMS Classification :} 35G25; 35Q35; 76B15; 35B65; 35B35.

%\medskip

\tableofcontents

%-----------------------------------%

 %------- fancy header pour la suite du document
 \pagestyle{fancy}
\fancyhead[R]{\thepage}
\fancyfoot[C]{}
\fancyfoot[L]{}
\fancyfoot[R]{}
\renewcommand{\headrulewidth}{0.2pt} %0.4
\renewcommand{\footrulewidth}{0pt} %0

%%%%%%%%%%%%%%%%%%%%%%%%%%%%%%%%%%%%%%%%%%%%%%%%%%%%%%%%%%%%%%%%%%%%%%%%%%%%%%%%%%%%%%%%%%%%%%%%%%%%%%%%%%%%%%%%%%%%%%%%%%%%%%%%%%%%%%%%%%%%%%%%%%%%%%%%%%%%%%%%%%%%%%%%%%%%%%%%%%%%%%%%%%%%%%%%%%%%%%%%%%%%%%%%%%%%%%%%%%%%%%%%%%%%%%%%%%%%%%%%%%%%%%%%%%%%%%%%%%%%%%%%%%%%%%%%%%%%%%%%%%%%%%%%%%%%%%%%%%%%%%%%%%%%%%%%%%%%%%%%%%%%%%%%%%%%%%%%%%%%%%%%%%%%%%%%%%%%
\section{Introduction}
%%%%%%%%%%%%%%%%%%%%%%%%%%%%%%%%%%%%%%%%%%%%%%%%%%%%%%%%%%%%%%%%%%%%%%%%%%%%%%%%%%%%%%%%%%%%%%%%%%%%%%%%%%%%%%%%%%%%%%%%%%%%%%%%%%%%%%%%%%%%%%%%%%%%%%%%%%%%%%%%%%%%%%%%%%%%%%%%%%%%%%%%%%%%%%%%%%%%%%%%%%%%%%%%%%%%%%%%%%%%%%%%%%%%%%%%%%%%%%%%%%%%%%%%%%%%%%%%%%%%%%%%%%%%%%%%%%%%%%%%%%%%%%%%%%%%%%%%%%%%%%%%%%%%%%%%%%%%%%%%%%%%%%%%%%%%%%%%%%%%%%%%%%%%%%%%%%%%
\color{black}
In the late decades, CAT$(0)$ cube complexes gained an important place in the field of geometric group theory and low dimensional topology. Many interesting properties can be deduced for groups acting in a particular way on such spaces, such as the linearity (see \cite{Hag-Wise-Special}). It was shown that fundamental groups of three dimensional closed hyperbolic manifolds act nicely on CAT($0$) cube complexes and this leads to the settlement of the virtually Haken conjecture and the Thurston's virtually fibered conjecture (see \cite{Berg-Hag-Wise}, \cite{Berg-Wise}, \cite{Hag-Wise-Hierarchy}, \cite{Agol}).\par 
Zero-skeletons of CAT$(0)$ cube complexes exhibit a similar behaviour to trees in the sense that to any three points, there exist a ``median" point, endowing the space with a structure of a median space. A median space is a metric space such that for any three points, the three intervals relating each two points intersect in a unique point. The interval between two point $a,b\in (X,d)$, denoted by $[a,b]$, is the set of point $x\in X$ such that $d(a,b)=d(a,x)+d(x,b)$.\par 
Natural examples of median spaces are given by $\RR$-trees where in these cases, the intervals coincide with the geodesics, and it is the only case where this occurs.\par
First examples of median spaces go back to \cite{Birkhoff_median} and are given by metric distributive lattices, see \cite{Birkhoff_lattice} Ch. V, §9 for a definition of metric lattices.\par
%In the past decade, median spaces drove attention in geometric group theory as they give a dynamical characterization of Kazhdan property (T) and offer a unified framework for the study of group action on both real trees and CAT$(0)$ cubes complexes. \par
Zero-skeletons of CAT$(0)$ cube complexes when endowed with the combinatorial metric constitute an important class of median spaces. More precisely, a metric graph is a median space if and only if it arises as the zero-skeleton of a CAT($0$) cube complex (see \cite{Chepoi}).\par 
Median spaces of finite rank generalize finite-dimensional CAT($0$) cube complexes the same way $\RR$-trees generalize simplicial trees where the rank is the natural notion to characterize dimension in the median setting. Loosely speaking, it detects the highest dimension of discrete cubes, endowed with the $\ell^1$-metric, that can be isometrically embedded into the space, see Definitions \ref{definition_rank} for a rigorous description. %In the other hand, it was shown in \cite{BowditchCAT(0)} that any median space of finite rank admits a bilipschitz equivalent metric which is CAT$(0)$.

\par Many results of isometric groups action on real trees and CAT$(0)$ cube complexes extend to the case of finite rank median spaces. For instance, a superrigidity phenomenon, which yields the fixed point property, was shown in \cite{Shalom} and \cite{ChaIF} for actions of irreducible lattices in a product of locally compact group, on trees and CAT$(0)$ cube complexes respectively and was extended to the case of finite rank median spaces in \cite{Fior_superrigidity}. Another example is given by a version of Tits alternative for group acting on CAT($0$) cube complexes, shown in \cite{SagW}, \cite{SagC}, and which was proved to hold in the case of finite rank median spaces as presented in \cite{Fior_tits}.
% on theses spaces, of irreducible lattices in a product of locally compact group which yield the fixed point property for such actions (see \cite{Shalom},\cite{ChaIF},\cite{Fior_superrigidity}). A version \par 
\par  This naturally leads to the question of whether any geometric action, i.e. a properly discontinuous cocompact action, on a finite rank median space gives rise to a geometric action on a finite-dimensional CAT($0$) cube complex, see \cite{ChaD_hyperbolic} subsection 1.b. No counterexamples exist for the finite rank case. The is due to the fact that generally, the obstructions indicating that a group cannot act properly on a CAT$(0)$ cube complexes are generally the same for actions on finite rank median spaces, for instance the Solvable subgroup theorem (\cite{SagW}, \cite{Hag}, \cite{Fior_tits}, \cite{Fior_auto}) or the Kazhdan's property (T) (\cite{Niblo-Reeves}, \cite{Haglund-Paulin}, \cite{Niblo-Roller},\cite{ChaDH_median}). We note that the question is answered negatively in the case of infinite rank median spaces as it was shown in \cite{ChaD_hyperbolic} that irreducible cocompact lattices in a product of $\mathrm{SO}(n,1)$ act geometrically on median spaces of infinite rank. These lattices are known to not act properly on finite dimensional CAT($0)$ cube complexes and median spaces of finite rank (see \cite{ChaIF} and \cite{Fior_superrigidity}).\par

\par In this paper, we will be investigating the case of isometric actions on complete locally compact median space of finite rank, with the purpose of tackling the question cited in the paragraph above.
\par  We first give the following classification of homogeneous connected locally compact median spaces of finite rank: 

%%%%%%%%%%%%%%%%%%%%%%%%%%%%%%%%%%%%%%%%%%%%%%%%%%%%%%%%%%%%%%%%%%%%%%%%%%%%%%%%%%%%%%%%%%%%%%%%%%%%%%%%%%%%%%%%%%%%%%%%%%%%%%%%%%%%%%%%%%%%%%%
\begin{thm}\label{rigidity}
Let $X$ be a connected locally compact median space of finite rank which admits a transitive action, then $X$ is isometric to $(\RR^n,l^1)$.
\end{thm} 
%%%%%%%%%%%%%%%%%%%%%%%%%%%%%%%%%%%%%%%%%%%%%%%%%%%%%%%%%%%%%%%%%%%%%%%%%%%%%%%%%%%%%%%%%%%%%%%%%%%%%%%%%%%%%%%%%%%%%%%%%%%%%%%%%%%%%%%%%%%%%%%
Examples of connected homogeneous median spaces of finite rank are given by asymptotic cones of coarse median groups, that is of groups which admit a coarse median structure (see \cite{Bowd_conv} for a definition). In particular, the asymptotic cones of non-elementary hyperbolic groups are $2^{\aleph_0}$-universal real tree (see \cite{universal-tree} for an explicit construction), that is a connected real tree where any point is a branching point with a number of branches in the order of the continuum. Another example which is of higher rank is given by the asymptotic cone of the mapping class group (\cite{Drutu-Sapir},\cite{Bowd_conv}). Theses examples are far from being locally compact. More generally, an asymptotic cone of a finitely generated group is locally compact if and only if the group is nilpotent (see \cite{Gromov-Nil}, \cite{Drutu-Nil}). With regards to homogeneous connected median spaces of finite rank, Theorem \ref{rigidity} says that the only examples of such spaces which are locally compact are $(\RR^n,l^1)$.
\par 

%The strategy to prove Theorem \ref{rigidity} consists of considering the set of halfspaces which are ``branched" at an arbitrary point.  The set $\HH_x(X)$ of halfspaces branched at a point $x\in X$ in a complete median space of finite rank is the set of halfspaces $\hh \subset X$ such that $x\in \bar{\hh}\cap \bar{\hh^c}$. The set $\HH_x(X)$ can be seen as the extension of the notion of the valency at a point $x$ in an $\RR$-tree to the case of complete median space of finite rank.\par 

The assumption that the group of isometries acts transitively is strong enough that no additional assumptions need to be made regarding their action on the boundary of the space. Here the boundary that we are referring to is the Roller boundary which is more adequate in the setting of median spaces (see \cite{Roller} and \cite{Fior_median_property} for a definition). Now assuming some minimality and non elementarity conditions on the actions of the group isometry of the median space, we obtain that all orbits are discrete:
%%%%%%%%%%%%%%%%%%%%%%%%%%%%%%%%%%%%%%%%%%%%%%%%%%%%%%%%%%%%%%%%%%%%%%%%%%%%%%%%%%%%%%%%%%%%%%%%%%%%%%%%%%%%%%%%%%%%%%%%%%%%%%%%%%%%%%%%%%%%%%%%%%%%%%%%%%%%%%%%%%%%%%%%%%%%%%%%%%%%%%%%%%%%%%%%%%%%%%%%%%%%%%%%%%%%%
\begin{thm}\label{Theorem_discrete_orbit}
Let $X$ be an irreducible complete connected locally compact median space of finite rank. Let us assume that the action of $G:=Isom(X)$ on $X$ is Roller non elementary, Roller minimal and minimal. Then any $G$-orbit is discrete.
\end{thm}
%%%%%%%%%%%%%%%%%%%%%%%%%%%%%%%%%%%%%%%%%%%%%%%%%%%%%%%%%%%%%%%%%%%%%%%%%%%%%%%%%%%%%%%%%%%%%%%%%%%%%%%%%%%%%%%%%%%%%%%%%%%%%%%%%%%%%%%%%%%%%%%
An action on a median space $X$ is said to be Roller non elementary and Roller minimal if it has no finite orbit and no invariant closed convex subset, respectively, in the Roller compactification of the space. The action is said to be minimal if there is no invariant convex subset (not necessarily closed) in $X$. These assumptions may seem too strong for the conclusion of Theorem \ref{Theorem_discrete_orbit}, especially the minimality condition, but they are necessary as shown in the examples exhibited in Subsection \ref{Subsection_discussion_condition}. 
\par From the discrete orbits of the spaces considered in Theorem \ref{Theorem_discrete_orbit}, one is tempted to extract a cubical structure on the spaces. Having just the data of the orbit being discrete is not enough to extract an action on a CAT$(0)$ cube complex. The orbits are not necessarily closed under the median ternary operation, i.e. the ternary operation which associates to each triple of points their unique median point, and the smallest median subspace containing the orbits can be dense in the space. A canonical way to construct an action on a CAT$(0)$ cube complex is through the process of cubulation, by considering the dual of a system of discrete walls on which the group acts (see \cite{Sageev}, \cite{Haglund-Paulin}, \cite{Roller}, \cite{Chatterji-Niblo}, \cite{Nica}). Hence, with regards to the discrete orbits, another attempts to construct an action on a CAT$(0)$ cube complex is by considering a discrete family of halfspaces which separate the points of the orbit and which is invariants under the action of the group. This attempt also fails as one may consider for instance $\Gamma=\ZZ^2$ irrationally embedded into $(\RR^2,\ell^1)$ where any invariant subset of halfspaces is non discrete (although for this example, there is a choice of discrete wall which is invariant but its choice is not canonical and is independent from the median structure of the space).

%The argument of the proof relies in an essential way on the machinery developed in \cite{Fior_superrigidity}.\par
Assuming that the median space is locally compact imposes a certain configuration on the halfspaces which are transverse to a ball. A halfspace is a convex subset such that its complement is also convex. In our way to prove the theorems cited above, we give the following characterization of compact subsets by the combinatoric of the halfspaces which are transverse to the subset:
%%%%%%%%%%%%%%%%%%%%%%%%%%%%%%%%%%%%%%%%%%%%%%%%%%%%%%%%%%%%%%%%%%%%%%%%%%%%%%%%%%%%%%%%%%%%%%%%%%%%%%%%%%%%%%%%%%%%%%%%%%%%%%%%%%%%%%%%%%%%%%%
\begin{thm}\label{local_compactness}
Let $X$ be a complete  connected median space of rank $n$. Let $C$ be a closed bounded subset of $X$. Then the following are equivalent:
\begin{enumerate}
\item The subset $C$ is compact.
%\item For any $a_0\in C$ and $\epsilon >0$, there exist $D_1,....,d_{k_\epsilon}\in \DD_{a_0,C}$ such that for any $D\in\DD_{a_0,C}$, we have $\mu (D\backslash D_i\leq\epsilon$ for some $i\in\{1,...,k_\epsilon\}$.
\item For any $x_0\in C$ and $\epsilon >0$, there exist $x_1,...,x_{k_\epsilon}\in C$ such that for any $x\in C$ we have $d(x,[x_0,x_i])\leq \epsilon$ for some $i\in\{1,...,k_\epsilon\}$.
\item For any $\epsilon>0$, if $\HH_\epsilon$ is a family of pairwise disjoint halfspaces transverse to $C$ and of depth bigger than $\epsilon$ in the convex hull of $C$, then it is finite.
\end{enumerate} 
\end{thm}
%%%%%%%%%%%%%%%%%%%%%%%%%%%%%%%%%%%%%%%%%%%%%%%%%%%%%%%%%%%%%%%%%%%%%%%%%%%%%%%%%%%%%%%%%%%%%%%%%%%%%%%%%%%%%%%%%%%%%%%%%%%%%%%%%%%%%%%%%%%%%%%
Where the depth of a subset $H\subset X$ inside another subset $C$ is defined as $depth_C(H):=sup(\{d(x,H^c)\ |\ x\in C \}$. \par
Theorem \ref{local_compactness} falls within the framework of the duality between the category of median spaces and the category of pointed measured partially ordered sets with inverse operation (see Subsections \ref{halfspaces}, \ref{duality} for definitions and \cite{Roller}, \cite{ChaDH_median}, \cite{Fior_median_property} for a deeper delve). It characterizes the subcategory of the latter category which is dual to the subcategory of compact connected median space of finite rank. \par

\subsection{Structure of the paper}
%%%%%%%%%%%%%%%%%%%%%%%%%%%%%%%%%%%%%%%%%%%%%%%%%%%%%%%%%%%%%%%%%%%%%%%%%%%%%%%%%%%%%%%%%%%%%%%%%%%%%%
%%%%%%%%%%%%%%%%%%%%%%%%%%%%%%%%%%%%%%%%%%%%%%%%%%%%%%%%%%%%%%%%%%%%%%%%%%%%%%%%%%%%%%%%%%%%%%%%%%%%%%%%%%%%%%%%%%%%%%%%%%%%%%%%%%%%%%%%%%%%%%%
\ \ \ \ Section 2 is a quick overview of the geometry of median spaces, the structure of their halfspaces and the duality between them.\par
Section 3 is devoted to the proof of an embedding lemma of the convex hull between two convex subsets. The lemma will be needed in the proof of Theorem \ref{rigidity}.\par
In section 4, we recall results about the compactness of the convex hull before proving Theorem \ref{local_compactness} about the characterization of compact subsets.\par
Section 5 is devoted to prove Theorem \ref{rigidity} regarding the classification of homogeneous connected locally compact median spaces of finite rank.\par
In the last section, we first recall definitions and results from the machinery developed in \cite{Fior_superrigidity}, then discuss the hypothesis of Theorem \ref{Theorem_discrete_orbit} before proving it. \\

 \textbf{\textit{Acknowledgments:}} \ The author expresses his sincere gratitude to his PhD advisor Professor Indira Chatterji for insightful discussions and unwavering encouragements throughout the preparation of this work. The author also warmly thanks Elia Fioravanti, Mark Hagen and Graham Niblo for their valuable suggestions and comments on an earlier version.   

%%%%%%%%%%%%%%%%%%%%%%%%%%%%%%%%%%%%%%%%%%%%%%%%%%%%%%%%%%%%%%%%%%%%%%%%%%%%%%%%%%%%%%%%%%%%%%%%%%%%%%%%%%%%%%%%%%%%%%%%%%%%%%%%%%%%%%%%%%%%%%%%%%%%%%%%%%%%%%%%%%%%%%%%%%%%%%%%%%%%%%%%%%%%%%%%%%%%
\section{Preliminaries}%%%%%%%%%%%%%%%%%%%%%%%%%%%%%%%%%%%%%%%%%%%%%%%%%%%%%%%%%%%%%%%%%%%%%%%%%%%%%%%%%%%%%%%%%%%%%%%%%%%%%%%%%%%%%%%%%%%%%%%%%%%%%%%%%%%%%%%%%%%%%%%%%%%%%%%%%%%%%%%%%%%%%%%%%%%%%%%%%%%%%%%%%%%%%%%%%%%%%%%%%%%%%%%%%%%%%%%%%%%%%%%%%%%%%%%%%%%%%%%%%%%%%%%%%%%%%%%%%%%%%%%%%%%%%%%%%%%%%%%%%%%%%%%%%%%%%%%%%%%%%%%%%%%%%%%%%%%%%%%%%%%%%%%%%%%%%%%%%%%%%%%%%%%%%%%%%%%%
\subsection{Median spaces and algebras}
%%%%%%%%%%%%%%%%%%%%%%%%%%%%%%%%%%%%%%%%%%%%%%%%%%%%%%%%%%%%%%%%%%%%%%%%%%%%%%%%%%%%%%%%%%%%%%%%%%%%%%%%%%%%%%%%%%%%%%%%%%%%%%%%%%%%%%%%%%%%%%%%%%%%%%%%%%%%%%%%%%%%%%%%%%%%%%%%%%%%%%%
A \textbf{\textit{median space}} is a metric space $(X,d)$ such that for any triple of points $a,b,c\in X$, there exists a unique point $m\in X$ such that 
\[
[a,b]\cap [b,c]\cap [a,c] =\{ m\}
\]
where $[a,b]=\{x\in X \ |\  d(a,b)=d(a,x)+d(x,b)\}$ is the \textbf{\textit{interval}} between $a$ and $b$. Let us denote by $m:X\times X\times X:\rightarrow X$ the ternary operation that associates to each triple $(x,y,z)$ their unique median point $m(x,y,z)$. We call the latter the \textbf{\textit{median operation}}. We say that a map $f$ between two median spaces is a \textbf{\textit{median morphism}} if it respects the median operation, i.e. $f(m_X(x,y,z))=m_Y(f(x),f(y),f(z))$. Note that any isometry between two median spaces is a median morphism.
%%%%%%%%%%%%%%%%%%%%%%%%%%%%%%%%%%%%%%%%%%%%%%%%%%%%%%%%%%%%%%%%%%%%%%%%%%%%%%%%%%%%%%%%%%%%%%%%%%%%%%%%%%%%%%%%%%%%%%%%%%%%%%%%%%%%%%%%%%%%%%%%%%%%%%%%%%%%%%%%%%%%%%%%%%%%%%%%%%
\par A subset $C\subseteq X$ is \textbf{\textit{convex}} if for any $a,b\in C$, the interval $[a,b]$ between $a$ and $b$ is contained in $C$. Each convex subset is naturally endowed with a median structure induced from the ambient space. In a complete median space, the nearest point projection to a closed convex subset $C$, that we denote by $\pi_C$, is a $1$-Lipschitz median morphism and verifies the following property: For any $x\in X$ and $c\in C$, we have $\pi_{C}(x)\in [c,x]$ (see Lemma 2.13 \cite{ChaDH_median}). In the particular case where the convex subset is an interval $[a,b]$, the nearest point projection coincide with the morphism $m(a,b,*)$ and the interval $[a,b]$ correspond to the fixed point of the latter application. Hence, the notion of convexity can be described by the ternary operation $m$, in fact, most of the interesting geometric feature of median spaces are encoded algebraically in the ternary operation $m$. More precisely, the given ternary operation $m$ endow $X$ with a structure of a median algebra, which is defined as follow:
\begin{defi}
A \textbf{\textit{median algebra}} $(M,m)$ is set $M$ endowed with a ternary operation $m:M\times M\times M\rightarrow M$ which verifies the following equations:
\begin{eqnarray*}
m(x,x,y)&=&x\\
m(x,y,z)&=&m(y,x,z)\ =\ m(x,z,y)\\
m(m(x,y,z),u,v)&=&m(x,m(y,u,v),m(z,u,v))
\end{eqnarray*}
\end{defi}
Let $M$ be a median algebra, and let $a,b\in M$, then the interval between $a,b$ is defined as the fixed points of the projection map $m(a,b,*)$, i.e. $[a,b]:=\{x\in M \ |\  m(a,b,x)=x\}$. The convexity is defined with respect to this notion of interval.
\par Conversely, if $(M,m)$ is a median algebra and there exists a metric $d$ on $M$ such that the intervals corresponding to the metric $d$ coincide with the intervals corresponding to the ternary operation $m$, then the metric space $(M,d)$ is a median space.

\par Let $X$ be a median space and let $\tilde{X}$ be its metric completion. The median operation being $1$-lipschitz, where $X^3$ is endowed with the $\ell^1$-product metric (Corollary 2.15 in \cite{ChaDH_median}), it extends to $\tilde{X}$. The equation arising in the axiom of the median algebra are defined by continuous functions with respect to $\tilde{X}$. Hence, the set of solution are closed in $\tilde{X}^k$ and contains the dense subset $X^k$, where $k$ is the number of variables describing an equation. Therefore the equations are verified by all $k$-tuple of points of $\tilde{X}$. We conclude that the metric completion is also a median space (see Proposition 2.21 in \cite{ChaDH_median}).\par
A well known fact, due to Eduard Helly (see \cite{Helly}), is that the intersection of a finite family of convex subsets in the euclidean space $\RR^n$ is empty if and only if the intersection of some subfamily of cardinal less than or equal $n+1$ is empty. In the particular case of the real line, the intersection of a finite family of convex subsets is empty if and only if there exist two convex subsets of the family which are disjoint. The same holds for median algebras (see Theorem 2.2 \cite{Roller}):
%%%%%%%%%%%%%%%%%%%%%%%%%%%%%%%%%%%%%%%%%%%%%%%%%%%%%%%%%%%%%%%%%%%%%%%%%%%%%%%%%%%%%%%%%%%%%%%%%%%%%%%%%%%%%%%%%%%%%%%%%%%%%%%%%%%%%%%%%%%%%%%%%%%%%%%%%%%%%%%%%%%%%%%%%%%%%%%%%%%%%%
\begin{thm}[Helly's Theorem]\label{helly_theorem}
Let $X$ be a median algebra and let $C_1,...,C_n\in X$ be a family of pairwise intersecting convex subsets. Then their intersection is not empty.
\end{thm}
%%%%%%%%%%%%%%%%%%%%%%%%%%%%%%%%%%%%%%%%%%%%%%%%%%%%%%%%%%%%%%%%%%%%%%%%%%%%%%%%%%%%%%%%%%%%%%%%%%%%%%%%%%%%%%%%%%%%%%%%%%%%%%%%%%%%%%%%%%%%%%%%%%%%%%%%%%%%%%%%%%%%%%%%%%%%%%%%%%%%%%%
\subsection{Halfspaces and poc sets}\label{halfspaces}
%%%%%%%%%%%%%%%%%%%%%%%%%%%%%%%%%%%%%%%%%%%%%%%%%%%%%%%%%%%%%%%%%%%%%%%%%%%%%%%%%%%%%%%%%%%%%%%%%%%%%%%%%%%%%%%%%%%%%%%%%%%%%%%%%%%%%%%%%%%%%%%%%%%%%%%%%%%%%%%%%%%%%%%%%%%%%%%%%%%%%%%
A central object in the geometry of median space are convex subset as they encode totally the median structure. 
\par
%%%%%%%%%%%%%%%%%%%%%%%%%%%%%%%%%%%%%%%%%%%%%%%%%%%%%%%%%%%%%%%%%%%%%%%%%%%%%%%%%%%%%%%%%%%%%%%%%%%%%%%%%%%%%%%%%%%%%%%%%%%%%%%%%%%%%%%%%%%%%%%%%%%%%%%%%%%%%%%%%%%%%%%%%%%%%%%%%%
\begin{defi}[Halfspaces]
 A convex subset of a median algebra is a \textbf{\textit{halfspace}} if its complement is also convex. We denote by $\HH (X)$ the set of all halfspaces of the median algebra $X$.
 \end{defi}
 %%%%%%%%%%%%%%%%%%%%%%%%%%%%%%%%%%%%%%%%%%%%%%%%%%%%%%%%%%%%%%%%%%%%%%%%%%%%%%%%%%%%%%%%%%%%%%%%%%%%%%%%%%%%%%%%%%%%%%%%%%%%%%%%%%%%%%%%%%%%%%%%%%%%%%%%%%%%%%%%%%%%%%%%%%%%%%%%%%
  Assuming Zorn's lemma, such halfspaces do exist. Moreover, any two disjoint convex subsets in a median algebra are separated by a halfspace as stated in the following theorem (see Theorem 2.8 \cite{Roller}):
%%%%%%%%%%%%%%%%%%%%%%%%%%%%%%%%%%%%%%%%%%%%%%%%%%%%%%%%%%%%%%%%%%%%%%%%%%%%%%%%%%%%%%%%%%%%%%%%%%%%%%%%%%%%%%%%%%%%%%%%%%%%%%%%%%%%%%%%%%%%%%%%%%%%%%%%%%%%%%%%%%%%%%%%%%%%%%%%%%
\begin{thm}[Separation theorem]\label{separation_theorem}
Let $X$ be a median algebra, then for any two convex subsets $C_1,C_2\subseteq X$ such that $C_1\cap C_2= \emptyset$, there exists a halfspace $\hh\in \HH(X)$ such that $C_1\subseteq \hh$ and $C_2\subseteq \hh^c$.
\end{thm}
%%%%%%%%%%%%%%%%%%%%%%%%%%%%%%%%%%%%%%%%%%%%%%%%%%%%%%%%%%%%%%%%%%%%%%%%%%%%%%%%%%%%%%%%%%%%%%%%%%
%%%%%%%%%%%%%%%%%%%%%%%%%%%%%%%%%%%%%%%%%%%%%%%%%%%%%%%%%%%%%%%%%%%%%%%%%%%%%%%%%%%%%%%%%%%%%%%%%%%%%%%%%%%%%%%%%%%%%%%%%%%%%%%%%%%%%%%%%%%%%%%%%%%%%%%%%%%%%%%%%%%%%%%%%%%%%%%%%%
The above theorem serves as the cornerstone for the geometry of median spaces, lying at the heart, alongside Helly's theorem, of the duality between a median space and its set of halfspaces.\par 
%\begin{thm}\label{Separation_theorem}
%Let $X$ be a median space. Then for any two disjoint convex subsets $C_1,C_2 \subseteq X$, there exist a halfspace $H\in \HH(X)$ such that $C_1\subseteq H$ and $C_2\subseteq H^c$.
%\end{thm}
Many of the median spaces that arise in examples are discrete or totally disconnected. Hence the topological dimension is not well suited for median spaces. Instead, there is a combinatorial notion, called the \textbf{\textit{rank}}, to characterize the dimension of the median space (or algebra) through the number of its pairwise \textbf{\textit{transverse}} halfspaces.

% The topological dimension is not well suited for median spaces as they need not to be connected. For instance those arising from the $0$ skeleton of CAT$(0)$ cubes complexes endowed with the combinatorial metric are discrete. To keep track of $n$-dimensional cubes inside the median space (or algebra), one consider the notion of \textbf{\textit{rank}} which is more appropriate in the median context:
%%%%%%%%%%%%%%%%%%%%%%%%%%%%%%%%%%%%%%%%%%%%%%%%%%%%%%%%%%%%%%%%%%%%%%%%%%%%%%%%%%%%%%%%%%%%%%%%%%%%%%%%%%%%%%%%%%%%%%%%%%%%%%%%%%%%%%%%%%%%%%%%%%%%%%%%%%%%%%%%%%%%%%%%%%%%%%%%%%
\begin{defi}[and notations]\label{definition_rank}
\begin{itemize}

\item Two halfspaces $\hh_1,\hh_2\in\HH(X)$ are called \textbf{\textit{transverse}} if the following intersections $\hh_1\cap \hh_2, \hh_1^c\cap \hh_2, \hh_1\cap \hh_2^c, \hh_1^c\cap \hh_2^c$ are not empty.
\item We say that a median space is of \textbf{\textit{rank}} $n$ if there exists a family of pairwise transverse halfspaces $\hh_1,...,\hh_n\in\HH(X)$ where $n$ is maximal.
% In the case of connected median space, the rank coincide with the maximum of topological dimension over compact subsets\footnote{Add a reference maybe}. % or more generally when the median algebra is endowed with a topology which make the ternary operation continuous,
\item A \textbf{\textit{wall}} in $X$ is a couple $(\hh,\hh^c)$ where $\hh\in\HH(X)$ is a halfspace of $X$. For any $A,B\subseteq X$, we denote by $\HH(A,B)$ the set of halfspaces which separate $B$ from $A$ $i.e. \HH(A,B):=\{\hh\in\HH(X) \ |\  B\subseteq \hh, A\subseteq \hh^c \}$. In the same vein, we define the walls interval between $A$ and $B$ as $\mathcal{W}:=\{(\hh,\hh^c) \ |\  \hh\in \HH(A,B)\}$. When $A$ and $B$ are singletons, we simply write $\HH(x,y)$ and $\mathcal{W}(x,y)$.
\item  A halfspace $\hh\subseteq X$ is said to be \textbf{\textit{transverse}} to a subset $C\subseteq X$ if both $\hh\cap C$ and $\hh^c\cap C$ are non empty. In that case, we say also that the wall $(\hh,\hh^c)$ is transverse to $C$. We denote by $\HH(C)$ (respectively $\mathcal{W}(C)$) the set of halfspaces (respectively walls) which are transverse to $C$.
\end{itemize}
\end{defi}
%%%%%%%%%%%%%%%%%%%%%%%%%%%%%%%%%%%%%%%%%%%%%%%%%%%%%%%%%%%%%%%%%%%%%%%%%%%%%%%%%%%%%%%%%%%%%%%%%%%%%%%%%%%%%%%%%%%%%%%%%%%%%%%%%%%%%%%%%%%%%%%%%%%%%%%%%%%%%%%%%%%%%%%%%%%%%%%%%%
 A convex subset in a complete median space is naturally endowed with a median structure, induced from the ambient space. The following proposition states that the halfspaces of a closed convex subset are induced by the halfspaces of the ambient space:
%%%%%%%%%%%%%%%%%%%%%%%%%%%%%%%%%%%%%%%%%%%%%%%%%%%%%%%%%%%%%%%%%%%%%%%%%%%%%%%%%%%%%%%%%%%%%%%%%%%%%%%%%%%%%%%%%%%%%%%%%%%%%%%%%%%%%%%%%%%%%%%%%%%%%%%%%%%%%%%%%%%%%%%%%%%%%%%%%%
\begin{prop}[See Proposition 2.3 \cite{Fior_median_property}]\label{halfspace_gateconvex}
Let $X$ be a complete median space and let $C\in X$ be a closed convex subset. We have then the following:
\[
\HH (C)=\{\hh\cap C \ |\  \hh\in\HH (X)\}
\]
\end{prop}
%%%%%%%%%%%%%%%%%%%%%%%%%%%%%%%%%%%%%%%%%%%%%%%%%%%%%%%%%%%%%%%%%%%%%%%%%%%%%%%%%%%%%%%%%%%%%%%%%%%%%%%%%%%%%%%%%%%%%%%%%%%%%%%%%%%%%%%%%%%%%%%%%%%%%%%%%%%%%%%%%%%%%%%%%%%%%%%%%%
%\begin{proof}
%It is clear that the intersection of a halfspace with a convex set $C$ gives rise to a halfspace inside the latter. Now considering any halfspace $\hh$ in $C$, as the gate projection is a median morphism, the inverse image of $\hh$ by the gate projection onto $C$ is a convex subset of $X$. The complementary set of the inverse image being the inverse image of the complementary, we deduce that the inverse image by the projection map of any intrinsic halfspace is a halfspace of the ambient space.
%\end{proof}
%%%%%%%%%%%%%%%%%%%%%%%%%%%%%%%%%%%%%%%%%%%%%%%%%%%%%%%%%%%%%%%%%%%%%%%%%%%%%%%%%%%%%%%%%%%%%%%%%%%%%%%%%%%%%%%%%%%%%%%%%%%%%%%%%%%%%%%%%%%%%%%%%%%%%%%%%%%%%%%%%%%%%%%%%%%%%%%%%%   

\par
In \cite{Bowd_conv} and \cite{Fior_median_property}, the authors extended the notion of hyperplanes which arise in CAT$(0)$ cube complexes to the case of complete connected topological median space and median space of finite rank respectively. In a complete connected median space of finite rank $X$, the \textbf{\textit{hyperplane}} $\hat{\hh}\subseteq X$ bounding the halfspace $\hh\subseteq X$ is the convex subset obtained from the intersection of the closure of $\hh$ with the closure of its complement, i.e. $\hat{\hh}=\bar{\hh}\cap\bar{\hh^c}$. It is a closed convex subset of rank less or equal $n-1$. These hyperplanes are practical to give proofs using an argument by induction on the rank of the space $X$.\par
For a complete connected median space of finite rank, we set $\HH_{x}:=\{\hh\in\HH(X) \ |\  x\in\overline{\hh}\cap\overline{\hh^c}\}$. It consists of the set of halfspaces which are ``branched'' at the point $x$. It is the natural generalization of the valency of $\RR$-trees to the higher rank case.

\begin{rmk}\label{closed_branched_halfspaces}
In a complete finite rank median space, halfspaces are either open, closed or possibly both. The latter case occurs when the median space is not connected (see \cite{Fior_median_property} Corollary 2.23). Hence, any halfspace $\hh\in \HH_x$ which contains $x$ is necessarily closed.
\end{rmk}
The set of halfspaces $\HH(X)$ of a median algebra $X$ is naturally endowed with a partial order relation given by the inclusion and a complementary operation which associates to each halfspaces its complement in $X$, thus inversing the partial order relation. This provides $\HH(X)$ with a natural structure of a poc set. A \textbf{\textit{poc set}} $(P,\ \leq,\ ^*,0)$ is a partially ordered set $(P,\ \leq)$ with a minimal element $0$ and an involution $^*$ on $P$, called the complementary operation, which inverse the order and such that the minimal element $0$ is the unique element in $P$ which is smaller than its complement $0^*$.
%%%%%%%%%%%%%%%%%%%%%%%%%%%%%%%%%%%%%%%%%%%%%%%%%%%%%%%%%%%%%%%%%%%%%%%%%%%%%%%%%%%%%%%%%%%%%%%%%%%%%%%%%%%%%%%%%%%%%%%%%%%%%%%%%%%%%%%%%%%%%%%%%%%%%%%%%%%%%%%%%%%%%%%%%%%%%%%%%%%%%%%
\subsection{Duality}\label{duality}
%%%%%%%%%%%%%%%%%%%%%%%%%%%%%%%%%%%%%%%%%%%%%%%%%%%%%%%%%%%%%%%%%%%%%%%%%%%%%%%%%%%%%%%%%%%%%%%%%%%%%%%%%%%%%%%%%%%%%%%%%%%%%%%%%%%%%%%%%%%%%%%%%%%%%%%%%%%%%%%%%%%%%%%%%%%%%%%%%%%%%%%
The set of halfspaces of a median algebra $X$, seen as an abstract poc set, carries all the informations about the median structure of $X$. In \cite{Isbell}, the author extended the Stone duality between boolean algebras and stone spaces to a duality between (Stone) median algebras and (Stone) poc sets, see Theorem 6.13 therein and \cite{Roller} for more detail.% A Boolean algebra is a subalgebra of the set of subsets $(\mathbb{P}(A),\cap,\cup,()^c,\emptyset,A)$ of some set $A$. A Stone space is a totally separated compact space. Roller investigated two ways to generalize stone duality to the setting of median algebra and poc set: the first way is to look at boolean algebra as a special case of median algebra. One obtain then a duality between median algebra and poc sets endowed with a natural Stone topology compatible with their structure (the complementary operation is continuous with respect to this topology). The second way is to look at boolean algebra as a special case of poc set, as any boolean structure gives rise to a natural partial order relation with a complementary action. One obtain then a duality between poc sets and median algebra endowed with a stone topology which is compatible with the median structure (the ternary operation is continuous with respect to this topology). When we restrict to the metric case, one obtain a duality between the category of median spaces to the Category of measured poc set. This was investigated in \cite{ChaDH_median} and \cite{Fior_median_property}. 
\par Almost all of the properties described in the two previous subsections depend only on the median structure given by the ternary operation and does not depend on the metric. The metric of a median space endows the set of halfspaces with an additional structure of a measured space. In this case, one obtain a duality between median spaces and measurable poc sets. This was investigated in \cite{ChaDH_median} and \cite{Fior_median_property}. Let us describe quickly the duality in question and cite a result from \cite{Fior_median_property} that we will be using later in the paper.
\par
Let us fix a poc set $(P,\leq, ^*,0)$. A set $\uu\in\PP(P)$ of subsets of $P$ is an \textbf{\textit{ultrafilter}} over $P$ if it verifies the following:
\begin{itemize}
\item For any $p,q\in \uu$ we do not have $p<q^*$. 
\item For any $p\in P$ we have either $p\in\uu$ or $p^*\in\uu$
\end{itemize} 
The first condition is an adaptation of the condition that the intersection between $p$ and $q$ being not empty when the poc set structure of $P$ comes from a set of subsets, and the second condition is that $\uu$ is maximal with respect to the first condition. In \cite{ChaDH_median}, the authors used instead the language of admissible sections.\par
A \textbf{\textit{pointed measured poc set}} is a quadruple $(P,\mathcal{D},\mu,\uu_0)$ where $P$ is a poc set, $\mathcal{D}$ a $\sigma$-algebra over $P$, $\mu$ a measure defined on $\mathcal{D}$ and $\uu_0$ a fixed ultrafiter over $P$. To each such pointed measured poc set, we associates a median space as follow : we consider the sets of ultrafilters $\uu$ over $P$ such that the symmetric difference $\uu\Delta\uu_0$ lies in $\mathcal{D}$. We identify the ultrafilters which have symmetric difference of zero measure. The set of such classes is denoted by $\MM(P)$ and the metric given by the measure of the symmetric difference endows it with a structure of a median space.\par
Conversely to each median space $X$, its set of halfspaces comes with a natural structure of a measured poc set $(\HH(X),\mathcal{D},\mu,\uu_{a})$ where $\mathcal{D}$ is the $\sigma$-algebra generated by the walls intervals, $\mu$ a measure defined on $\mathcal{D}$ such that $\mu(\mathcal{W}(x,y))=d(x,y)$ and $\uu_{a}$ is the principal ultrafilter associated to a fixed base point $a\in X$ $i.e.$ the set of halfspaces containing $a$. 
 We have the following theorem:
\begin{thm}[Theorem A \cite{Fior_median_property}]\label{duality_thm}
Let $X$ be a complete locally convex median space. Then $\MM(\HH(X))$ is isometric to $X$. 
%\footnote{We need to precise that the basepoint ultrafilter is given by any principal ultrafilter}
\end{thm} 
The inverse image of a halfspace under a median morphism is also a halfspace, hence any isometry $f:X_1\rightarrow X_2$ between median spaces gives rise to a measure preserving poc set morphism $f^{-1}:(\HH(X_2),\mathcal{D}_2,\mu_2,\uu_{f(x)})\rightarrow (\HH(X_1),\mathcal{D}_1,\mu_1,\uu_{x})$. Conversely, the inverse image of an ultrafilter under a poc set morphism is also an ultrafilter, hence any measure preserving poc set automorphism $g:(P_1,\mathcal{D}_1,\mu_1,\uu_1)\rightarrow (P_2,\mathcal{D}_2,\mu_2,\uu_2)$, such that $\mu_2(g^{-1}(\uu_2)\Delta \uu_1)<+\infty$, gives rise to an isometry $g^{-1}:\MM(P_2)\rightarrow \MM(P_1)$. Therefore, the associations given by $\MM(*)$ and $\HH(*)$ in the above duality process are functorial. Moreover, the functors $\HH$ and $\MM$ verify the following contravariance property:
\begin{prop}[Proposition 2.10 \cite{Fior_tits}, Theorem A \cite{Fior_median_property}]\label{contravariance}
Let $(X_1,d_1),(X_2,d_2)$ be two median spaces and let $(P_1,\mu_1,p_1)$,$(P_1,\mu_1,p_1)$ be two pointed measured poc set. Then we have:
\begin{itemize}
\item The median space $\MM (P_1 \sqcup P_2)$ is isometric to $(\MM(P_1)\times \MM(P_2),d_{\ell^1})$.
\item The measured poc set $\HH (X_1 \times X_2)$ is isomorphic to $\HH(X_1)\sqcup \HH(X_2)$.
\end{itemize}
\end{prop}

%%%%%%%%%%%%%%%%%%%%%%%%%%%%%%%%%%%%%%%%%%%%%%%%%%%%%%%%%%%%%%%%%%%%%%%%%%%%%%%%%%%%%%%%%%%%%%%%%%%%%%%%%%%%%%%%%%%%%%%%%%%%%%%%%%%%%%%%%%%%%%%%%%%%%%%%%%%%%%%%%%%%%%%%%%%%%%%%%%%%%%%%%%%%%%%%%%%%%%%%%%%%%%%%%%%%%%%%%%%%%%%%%%%%%%%%%%%%%%%%%%%%%%%%%%%%%%%%%%%%%%%%%%%%%%%%%%%%%%%%%%%%%%%%%%%%%%%%%%%%%%%%%%%%%%%%%%%%%%%%%%%%%%%%%%%%%%%%%%%%%%%%%%%%%%%%%%%%
\section{An embedding of the convex hull}%%%%%%%%%%%%%%%%%%%%%%%%%%%%%%%%%%%%%%%%%%%%%%%%%%%%%%%%%%%%%%%%%%%%%%%%%%%%%%%%%%%%%%%%%%%%%%%%%%%%%%%%%%%%%%%%%%%%%%%%%%%%%%%%%%%%%%%%%%%%%%%%%%%%%%%%%%%%%%%%%%%%%%%%%%%%%%%%%%%%%%%%%%%%%%%%%%%%%%%%%%%%%%%%%%%%%%%%%%%%%%%%%%%%%%%%%%%%%%%%%%%%%%%%%%%%%%%%%%%%%%%%%%%%%%%%%%%%%%%%%%%%%%%%%%%%%%%%%%%%%%%%%%%%%%%%%%%%%%%%%%%%%%%%%%%%%%%%%%%%%%%%%%%%%%%%%%%%
The aim of this section is to prove an embedding lemma of the convex hull between two convex subsets which do not share a transverse halfspace.
%%%%%%%%%%%%%%%%%%%%%%%%%%%%%%%%%%%%%%%%%%%%%%%%%%%%%%%%%%%%%%%%%%%%%%%%%%%%%%%%%%%%%%%%%%%%%%%%%%%%%%%%%%%%%%%%%%%%%%%%%%%%%%%%%%%%%%%%%%%%%%%%%%%%%%%%%%%%%%%%%%%%%%%%%%%%%%%%%%%%%%%%%%%%%%%%%%%%%%%%%%%%%%%%%%%%%%%%%%%%%%%%%%%%%%%%%%%
\subsection{Description of the convex hull}%%%%%%%%%%%%%%%%%%%%%%%%%%%%%%%%%%%%%%%%%%%%%%%%%%%%%%%%%%%%%%%%%%%%%%%%%%%%%%%%%%%%%%%%%%%%%%%%%%%%%%%%%%%%%%%%%%%%%%%%%%%%%%%%%%%%%%%%%%%%%%%%%%%%%%%%%%%%%%%%%%%%%%%%%%%%%%%%%%%%%%%%%%%%%%%%%%%%%%%%%%%%%%%%%%%%%%%%%%%%%%%%%%%
%Throughout this section, $X$ will be a complete connected median space. For a subset $A\subset X$, we denote by $Conv(A)$ the convex hull of $A$ in $X$, that is the smallest convex in $X$ containing the subset $A$.
%%%%%%%%%%%%%%%%%%%%%%%%%%%%%%%%%%%%%%%%%%%%%%%%%%%%%%%%%%%%%%%%%%%%%%%%%%%%%%%%%%%%%%%%%%%%%%%%%%%%%%%%%%%%%%%%%%%%%%%%%%%%%%%%%%%%%%%%%%%%%%%%%%%%%%%%%%%%%%%%%%%%%%%%%%%%%%%%%%
Let $X$ be a median algebra, the \textbf{\textit{convex hull}} of a subset $A\subseteq X$, that we denote by $Conv(A)$ is the intersection of all convex subsets containing $A$. The main goal of this subsection is to prove Proposition \ref{convex_hull_ball} which states that the convex hull of a ball of radius $r$ in a median space of rank $n$ is contained in a ball of radius $nr$. Before proving the claim, we first recall properties of the convex hull in the median context.
\par  The convex hull of the union of subsets which form a directed sets is the union of the convex hull of each sets.  Therefore, we get the following proposition (Corollary 2.5 \cite{Roller}):
%%%%%%%%%%%%%%%%%%%%%%%%%%%%%%%%%%%%%%%%%%%%%%%%%%%%%%%%%%%%%%%%%%%%%%%%%%%%%%%%%%%%%%%%%%%%%%%%%%%%%%%%%%%%%%%%%%%%%%%%%%%%%%%%%%%%%%%%%%%%%%%%%%%%%%%%%%%%%%
\begin{prop}
Let $X$ be a median algebra, then for any $A\subseteq X$ we have:
\[
Conv(A)=\displaystyle{\bigcup_{x_1,...,x_n\in A}Conv(\{x_1,...,x_n\})}
\]
\end{prop}
%%%%%%%%%%%%%%%%%%%%%%%%%%%%%%%%%%%%%%%%%%%%%%%%%%%%%%%%%%%%%%%%%%%%%%%%%%%%%%%%%%%%%%%%%%%%%%%%%%%%%%%%%%%%%%%%%%%%%%%%%%%%%%%%%%%%%%%%%%%%%%%%%%%%%%%%%%%%%%
The \textbf{\textit{join}} $[A,B]$ between two subsets $A,B$ is the 
union of all intervals having endpoints in $A$ and $B$. We have then the following:
%%%%%%%%%%%%%%%%%%%%%%%%%%%%%%%%%%%%%%%%%%%%%%%%%%%%%%%%%%%%%%%%%%%%%%%%%%%%%%%%%%%%%%%%%%%%%%%%%%%%%%%%%%%%%%%%%%%%%%%%%%%%%%%%%%%%%%%%%%%%%%%%%%%%%%%%%%%%%%%%%%%%%%%%%%%%%%%%%%
\begin{prop}\label{properties_of_convex_hull}
Let $C_1,C_2\subseteq X$ be two closed convex subsets. Then any point $x\in Conv(C_1,C_2)$ lies in the interval $[\pi_{C_1}(x),\pi_{C_2}(x)]$. In particular, we have: 
\[
Conv(C_1,C_2)=[C_1,C_2]=\displaystyle{\bigcup_{\substack{x\in C_1\\ y\in C_2}}[x,y]}
\]
\end{prop}
%%%%%%%%%%%%%%%%%%%%%%%%%%%%%%%%%%%%%%%%%%%%%%%%%%%%%%%%%%%%%%%%%%%%%%%%%%%%%%%%%%%%%%%%%%%%%%%%%%%%%%%%%%%%%%%%%%%%%%%%%%%%%%%%%%%%%%%%%%%%%%%%%%%%%%%%%%%%%%%%%%%%%%%%%%%%%%%%%%
\begin{proof}
Any halfspace separating $x$ from $\pi_{C_i}(x)$ separates $x$ from $C_i$. Hence, any halfspace which separates $x$ from $[\pi_{C_1}(x),\pi_{C_2}(x)]$ separates $x$ from $C_1$ and $C_2$, thus from $Conv(C_1,C_2)$. Therefore, if a point $x\in X$ lies outside $[\pi_{C_1}(x),\pi_{C_2}(x)]$ then it does not belongs to $Conv(C_1,C_2)$.
\end{proof}
%%%%%%%%%%%%%%%%%%%%%%%%%%%%%%%%%%%%%%%%%%%%%%%%%%%%%%%%%%%%%%%%%%%%%%%%%%%%%%%%%%%%%%%%%%%%%%%%%%%%%%%%%%%%%
We have the following constraint on the projection onto closed convex subset :
%%%%%%%%%%%%%%%%%%%%%%%%%%%%%%%%%%%%%%%%%%%%%%%%%%%%%%%%%%%%%%%%%%%%%%%%%%%%%%%%%%%%%%%%%%%%%%%%%%%%%%%%%%%%%%
\begin{lem}\label{constraint_on_the_projection}
Let $C$ be a closed convex subset and $A$ a convex set such that $A\cap C\neq \emptyset$. Then the projection of $A$ into $C$ lies in $A\cap C\neq \emptyset$.
\end{lem}
%%%%%%%%%%%%%%%%%%%%%%%%%%%%%%%%%%%%%%%%%%%%%%%%%%%%%%%%%%%%%%%%%%%%%%%%%%%%%%%%%%%%%%%%%%%%%%%%%%%%%%%%%%%%%%%%%%%%%%%%%%%%%%%%%%%%%%%%%%%%%%%%%%%%%%%%%%%%%%%%%%%%%%%%%%%%%%%%%%
\begin{proof}
Let us fix $c\in A\cap C$ and consider any $a\in A$. We have $\pi_C(a)\in[a,c]\subseteq A$. Hence $\pi_{C}(A)\subseteq A\cap C$. 
\end{proof}
%%%%%%%%%%%%%%%%%%%%%%%%%%%%%%%%%%%%%%%%%%%%%%%%%%%%%%%%%%%%%%%%%%%%%%%%%%%%%%%%%%%%%%%%%%%%%%%%%%%%%%%%%%%%%%%%%%%%%%%%%%%%%%%%%%%%%%%%%%%%%%%%%%%%%%%%%%%%%%%%%%%%%%%%%%%%%%%%%%
\begin{rmk}\label{strong_separation_and_singleton_bridge}
Let $X$ be a complete median space and let $C_1,C_2\subseteq X$ be two closed convex subsets.
\begin{enumerate}
\item By Lemma \ref{constraint_on_the_projection}, we have $\HH(x,\pi_{C_1}(x))=\HH(x,C_1)$ for any $x\in X$ . 
\item The projections $\pi_{C_1}(C_2)$ and $\pi_{C_2}(C_1)$ are singletons if and only if there is no halfspace which is transverse to both $C_1$ and $C_2$ . The first implication is due to Lemma \ref{constraint_on_the_projection}. The other direction is given by the fact that the image under a projection into a closed convex subset is also convex and that the inverse image of any halfspace by such projections, which are median morphisms, is a halfspace of the ambient space (See Proposition \ref{halfspace_gateconvex}).
\end{enumerate} 
\end{rmk}
%%%%%%%%%%%%%%%%%%%%%%%%%%%%%%%%%%%%%%%%%%%%%%%%%%%%%%%%%%%%%%%%%%%%%%%%%%%%%%%%%%%%%%%%%%%%%%%%%%%%%%%%%%%%%%%%%%%%%%%%%%%%%%%%%%%%%%%%%%%%%%%%%%%%%%%%%%%%%%%%%%%%%%%%%%%%%%%%%%
The \textbf{\textit{$n$-iterated}} join $J^n(A)$ of a subset $A$ is defined recursively by $J^n(A)=[J^{n-1}(A),J^{n-1}(A)]$ and $J^0(A)=A$. \par In a median space, the balls are convex if and only if the median space is of rank $1$, that is, it is tree like. In the general setting, median spaces are not necessarily locally convex. However, in the finite rank case, the median space is locally convex and the convex hull of any subset is obtained after iterating a finite number of time its join:
%%%%%%%%%%%%%%%%%%%%%%%%%%%%%%%%%%%%%%%%%%%%%%%%%%%%%%%%%%%%%%%%%%%%%%%%%%%%%%%%%%%%%%%%%%%%%%%%%%%%%%%%%%%%%%%%%%%%%%%%%%%%%%%%%%%%%%%%%%%%%%%%%%%%%%%%%%%%%%%%%%%%%%%%%%%%%%%%%%%%%%%%%%%%%%%%%%%%%%%%%%%%%%%%%
\begin{prop}[Lemma 6.4 \cite{Bowd_conv}]
Let $X$ be a median space of rank $n$. Then for any $A\subseteq X$ we have $Conv(A)=J^n(A)$.
\end{prop}
%%%%%%%%%%%%%%%%%%%%%%%%%%%%%%%%%%%%%%%%%%%%%%%%%%%%%%%%%%%%%%%%%%%%%%%%%%%%%%%%%%%%%%%%%%%%%%%%%%%%%%%%%%%%%%%%%%%%%%%%%%%%%%%%%%%%%%%%%%%%%%%%%%%%%%%%%%%%%%%%%%%%%%%%%%%%%%%%%%%%%%%%%%%%%%%%%%%%%%%%%%%%%%%%%
 
One may remark from the latter Proposition that the convex hull of the ball of radius $r$ lies in a ball of radius $2^{n-1}r$. Moreover, we prove the following

%%%%%%%%%%%%%%%%%%%%%%%%%%%%%%%%%%%%%%%%%%%%%%%%%%%%%%%%%%%%%%%%%%%%%%%%%%%%%%%%%%%%%%%%%%%%%%%%%%%%%%%%%%%%%%%%%%%%%%%%%%%%%%%%%%%%%%%%%%%%%%%%%%%%%%%%%%%%%%%%%%%%%%%%%%%%%%%%%%%%%%%%%%%%%%%%%%%%%%%%%%%%%%%%%
\begin{prop}\label{convex_hull_ball}
Let $X$ be a median space of rank $n$. Then for any $a\in X$ and $r>0$, we have:
\[
Conv(B(a,r))\subseteq B(a,nr)
\]
\end{prop}
%%%%%%%%%%%%%%%%%%%%%%%%%%%%%%%%%%%%%%%%%%%%%%%%%%%%%%%%%%%%%%%%%%%%%%%%%%%%%%%%%%%%%%%%%%%%%%%%%%%%%%%%%%%%%%%%%%%%%%%%%%%%%%%%%%%%%%%%%%%%%%%%%%%%%%%%%%%%%%%%%%%%%%%%%%%%%%%%%%%%%%%%%%%%%%%%%%%%%%%%%%%%%%%%%
Note that this bound is reached on the pieces of the median space which are isometric to pieces of $(\RR^n,\ell^1)$. Before proving the statement, we will be needing some lemmas. The following lemma is a strengthening of the separation Theorem \ref{separation_theorem} in the case of complete median space of finite rank:
%%%%%%%%%%%%%%%%%%%%%%%%%%%%%%%%%%%%%%%%%%%%%%%%%%%%%%%%%%%%%%%%%%%%%%%%%%%%%%%%%%%%%%%%%%%%%%%%%%%%%%%%%%%%%%%%%%%%%%%%%%%%%%%%%%%%%%%%%%%%%%%%%%%%%%%%%%%%%%%%%%%%%%%%%%%%%%%%%%
\begin{lem}\label{strenghtening_separation_theorem}
Let $X$ be a complete connected median space of finite rank. Then for any $a,b\in X$, there exists a halfspace $\hh\in\HH (a,b)$ such that $d(a,\hh)=0$. \\In particular, we have $\displaystyle{\bigcap_{\substack{\hh\in \HH_a(X) \\ a \in \hh}}}=\{a\}$
\end{lem}
%%%%%%%%%%%%%%%%%%%%%%%%%%%%%%%%%%%%%%%%%%%%%%%%%%%%%%%%%%%%%%%%%%%%%%%%%%%%%%%%%%%%%%%%%%%%%%%%%%%%%%%%%%%%%%%%%%%%%%%%%%%%%%%%%%%%%%%%%%%%%%%%%%%%%%%%%%%%%%%%%%%%%%%%%%%%%%%%%%
\begin{proof}
Let us consider two distinct points $a,b\in X$. The median space $X$ being complete and connected, there exists a midpoint $b_1\in [a,b]$, i.e. $d(a,b_1)=d(b_1,b)=\frac{d(a,b)}{2}$. Let $\hh_1\in \HH (a,b_1)$ be a halfspace separating $b_1$ from $a$. Let $b_2$ be a midpoint of $[a,\pi_{\bar{\hh}_1}(a)]$ and let $\hh_2\in\HH (a,b_2)$ be a halfspace separating $b_2$ from $a$. The halfspace $\hh_2$ separates $ \pi_{\bar{\hh}_1}(a)$ from $a$, hence it contains $\bar{\hh}_1$ (see Remark \ref{strong_separation_and_singleton_bridge}). Proceeding by iteration, we obtain an ascending sequence of halfspaces $(\hh_n)_{n\in\NN^*}$ separating $b$ from $a$ and such that $\displaystyle{\lim_{n\rightarrow\infty}d(\hh_n,a)=0}$. The subset $\hh:=\displaystyle{\bigcup_{n\in\NN^*}\hh_n}$ is our desired halfspace.
\end{proof}
%%%%%%%%%%%%%%%%%%%%%%%%%%%%%%%%%%%%%%%%%%%%%%%%%%%%%%%%%%%%%%%%%%%%%%%%%%%%%%%%%%%%%%%%%%%%%%%%%%%%%%%%%%%%%%%%%%%%%%%%%%%%%%%%%%%%%%%%%%%%%%%%%%%%%%%%%%%%%%%%%%%%%%%%%%%%%%%%%%%%%%%
\begin{rmk}
Lemma \ref{strenghtening_separation_theorem} above remains true in the case of a complete connected locally convex median space. One may adapt the argument given in \cite{Roller} by looking at a maximal element in the family of convex subsets which contains $b$ in their interior and separates it from $a$, and show that such a maximal element is a halfspace which contains $a$ on its closure. 
\end{rmk}
%%%%%%%%%%%%%%%%%%%%%%%%%%%%%%%%%%%%%%%%%%%%%%%%%%%%%%%%%%%%%%%%%%%%%%%%%%%%%%%%%%%%%%%%%%%%%%%%%%%%%%%%%%%%%%%%%%%%%%%%%%%%%%%%%%%%%%%%%%%%%%%%%%%%%%%%%%%%%%%%%%%%%%%%%%%%%%%%%%%%%%%%%%%%%%%%%%%%%%%%%%%%%%%%%%%%
\par We deduce the following lemma: 
%%%%%%%%%%%%%%%%%%%%%%%%%%%%%%%%%%%%%%%%%%%%%%%%%%%%%%%%%%%%%%%%%%%%%%%%%%%%%%%%%%%%%%%%%%%%%%%%%%%%%%%%%%%%%%%%%%%%%%%%%%%%%%%%%%%%%%%%%%%%%%%%%%%%%%%%%%%%%%%%%%%%%%%%%%%%%%%%%%%%%%%%%%%%%%%%%%%%%%%%%%%%%%%%%
\begin{lem}\label{lemma_local_convexity}
Let $X$ be a median space of rank $n$, and $a,b\in X$. Then there exists a halfspace $\hh\in\HH (a,b) $ such that $d(a,\hh)=0$ and $d(b,\hh^c)\geq \frac{d(a,b)}{n}$.
\end{lem}
%%%%%%%%%%%%%%%%%%%%%%%%%%%%%%%%%%%%%%%%%%%%%%%%%%%%%%%%%%%%%%%%%%%%%%%%%%%%%%%%%%%%%%%%%%%%%%%%%%%%%%%%%%%%%%%%%%%%%%%%%%%%%%%%%%%%%%%%%%%%%%%%%%%%%%%%%%%%%%%%%%%%%%%%%%%%%%%%%%%%%%%%%%%%%%%%%%%%%%%%%%%%%%%%%
\begin{proof}
We proceed by induction on the rank of the space $X$. The lemma is trivial for a connected $\RR$-tree. Let us assume that the lemma is true for median spaces of rank $n-1$. Let $X$ be a median space of rank $n$, and let $a,b\in X$. Let us take a halfspace $\hh\in\HH(a,b)$ such that $d(a,\hh)=0$, Lemma \ref{strenghtening_separation_theorem} ensures the existence of such halfspace. Let us assume that $\hh$ is not our desired halfspace, that is $d(b,\hh^c)< \frac{d(a,b)}{n}$. Let us consider then the projections $\tilde{b}:=\pi_{\hat{\hh}}(b)$, $\pi_{\hat{\hh}}(a)=a$. We have then:
\[
d(a,b)=d(a,\tilde{b})+d(\tilde{b},b)
\]
We deduce the following:
\begin{eqnarray*}
d(a,\tilde{b})&=&d(a,b)-d(b,\tilde{b})\\
\ \ \ \ \ \           &\geq & d(a,b)-\frac{d(a,b)}{n}\\
                      &\geq & \frac{(n-1)d(a,b)}{n}
\end{eqnarray*}
The interval $[a,\tilde{b}]$ lies in the hyperplane $\hat{\hh}$, which is a median space of rank $n-1$. Hence, there exists a halfspace $\tilde{\hh}\in\HH(a,\tilde{b})$ such that $d(a,\tilde{\hh})=0$ and $d(\tilde{b},\tilde{\hh}^c)\geq \frac{d(a,\tilde{b})}{n-1}$. The projection into a convex subset being $1$-lipschitz, we get $d(b,\tilde{\hh}^c)\geq d(\tilde{b},\tilde{\hh}^c)$. We conclude then:
\begin{eqnarray*}
d(b,\tilde{\hh}^c)&\geq & d(\tilde{b},\tilde{\hh}^c) \geq  \frac{d(\tilde{a},b)}{n-1}\\
                &\geq & \frac{\frac{(n-1)d(a,b)}{n}}{n-1}\geq  \frac{d(a,b)}{n} 
\end{eqnarray*}
Therefore $\tilde{\hh}$ is our desired halfspace.
\end{proof}
%%%%%%%%%%%%%%%%%%%%%%%%%%%%%%%%%%%%%%%%%%%%%%%%%%%%%%%%%%%%%%%%%%%%%%%%%%%%%%%%%%%%%%%%%%%%%%%%%%%%%%%%%%%%%%%%%%%%%%%%%%%%%%%%%%%%%%%%%%%%%%%%%%%%%%%%%%%%%%%%%%%%%%%%%%%%%%%%%%%%%%%%%%%%%%%%%%%%%%%%%%%%%%%%%%%%%%%%%%%%%%%%%%%%%%%%%%%

\begin{proof}[Proof of Proposition \ref{convex_hull_ball}]
Note that there is no loss of generality if we assume $X$ to be complete. Let $b\in X$ such that $d(a,b)>nr$. By Lemma \ref{lemma_local_convexity}, there exists a halfspace $\hh\in\HH (a,b)$ such that $d(a,\hh)=0$ and $d(b,\hh^c)\geq\frac{d(a,b)}{n}$.
Thus, the ball $B(b,r)$ is contained in the halfspace $\hh$. Hence, the convex hull of $B(b,r)$ also lies in $\hh$. We conclude that any point which is at distance bigger than $nr$ from the point $b$ lies outside the convex hull of $B(b,r)$. 
\end{proof}
%%%%%%%%%%%%%%%%%%%%%%%%%%%%%%%%%%%%%%%%%%%%%%%%%%%%%%%%%%%%%%%%%%%%%%%%%%%%%%%%%%%%%%%%%%%%%%%%%%%%%%%%%%%%%%%%%%%%%%%%%%%%%%%%%%%%%%%%%%%%%%%%%%%%%%%%%%%%%%%%%%%%%%%%%%%%%%%%%%%%%%%%%%%%%%%%%%%%%%%%%%%%%%%%%%%%%%%%%%%%%%%%%%%%%%%%%%%%%%%%%%%%%%%%%%%%%%%%%%%%%%%%%%%%%%%%%%%%%%%%%%%%%%%%%%%%%%%%%%%%%%%%%%%%%%%%%%%%%%%%%%%%%%%%%%%%%%%%%%%%%%%%%%%%%%%%%%%%
\begin{rmk}\label{remark_Hagen}
As it was pointed out to us by M. Hagen, Proposition \ref{convex_hull_ball} remains true when we replace the point $a$ by a closed convex subset $C$ and consider the tubular neighbourhood of it. More precisely, for any $r>0$ we have:
\[
  Conv(\mathcal{N}_r(C))\subseteq \mathcal{N}_{nr}(C).
\]
To see this, let us consider a point $x\in X$ which is at distance greater than $n.r$ and show that there exist a halfpsace $\hh\in\HH(C,x)$ such that $d(C,\hh)>r$. By Lemma \ref{lemma_local_convexity}, there exist a halfspace $\hh\in\HH(x_C,x)$, where $x_C:=\pi_C(x)$ and such that $d(x_C,\hh)>r$. By the bridge Lemma, we have $d(C,\hh)=d(x_C,\pi_{\bar{\hh}}(x_C))$. Hence the $r$-tubular neighbourhood $\mathcal{N}_r(C)$ of $C$ is contained in $\hh^c$. Therefore we conclude that  $Conv(\mathcal{N}_r(C))\subseteq \hh^c$ which implies that the point $x$ is not contained in $Conv(\mathcal{N}_r(C))$.
\end{rmk}
%%%%%%%%%%%%%%%%%%%%%%%%%%%%%%%%%%%%%%%%%%%%%%%%%%%%%%%%%%%%%%%%%%%%%%%%%%%%%%%%%%%%%%%%%
%%%%%%%%%%%%%%%%%%%%%%%%%%%%%%%%%%%%%%%%%%%%%%%%%%%%%%%%%%%%%%%%%%%%%%%%%%%%%%%%%%%%%%%%%%%%%%%%%%%%%%%%%%%%%%%%%%%%%%%%%%%%%%%%%%%%%%%%%%%%%%%%%%%%%%%%%%%%%%%%%%%%%%%%%%%%%%%%%%%%%%%%%%%%%%%%%%%%%%%%%%%%%%%%%
\subsection{Embedding lemma}%%%%%%%%%%%%%%%%%%%%%%%%%%%%%%%%%%%%%%%%%%%%%%%%%%%%%%%%%%%%%%%%%%%%%%%%%%%%%%%%%%%%%%%%%%%%%%%%%%%%%%%%%%%%%%%%%%%%%%%%%%%%%%%%%%%%%%%%%%%%%%%%%%%%%%%%%%%%%%%%%%%%%%%%%%%%%%%%%%%%%%%%%%%%%%%%%%%%%%%%%%%%%%%%%%%%%%%%%%%%%%%%%%%
The following embedding property of the convex hull of two closed convex subsets with no transverse halfspace in common will be a key ingredient in the proof of Theorem \ref{thm_principal} :
%%%%%%%%%%%%%%%%%%%%%%%%%%%%%%%%%%%%%%%%%%%%%%%%%%%%%%%%%%%%%%%%%%%%%%%%%%%%%%%%%%%%%%%%%%%%%%%%%%%%%%%%%%%%%%%%%%%%%%%%%%%%%%%%%%%%%%%%%%%%%%%%%%%%%%%%%%%%%%%%%%%%%%%%%%%%%%%%%%
\begin{prop}\label{embed_lemm}
Let $X$ be a complete median space of finite rank and let $C_1,C_2\in X$ be closed convex subsets such that there is no half-space which is transverse to them both. Then the following map 
\begin{eqnarray*}
\ \ \ \ \ \ f:Conv(C_1,C_2)&\rightarrow & (C_1\times C_2 \times [c_1,c_2],d_\ell^1)\\
\ \ \ \ \ \   x & \mapsto & f(x)=(\pi_{C_1}(x),\pi_{C_2}(x),\pi_{[c_1,c_2]}(x))       
\end{eqnarray*}   
is an isometric embedding, where $\pi_{C_i}$ denote the projection onto the closed convex subset $C_i$, $c_1=\pi_{C_1}(C_2)$ and $c_2=\pi_{C_2}(C_1)$ (see Remark \ref{strong_separation_and_singleton_bridge}).
\end{prop}
%%%%%%%%%%%%%%%%%%%%%%%%%%%%%%%%%%%%%%%%%%%%%%%%%%%%%%%%%%%%%%%%%%%%%%%%%%%%%%%%%%%%%%%%%%%%%%%%%%%%%%%%%%%%%%%%%%%%%%%%%%%%%%%%%%%%%%%%%%%%%%%%%%%%%%%%%%%%%%%%%%%%%%%%%%%%%%%%%%
For the proof of Proposition \ref{embed_lemm}, we need the following lemma:
%%%%%%%%%%%%%%%%%%%%%%%%%%%%%%%%%%%%%%%%%%%%%%%%%%%%%%%%%%%%%%%%%%%%%%%%%%%%%%%%%%%%%%%%%%%%%%%%%%%%%%%%%%%%%%%%%%%%%%%%%%%%%%%%%%%%%%%%%%%%%%%%%%%%%%%%%%%%%%%%%%%%%%%%%%%%%%%%%%
\begin{lem}\label{decomposition_of_separating_halfpsace}
Let $X$ be a complete median space of finite rank and let $C_1,C_2\subseteq X$ such that there is no halfspace which is transverse to both. After setting $c_1=\pi_{C_1}(C_2)$, $c_2=\pi_{C_2}(C_1)$ and considering any $x,y\in Conv(C_1,C_2)$ we have:
\begin{equation}
\mathcal{W} (x,y)= \mathcal{W} (\pi_{C_1}(x),\pi_{C_1}(y)) \sqcup \mathcal{W} (\pi_{C_2}(x),\pi_{C_2}(y)) \sqcup \mathcal{W} (\pi_{[c_1,c_2]}(x),\pi_{[c_1,c_2]}(y))
\end{equation}
\end{lem}
%%%%%%%%%%%%%%%%%%%%%%%%%%%%%%%%%%%%%%%%%%%%%%%%%%%%%%%%%%%%%%%%%%%%%%%%%%%%%%%%%%%%%%%%%%%%%%%%%%%%%%%%%%%%%%%%%%%%%%%%%%%%%%%%%%%%%%%%%%%%%%%%%%%%%%%%%%%%%%%%%%%%%%%%%%%%%%%%%%
\begin{proof}
For any closed convex subset $C\subseteq X$ and any $x,y\in X$, we have:
\[
\mathcal{W}(\pi_C(x),\pi_C(y))\subseteq \mathcal{W}(x,y)
\]
This come from the fact that for any $c\in C$, the interval $[c,x]$ contains $\pi_C(x)$. Thus, we have the inclusion of the right hand of the equality $(1)$ into the left side. For the other inclusion, let us consider $x,y\in Conv(C_1,C_2)$ and $\hh\in\HH(x,y)$. Let us assume that $\hh$ does not separates $\pi_{C_1}(y)$ from $\pi_{C_1}(x)$ and $\pi_{C_2}(y)$ from $\pi_{C_2}(x)$. By Proposition \ref{properties_of_convex_hull}, if the projections $\pi_{C_1}(x)$, $\pi_{C_1}(y)$, $\pi_{C_2}(x)$ and $\pi_{C_2}(y)$ lie in a halfspace $\hh$, then so do $x$ and $y$. As the halfspace $\hh$ separates $y$ from $x$ there is no loss of generality if we assume that $\pi_{C_1}(x),\pi_{C_1}(y)$ belong to $\hh^c$ and $\pi_{C_2}(x),\pi_{C_2}(y)$ to $\hh$. As $\pi_{C_1}(y)\in[c_1,y]$ and $\pi_{C_2}(x)\in[c_2,x]$, we necessarily get that $c_1\in\hh^c$ and $c_2\in\hh$. We conclude that $\pi_{[c_1,c_2]}(x)\in[c_1,x]\subseteq\hh^c$ and $\pi_{[c_1,c_2]}(y)\in[c_2,y]\subseteq\hh$ (see the Figure \ref{example1} below). Therefore, the halfspace $\hh$ lies in $\HH (\pi_{[c_1,c_2]}(x),\pi_{[c_1,c_2]}(y))$.\par
It is left to show that the sets arising in the right hand of the equality are indeed disjoint. Under our assumtpion that the convex subsets $C_1$ and $C_2$ being strongly separated, we already have the disjointness of $\mathcal{W} (\pi_{C_1}(x),\pi_{C_1}(y))$ with $\mathcal{W} (\pi_{C_2}(x),\pi_{C_2}(y))$. A wall which separates two points of the interval $[c_1,c_2]$ must separate $c_1$ and $c_2$. The point $c_1$ being contained in any interval connecting $C_1$ to $C_2$, we deduce that any wall in $\mathcal{W} (\pi_{[c_1,c_2]}(x),\pi_{[c_1,c_2]}(y))$ must separates $C_1$ and $C_2$. Hence, such wall cannot be in $\mathcal{W} (\pi_{C_1}(x),\pi_{C_1}(y))$ nor in $\mathcal{W} (\pi_{C_2}(x),\pi_{C_2}(y))$. 
\end{proof}
\begin{figure}[h]
\includegraphics[scale=0.2]{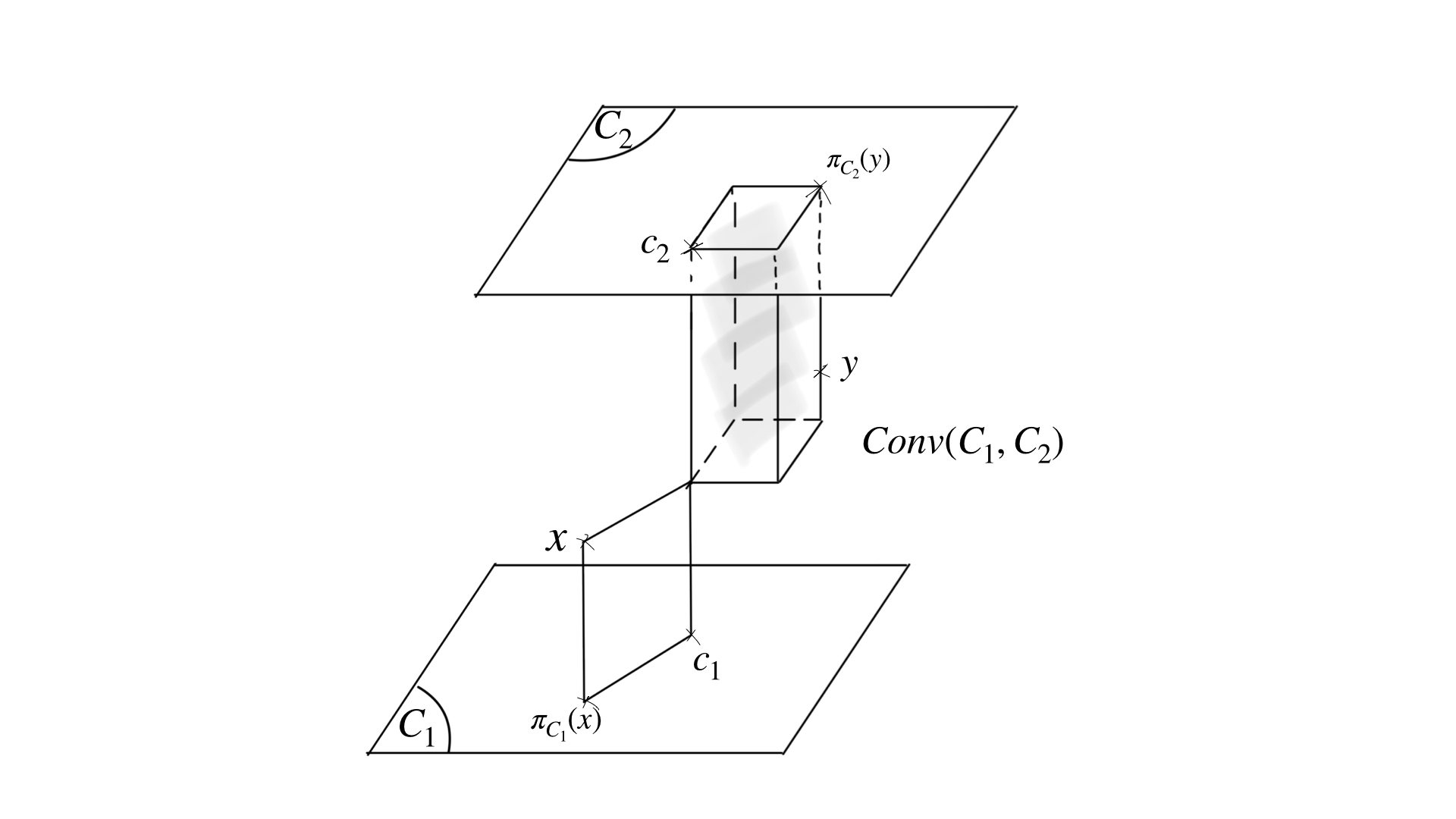}
\centering
\caption{Any halfspace which separates $x$ and $y$ is either transverse to $C_1$, to $C_2$ or to the interval $[c_1,c_2]$.}
\label{example1}
\end{figure}
%%%%%%%%%%%%%%%%%%%%%%%%%%%%%%%%%%%%%%%%%%%%%%%%%%%%%%%%%%%%%%%%%%%%%%%%%%%%%%%%%%%%%%%%%%%%%%%%%%%%%%%%%%%%%%%%%%%%%%%%%%%%%%%%%%%%%%%%%%%%%%%%%%%%%%%%%%%%%%%%%%%%%%%%%%%%%%%%%%
\begin{proof}[Proof of Proposition \ref{embed_lemm}]
As there is no halfspace which is transverse to both $C_1$ and $C_2$, the projection of $C_1$ (resp $C_2$) into $C_2$ (resp $C_1$) is a singleton, according to Remark \ref{strong_separation_and_singleton_bridge}. Let us set $c_1=\pi_{C_1}(C_2)$ and $c_2=\pi_{C_2}(C_1)$ and consider the following map:
\begin{eqnarray*}
\ \ \ \ \ \ f:Conv(C_1,C_2)&\rightarrow & C_1\times C_2 \times [c_1,c_2]\\
\ \ \ \ \ \   x & \mapsto & f(x)=(\pi_{C_1}(x),\pi_{C_2}(x),\pi_{[c_1,c_2]}(x))       
\end{eqnarray*}   
where $\pi_C$ denote the projection onto the closed convex subset $C$.\par  
By Lemma \ref{decomposition_of_separating_halfpsace}, we get the following:
\[
\mathcal{W} (x,y)= \mathcal{W} (\pi_{C_1}(x),\pi_{C_1}(y)) \sqcup \mathcal{W} (\pi_{C_2}(x),\pi_{C_2}(y)) \sqcup \mathcal{W} (\pi_{[c_1,c_2]}(x),\pi_{[c_1,c_2]}(y))
\]
We deduce then the following:
\begin{eqnarray*}
\mu (\mathcal{W} (x,y)) &=& \mu (\mathcal{W} (\pi_{C_1}(x),\pi_{C_1}(y)))+\mu (\mathcal{W} (\pi_{C_2}(x),\pi_{C_2}(y)))+\mu (\mathcal{W} (\pi_{[c_1,c_2]}(x),\pi_{[c_1,c_2]}(y)))\\
d(x,y)&=& d(\pi_{C_1}(x),\pi_{C_1}(y))+d(\pi_{C_2}(x),\pi_{C_2}(y))+d(\pi_{[c_1,c_2]}(x),\pi_{[c_1,c_2]}(y))\\
 &=& d_{\ell ^1}(f(x),f(y))
\end{eqnarray*}
where $\mu$ is the measure described in Subsection \ref{duality}.
\end{proof}
%%%%%%%%%%%%%%%%%%%%%%%%%%%%%%%%%%%%%%%%%%%%%%%%%%%%%%%%%%%%%%%%%%%%%%%%%%%%%%%%%%%%%%%%%%%%%%%%%%%%%%%%%%%%%%%%%%%%%%%%%%%%%%%%%%%%%%%%%%%%%%%%%%%%%%%%%%%%%%%%%%%%%%%%%%%%%%%%%%
%\begin{rmk}\label{description_interval_convex_hull}
%Let $X$ be a complete median space of finite rank. Then for any two strongly separated closed convex subset $C_1,C_2\subset X$ and any points $a\in C_1, b\in C_2$, we have the following decomposition of the wall interval between $a$ and $b$: 
%\[
%\mathcal{W} (a,b)=\mathcal{W} (a,c_1) \sqcup \mathcal{W} (c_1,c_2) \sqcup \mathcal{W} (c_2, b)
%\]
%where $c_1:=\pi_{C_1}(C_2)$ and $c_2:=\pi_{C_2}(C_1)$.
%\end{rmk}
%%%%%%%%%%%%%%%%%%%%%%%%%%%%%%%%%%%%%%%%%%%%%%%%%%%%%%%%%%%%%%%%%%%%%%%%%%%%%%%%%%%%%%%%%%%%%%%%%%%%%%%%%%%%%%%%%%%%%%%%%%%%%%%%%%%%%%%%%%%%%%%%%%%%%%%%%%%%%%%%%%%%%%%%%%%%%%%%%%
By considering a stronger assumption in Proposition \ref{embed_lemm}, we get a local version of the first part of Theorem \ref{thm_principal}:
\begin{prop}\label{ncube}
Let $X$ be a complete median space of finite rank and let $C_1,C_2\subseteq X$ be two closed convex subsets such that $C_1\cap C_2=\{x_0\}$. Then, for any $x\in Conv(C_1,C_2)$, the interval $[x,x_0]$ is isometric to the $\ell^1$-product of $[\pi_{C_1}(x),x_0]$ and $[\pi_{C_2}(x),x_0]$.
\end{prop}
%%%%%%%%%%%%%%%%%%%%%%%%%%%%%%%%%%%%%%%%%%%%%%%%%%%%%%%%%%%%%%%%%%%%%%%%%%%%%%%%%%%%%%%%%%%%%%%%%%%%%%%%%%%%%%%%%%%%%%%%%%%%%%%%%%%%%%%%%%%%%%%%%%%%%%%%%%%%%%%%%%%%%%%%%%%%%%%%%%
\begin{proof}
For any $x\in Conv(C_1,C_2)$, we have:
\[
\HH([x,x_0])\ =\ \mathcal{W}(x,x_0)\ =\footnote{This equality comes from Lemma \ref{decomposition_of_separating_halfpsace}}\ \mathcal{W}(\pi_{C_1}(x),x_0)\sqcup\mathcal{W}(\pi_{C_2}(x),x_0)
\] 
where any halfspace in $\HH(\pi_{C_1}(x),x_0)$ is transverse to any halfspace in $\HH(\pi_{C_2}(x),x_0)$. Hence, by Proposition \ref{contravariance} the median space $\MM(\HH([x,x_0]))$ is isometric to the $\ell^1$-product of $\MM(\HH([\pi_{C_1}(x),x_0]))$ with $\MM(\HH([\pi_{C_2}(x),x_0]))$. The intervals being closed subsets and the median space being complete, we conclude by Theorem \ref{duality_thm} that the interval $[x,x_0]$ is isomorphic to $\MM(\HH([x,x_0]))$ and that the latter is isomorphic to the $\ell^1$-product of the two intervals $[\pi_{C_1}(x),x_0],[\pi_{C_2}(x),x_0]$, by Proposition \ref{contravariance}.
\end{proof}
%%%%%%%%%%%%%%%%%%%%%%%%%%%%%%%%%%%%%%%%%%%%%%%%%%%%%%%%%%%%%%%%%%%%%%%%%%%%%%%%%%%%%%%%%%%%%%%%%%%%%%%%%%%%%%%%%%%%%%%%%%%%%%%%%%%%%%%%%%%%%%%%%%%%%%%%%%%%%%%%%%%%%%%%%%%%%%%%%%
%\begin{figure}[h]
%\includegraphics[width=8cm]{Example2}
%\centering
%\caption{Here the convex hull embeds into the $\ell^1$-product of $C_1$ and $C_2$}
%\label{example2}
%\end{figure}
Median space are in general far from being CAT(0) spaces or even geodesic spaces. Nevertheless, convexity in median space being defined by mean of intervals, convex subsets in median spaces are rigid enough to share many properties that are featured in CAT$(0)$ spaces.  
Let $X$ be a complete median space and let $C_1,C_2\subseteq X$ be two closed convex subsets. Then, their respective convex subsets $\pi_{C_1}(C_2)$ and $\pi_{C_2}(C_1)$ are isometric and the isometry is given by the gate projection. Moreover, the following median version of the Sandwich lemma ( \cite{Brid_Haef} Exercise II.2.12) holds:
%%%%%%%%%%%%%%%%%%%%%%%%%%%%%%%%%%%%%%%%%%%%%%%%%%%%%%%%%%%%%%%%%%%%%%%%%%%%%%%%%%%%%%%%%%%%%%%%%%%%%%%%%%%%%%%%%%%%%%%%%%%%%%%%%%%%%%%%%%%%%%%%%%%%%%%%%%%%%%%%%%%%%%%%%%%%%%%%%%
\begin{prop}[Proposition 2.21 \cite{Fior_superrigidity}]\label{description_gate}
Let $X$ be a median space and let $C_1,C_2\subseteq X$ be two closed convex subset. Then $Conv(\pi_{C_1}(C_2),\pi_{C_2}(C_1))$ is isometric to $\pi_{C_1}(C_2)\times	[x,\pi_{C_2}(x)]$ where $x$ is any point in $\pi_{C_1}(C_2)$.
\end{prop}
%%%%%%%%%%%%%%%%%%%%%%%%%%%%%%%%%%%%%%%%%%%%%%%%%%%%%%%%%%%%%%%%%%%%%%%%%%%%%%%%%%%%%%%%%%%%%%%%%%%%%%%%%%%%%%%%%%%%%%%
%%%%%%%%%%%%%%%%%%%%%%%%%%%%%%%%%%%%%%%%%%%%%%%%%%%%%%%%%%%%%%%%%%%%%%%%%%%%%%%%%%%%%%%%%%%%%%%%%%%%%%%%%%%%%%%%%%%%%%%%%%%%%%%%%%%%%%%%%%%%%%%%%%%%%%%%%%%%%%%%%%%%%%%%%%%%%%%%%%
\begin{rmk}
Proposition \ref{embed_lemm} can be extended to the case where $C_1$ and $C_2$ admits a common transverse halfspace by taking the projection of the convex hull between $C_1$ and $C_2$ into the product $C_1\times C_2\times \BB(C_1,C_2)$ endowed with the $\ell^1$-product metric, where $\BB(C_1,C_2):=Conv(\pi_{C_1}(C_2),\pi_{C_2}(C_1))$. The map is not necessarily an isometry, it is a 2-lipschitz embedding. For any two points in $Conv(C_1\cup C_2)$ separated by halfspaces which are transverse to both $C_1$ and $C_2$, the horizontal distance with respect to $C_1$ and $C_2$ is counted twice in $C_1\times C_2\times \BB(C_1,C_2)$.
\end{rmk}
%%%%%%%%%%%%%%%%%%%%%%%%%%%%%%%%%%%%%%%%%%%%%%%%%%%%%%%%%%%%%%%%%%%%%%%%%%%%%%%%%%%%%%%%%%%%%%%%%%%%%%%%%%%%%%%%%%%%%%%%%%%%%%%%%%%%%%%%%%%%%%%%%%%%%%%%%%%%%%%%%%%%%%%%%%%%%%%%%%

%%%%%%%%%%%%%%%%%%%%%%%%%%%%%%%%%%%%%%%%%%%%%%%%%%%%%%%%%%%%%%%%%%%%%%%%%%%%%%%%%%%%%%%%%%%%%%%%%%%%%%%%%%%%%%%%%%%%%%%%%%%%%%%%%%%%%%%%%%%%%%%%%%%%%%%%%%%%%%%%%%%%%%%%%%%%%%%%%%
%%%%%%%%%%%%%%%%%%%%%%%%%%%%%%%%%%%%%%%%%%%%%%%%%%%%%%%%%%%%%%%%%%%%%%%%%%%%%%%%%%%%%%%%%%%%%%%%%%%%%%%%%%%%%%%%%%%%%%%%%%%%%%%%%%%%%%%%%%%%%%%%%%%%%%%%%%%%%%%%%%%%%%%%%%%%%%%%%%%%%%%%%%%%%%%%%%%%%%%%%%%%%%%%%%%%%%%%%%%%%%%%%%%%%%%%%%%%%%%%%%%%%%%%%%%%%%%%%%%%%%%%%%%%%%%%%%%%%%%%%%%%%%%%%%%%%%%%%%%%%%%%%%%%%%%%%%%%%%%%%%%%%%%%%%%%%%%%%%%%%%%%%%%%%%%%%%%%
\section{Characterization of compactness by mean of halfspaces}
%%%%%%%%%%%%%%%%%%%%%%%%%%%%%%%%%%%%%%%%%%%%%%%%%%%%%%%%%%%%%%%%%%%%%%%%%%%%%%%%%%%%%%%%%%%%%%%%%%%%%%%%%%%%%%%%%%%%%%%%%%%%%%%%%%%%%%%%%%%%%%%%%%%%%%%%%%%%%%%%%%%%%%%%%%%%%%%%%%%%%%%%%%%%%%%%%%%%%%%%%%%%%%%%%%%%%%%%%%%%%%%%%%%%%%%%%%%%%%%%%%%%%%%%%%%%%%%%%%%%%%%%%%%%%%%%%%%%%%%%%%%%%%%%%%%%%%%%%%%%%%%%%%%%%%%%%%%%%%%%%%%%%%%%%%%%%%%%%%%%%%%%%%%%%%%%%%%%
 In Subsection \ref{subsection_compact_subset} below, we recall some results about the convex hull of compact subsets in a median space. Subsection \ref{subsection_compactness_theorem} is devoted to the proof of Theorem \ref{local_compactness}.
%%%%%%%%%%%%%%%%%%%%%%%%%%%%%%%%%%%%%%%%%%%%%%%%%%%%%%%%%%%%%%%%%%%%%%%%%%%%%%%%%%%%%%%%%%%%%%%%%%%%%%%%%%%%%%%%%%%%%%%%%%%%%%%%%%%%%%%%%%%%%%%%%%%%%%%%%%%%%%%%%%%%%%%%%%%%%%%%%%%%%%%%%%%%%%%%%%%%%%%%%%%%%%%%%%%%%%%%%%%%%%%%%%%%%%%%%%%
\subsection{Convex hull of compact subsets}\label{subsection_compact_subset}
%%%%%%%%%%%%%%%%%%%%%%%%%%%%%%%%%%%%%%%%%%%%%%%%%%%%%%%%%%%%%%%%%%%%%%%%%%%%%%%%%%%%%%%%%%%%%%%%%%%%%%%%%%%%%%%%%%%%%%%%%%%%%%%%%%%%%%%%%%%%%%%%%%%%%%%%%%%%%%%%%%%%%%%%%%%%%%%%%%%%%%%%%%%%%%%%%%%%%%%%%%%%%%%%%%%%%%%%%%%%%%%%%%%

%\begin{thm}\label{Intervals}
%Let $X$ be a median space of rank $n$. Then for any $a,b\in X$ there exists $l_1,...,l_p>0$ where $p\leq n$ and $\displaystyle{\sum_{i=1}^p l_p}=d(a,b)$ such that the interval $[a,b]$ embeds isometrically into the $\ell^1$ product $\displaystyle{\prod_{i=1}^p[0,l_i]}$. The point $a$ and $b$ are mapped to $(0,...,0)$ and $(l_1,...,l_p)$ respectively.
%\end{thm}
\par
It was shown in \cite{Fior_median_property} that any interval in a median space of rank $n$ embeds isometrically into $\RR^n$, see Proposition 2.19 therein. A direct consequence is that the convex hull of a finite subset in a complete finite rank median space is compact. More generally, the convex hull of a compact subset in a complete finite rank median space is also compact, see Lemma 13.2.11 in \cite{Bowd_median-algebras}. This is not necessarily true for infinite rank median space where the interval are not necessarily compact, consider for instance intervals in $L^1(\RR)$. However, under the assumption that the intervals are compact, the convex hull between any two convex compact subsets is also compact.
%%%%%%%%%%%%%%%%%%%%%%%%
%%%%%%%%%%%%%%%%%%%%%%%%%%%%%%%%%%%%%%%%%%%%%%%%%%%%%%%%%%%%%%%%%%%%%%%%%%%%%%%%%%%%%%%%%%%%%%%%%%%%%%%%%%%%%%%%%%%%%%%%%%%%%%%%%%%%%%%%%%%%%%%%%%%%%%%%%%%%%%%%%%%%%%%%%%%%%%%%%%%%%%%%%%%%%%%%%%%%
\begin{prop}\label{convex_hull_compact_convex}
Let $X$ be a complete median space which have compact intervals. Then the convex hull between any two compact convex subsets is also compact.
\end{prop}
%%%%%%%%%%%%%%%%%%%%%%%%
%%%%%%%%%%%%%%%%%%%%%%%%%%%%%%%%%%%%%%%%%%%%%%%%%%%%%%%%%%%%%%%%%%%%%%%%%%%%%%%%%%%%%%%%%%%%%%%%%%%%%%%%%%%%%%%%%%%%%%%%%%%%%%%%%%%%%%%%%%%%%%%%%%%%%%%%%%%%%%%%%%%%%%%%%%%%%%%%%%%%%%%%%%%%%%%%%%%%
In particular, we have the following:
\begin{cor}\label{convex_hull_finite_subset_compact}
Let $X$ be a complete medians space with compact intervals. Then the convex hull of any finite subset is compact.
\end{cor}
%%%%%%%%%%%%%%%%%%%%%%%%%%%%%%%%%%%%%%%%%%%%%%%%%%%%%%%%%%%%%%%%%%%%%%%%%%%%%%%%%%%%%%%%%%%%%%%%%%%%%%%%%%%%%%%%%%%%%%%%%%%%%%%%%%%%%%%%%%%%%%%%%%%%%%%%%%%%%%%%%%%%%%%%%%%%%%%%%%%%%%%%%%%%%%%%%%%%
Before proving Proposition \ref{convex_hull_compact_convex}, we will be needing some lemmas:
\begin{lem}\label{distance_interval_from_convex}
Let $C\subseteq X$ be a convex subset and a point $x\in [a,b]\subseteq X$, we have:
\[
d(x,C)\leq d(a,C)+d(b,C)
\] 
\end{lem}
%%%%%%%%%%%%%%%%%%%%%%%%%%%%%%%%%%%%%%%%%%%%%%%%%%%%%%%%%%%%%%%%%%%%%%%%%%%%%%%%%%%%%%%%%%%%%%%%%%%%%%%%%%%%%%%%%%%%%%%%%%%%%%%%%%%%%%%%%%%%%%%%%%%%%%%%%%%%%%%%%%%%%%%%%%%%%%%%%%%%%%%%%%%%%%%%%%%%%%%%%%%%%%%%%%%%%%%%%%%%%%%%%%%%%%%%%%%%%%%%%%%%%%%%%%%%%%%%%%%%%%%%%%%%%%%%%%%%%%%%%%%%%%%%%%%%%%%%%%%%%%%%%%%%%%%%%%%%%%%%%%%%%%%%%%%%%%%%%%%%%%%%%%%
\begin{proof}
Any halfspace which separates $C$ from $x$ must separates it either from $a$ or from $b$ (or from both). Thus we get
\[
d(x,C)=\mu(\mathcal{W}(x,C))\leq  \mu(\mathcal{W}(a,C))+ \mu(\mathcal{W}(b,C))=d(a,C)+d(b,C)
\] 
\end{proof}
%%%%%%%%%%%%%%%%%%%%%%%%%%%%%%%%%%%%%%%%%%%%%%%%%%%%%%%%%%%%%%%%%%%%%%%%%%%%%%%%%%%%%%%%%%%%%%%%%%%%%%%%%%%%%%%%%%%%%%%%%%%%%%%%%%%%%%%%%%%%%%%%%%%%%%%%%%%%%%%%%%%%%%%%%%%%%%%%%%%%%%%%%%%%%%%%%%%%%%%%%%%%%%%%%%%%%%%%%%%%%%%%%%%%%%%%%%%%%%%%%%%%%%%%%%%%%%%%%%%%%%%%

We deduce the following lemma:
%%%%%%%%%%%%%%%%%%%%%%%%%%%%%%%%%%%%%%%%%%%%%%%%%%%%%%%%%%%%%%%%%%%%%%%%%%%%%%%%%%%%%%%%%%%%%%%%%%%%%%%%%%%%%%%%%%%%%%%%%%%%%%%%%%%%%%%%%%%%%%%%%%%%%%%%%%%%%%%%%%%%%%%%%%%%%%%%%%%%%%%%%%%%%%%%%%%%%%%%%%%%%%%%%%%%%%%%%%%%%%%%%%%%%%%%%%%%%%%%%%%%%%%%%%%%%%%%%%%%%%%%
\begin{lem}\label{join_between_closed_convex_is_closed}
Let $X$ be a complete median space. Then the join between any two closed convex subsets is closed. 
\end{lem}
%%%%%%%%%%%%%%%%%%%%%%%%%%%%%%%%%%%%%%%%%%%%%%%%%%%%%%%%%%%%%%%%%%%%%%%%%%%%%%%%%%%%%%%%%%%%%%%%%%%%%%%%%%%%%%%%%%%%%%%%%%%%%%%%%%%%%%%%%%%%%%%%%%%%%%%%%%%%%%%%%%%%%%%%%%%%%%%%%%%%%%%%%%%%%%%%%%%%
\begin{proof}
Let us consider two convex subsets $C_1,C_2\subseteq X$ and let $(x_n)_{n\in\NN}\subseteq [C_1,C_2]$ be a sequence of points converging to $x\in X$. Note that each $x_n$ lies in the interval $[\pi_{C_1}(x_n),\pi_{C_2}(x_n)]$. As gate projections are $1$-lipschitz, the sequences $(\pi_{C_1}(x_n))_{n\in\NN}$ and $(\pi_{C_2}(x_n))_{n\in\NN}$ are Cauchy sequences. Thus they converge to $a\in C_1$ and  $b\in C_2$ respectively. In the other hand, we have 
\begin{eqnarray*}
d(x,[a,b])=d(x,m(x,a,b)) &\leq & d(x,x_n) + d(x_n,m(x_n,a,b)) + d(m(x_n,a,b),m(x,a,b))\\
                         &=& d(x,x_n) + d(m(x_n,a,b),m(x,a,b)) + d(x_n,[a,b])
\end{eqnarray*}

Where the right side tend to zero when $n$ goes to infinity by the continuity of the projection and Lemma \ref{distance_interval_from_convex}.
\end{proof}
%%%%%%%%%%%%%%%%%%%%%%%%%%%%%%%%%%%%%%%%%%%%%%%%%%%%%%%%%%%%%%%%%%%%%%%%%%%%%%%%%%%%%%%%%%%%%%%%%%%%%%%%%%%%%%%%%%%%%%%%%%%%%%%%%%%%%%%%%%%%%%%%%%%%%%%%%%%%%%%%%%%%%%%%%%%%%%%%%%%%%%%%%%%%%%%%%%%%%%%%%%%%%%%%%%%%%%%%%%%%%%%%%%%%%%%%%%%%%%%%%%%%%%%%%%%%%%%%%%%%%%%%
\begin{rmk}
The join between two closed subset of a complete median space of finite rank is not necessarily closed, even if we assume that the subsets are bounded. Take for instance the product of the closed segment of the real line with a star like simplicial tree with infinite edges of length $1$. One may consider then a sequence of points such that their projections into the star like simplicial tree run injectively through its vertices and their projections into the closed segment of the real line accumulate around $0$ but never attain it.
\end{rmk}
%%%%%%%%%%%%%%%%%%%%%%%%%%%%%%%%%%%%%%%%%%%%%%%%%%%%%%%%%%%%%%%%%%%%%%%%%%%%%%%%%%%%%%%%%%%%%%%%%%%%%%%%%%%%%%%%%%%%%%%%%%%%%%%%%%%%%%%%%%%%%%%%%%%%%%%%%%%%%%%%%%%%%%%%%%%%%%%%%%%%%%%%%%%%%%%%%%%%%%%%%%%%%%%%%%%%%%%%%%%%%%%%%%%%%%%%%%%%%%%%
%%%%%%%%%%%%%%%%%%%%%%%%%%%%%%%%%%%%%%%%%%%%%%%%%%%%%%%%%%%%%%%%%%%%%%%%%%%%%%%%%%%%%%%%%%%%%%%%%%%%%%%%%%%%%%%%%%%%%%%%%%%%%%%%%%%%%%%%%%%%%%%%%%%%%%%%%%%%%%%%%%%%%%%%%%%%%%%%%%%%%%%%%%%%%%%%%%%%%%%%%%%%%%%%%%%%%%%%%%%%%%%%%%%%%%%%%%%%%%%%%%%%%%%%%%%%%%%%%%%%%%%%
\begin{proof}[Proof of Proposition \ref{convex_hull_compact_convex}]
By Proposition \ref{embed_lemm} and Lemma \ref{join_between_closed_convex_is_closed}, for any closed convex subset $C_1,C_2\subseteq X$, their convex hull embeds as a closed subsets into the $\ell^1$-product of $C_1$, $C_2$ and $\BB(C_1,C_2)$, where the latter, by Proposition \ref{description_gate}, is isometric to an interval and a closed convex subset of $C_1$, which is compact. 
\end{proof}
%%%%%%%%%%%%%%%%%%%%%%%%%%%%%%%%%%%%%%%%%%%%%%%%%%%%%%%%%%%%%%%%%%%%%%%%%%%%%%%%%%%%%%%%%%%%%%%%%%%%%%%%%%%%%%%%%%%%%%%%%%%%%%%%%%%%%%%%%%%%%%%%%%%%%%%%%%%%%%%%%%%%%%%%%%%%%%%%%%%%%%%%%%%%%%%%%%%%%%%%%%%
%FALSE
%\begin{rmk}
%The convex hull of compact subset of a complete median space is not necessarily compact, even if we assume the median space to be locally convex with compact intervals. Consider for instance the median space obtained by gluing the cubes $[0,\frac{1}{n}]^n$, where $n\in\NN^*$, along the origin. The subsets consisting of taking the canonical edges of each cubes along with the origin is compact but its convex hull is not.
%\end{rmk}
%%%%%%%%%%%%%%%%%%%%%%%%%%%%%%%%%%%%%%%%%%%%%%%%%%%%%%%%%%%%%%%%%%%%%%%%%%%%%%%%%%%%%%%%%%%%%%%%%%%%%%%%%%%%%%%%%%%%%%%%%%%%%%%%%%%%%%%%%%%%%%%%%%%%%%%%%%%%%%%%%%%%%%%%%%%%%%%%%%%%%%%%%%%%%%%%%%%%%%%%%%%
In the following lemma, we show that the Hausdorff limit of compact subsets is a relatively compact subset:
%%%%%%%%%%%%%%%%%%%%%%%%%%%%%%%%%%%%%%%%%%%%%%%%%%%%%%%%%%%%%%%%%%%%%%%%%%%%%%%%%%%%%%%%%%%%%%%%%%%%%%%%%%%%%%%%%%%%%%%%%%%%%%%%%%%%%%%%%%%%%%%%%%%%%%%%%%%%%%%%%%%%%%%%%%%%%%%%%%%%%%%%%%%%%%%%%%%%%%%%%%%
\begin{lem}\label{lemma_hausdorf_limit_compact}
Let $X$ be a complete metric space and let $(K_i)_{i\in\NN}$ be a sequence of compact subsets of $X$ which converge, with respect to the Hausdorff metric, to a subset $K\subseteq X$. Then the closure of $K$ is a compact subset of $X$.
\end{lem}
\begin{proof}
Note that up to considering  the sequence of subsets $\tilde{K}_n=\displaystyle{\bigcup_{i=0}^n K_n}$, there is no loss of generality to assume that the sequence $(K_n)_{n\in\NN}$ is ascending. Let $(x_i)_{i\in\NN}$ be a sequence of points in $K$ and let us show that it contains a subsequence which converge to a point in $X$. If there exist $K_i$ such that it contains an infinite subsequence of $(x_i)_{i\in\NN}$ then we are done. Let us assume then that each $K_i$ contains finitely many points of $(x_i)_{i\in\NN}$. For each $n\in\NN$, let $i_n\in\NN$ be such that $d_{Haus}(K_{i_n},K) < \frac{1}{n}$. We consider a sequence $(\tilde{x}_{n,i})_{i\in\NN}\subseteq K_{i_n}$ such that $d(\tilde{x}_{n,i},x_i)<\frac{1}{n}$ for any $i\in\NN$. The subset $K_{n}$ being compact, there exists subsequence $(\tilde{x}_{n,\Phi(i)})_{i\in\NN}$ which converges to a point $\Tilde{x}_n\in K_{i_n}$. Iterating the same process for each $n$ and considering at each step a subsequence of the previous subsequence, we obtain the following configuration:
\begin{itemize}
\item For each $n$ there exist a sequence $(\tilde{x}_{n,i})_{i\in\NN}\subseteq K_{i_n}$ such that for any $i\in\NN$ we have $d(\tilde{x}_{n,i},x_{\Phi_n(i)})<\frac{1}{n}$, where each $\Phi_n:\NN\rightarrow \NN$ is an increasing injective  map and $\Phi_{n+1}(\NN)\subseteq\Phi_n(\NN)$. 
\item Each sequence $(\tilde{x}_{n,i})_{i\in\NN}$ converges to a point $\tilde{x}_n\in K_{i_n}$.
\end{itemize} 
\textbf{Claim 1:} The sequence $(\tilde{x}_n)_{n\in\NN}$ is a Cauchy sequence.\\
 Let us show that for any $n,m\in\NN$ we have $d(\tilde{x}_n,\tilde{x}_m)<\frac{1}{n}+\frac{1}{m}$. Let us fix $n,m\in\NN$ such that $m\geq n$ and consider $\epsilon>0$. Let $N\in\NN$ be such that for any integer $i\geq N$, we have $d(\tilde{x}_{n,i},\tilde{x}_n)<\epsilon$ and $d(\tilde{x}_{m,i},\tilde{x}_m)<\epsilon$. Hence for any $i\in\NN$ such that $min(i,\Phi_n^{-1}(\Phi_m(i)))\geq N$, we have:
\begin{eqnarray*}
d(\tilde{x}_n,\tilde{x}_m)&\leq & d(\tilde{x}_n,\tilde{x}_{n,\Phi_n^{-1}(\Phi_m(i))})+d(\tilde{x}_{n,\Phi_n^{-1}(\Phi_m(i))},x_{\Phi_m(i)})+d(x_{\Phi_m(i)},\tilde{x}_{m,i})+d(\tilde{x}_{m,i},\tilde{x}_m)\\
                          &\leq & \epsilon + \frac{1}{n} + \frac{1}{m} + \epsilon \\
                          &\leq & \frac{1}{n} + \frac{1}{m} + 2\epsilon .
\end{eqnarray*}
The $\epsilon$ being arbitrary, we conclude that $d(\tilde{x}_n,\tilde{x}_m)<\frac{1}{n}+\frac{1}{m}$.  
\par As the space $X$ is complete, the sequence $(\tilde{x}_n)_{n\in\NN}$ converges to a point $\tilde{x}$.\\
\textbf{Claim 2:} The point $\tilde{x}$ is an accumulation point for the sequence $(x_i)_{i\in\NN}$.\\ Let us fix $\epsilon>0$ and consider $n\in\NN$ such that $d(\tilde{x},\tilde{x}_n)<\epsilon$. For any $i\in\NN$ big enough such that $d(\tilde{x}_n,\tilde{x}_{n,i})\leq \epsilon$, we get 
\begin{eqnarray*}
d(\tilde{x},x_{\Phi_n(i)}) &\leq& d(\tilde{x},\tilde{x}_n) + d(\tilde{x}_n,\tilde{x}_{n,i}) +  d(\tilde{x}_{n,i},x_{\Phi_n(i)})\\
                                &\leq& 2 \epsilon + \frac{1}{n}
\end{eqnarray*} 
Which proves Claim 2 and finishes the proof of the lemma.
\end{proof}
%%%%%%%%%%%%%%%%%%%%%%%%%%%%%%%%%%%%%%%%%%%%%%%%%%%%%%%%%%%%%%%%%%%%%%%%%%%%%%%%%%%%%%%%%%%%%%%%%%%%%%%%%%%%%%%%%%%%%%%%%%%%%%%%%%%%%%%%%%%%%%%%%%%%%%%%%%%%%%%%%%%%%%%%%%%%%%%%%%%%%%%%%%%%%%%%%%%%%%%%%%%%%%%%%%%%%%%%%%%%%%%%%%%%%%%%%%%
\subsection{Proof of Theorem \ref{local_compactness}}\label{subsection_compactness_theorem}
%%%%%%%%%%%%%%%%%%%%%%%%%%%%%%%%%%%%%%%%%%%%%%%%%%%%%%%%%%%%%%%%%%%%%%%%%%%%%%%%%%%%%%%%%%%%%%%%%%%%%%%%%%%%%%%%%%%%%%%%%%%%%%%%%%%%%%%%%%%%%%%%%%%%%%%%%%%%%%%%%%%%%%%%%%%%%%%%%%%%%%%%%%%%%%%%%%%%%%%%%%%%%%%%%%%%%%%%%%%%%%%%%%%
The existence of infinitely many pairwise disjoint halfspaces in a bounded median space is not necessarily an obstruction to the compactness of the latter. Take for instance the example of $\RR$-trees where one may have a compact real tree with infinitely many branches. Consider the simple example of a star-like tree obtained from gluing the intervals $[0,\frac{1}{n}]$ at $0$, where each $]0,\frac{1}{n}]$ gives rise to a halfspace in the quotient space. The other element to take into account here in the obstruction to the compactness is the ``depth" of each branching, which motivates the following definition.
\begin{defi}
Let $X$ be a complete median space of finite rank and let $\hh\in\HH (X)$ be a halfspace. We call the \textbf{\textit{depth}} of $\hh$ in $A\subseteq X$, that we denote by $depth_A(H)$, the maximum distance between points lying in $\hh\cap A$ and the hyperplane $\hat{\hh}$ bounding $\hh$, i.e. $depth_A(H):=sup\{d(x,\hat{\hh}) \ | \ x\in\hh\cap A \}$.
\end{defi} 
%%%%%%%%%%%%%%%%%%%%%%%%%%%%%%%%%%%%%%%%%%%%%%%%%%%%%%%%%%%%%%%%%%%%%%%%%%%%%%%%%%%%%%%%%%%%%%%%%%%%%%%%%%%%%%%%%%%%%%%%%%%%%%%%%%%%%%%%%%%%%%%%%%%%%%%%%%%%%%%%%%%%%%%%%%%%%%%%%%

%%%%%%%%%%%%%%%%%%%%%%%%%%%%%%%%%%%%%%%%%%%%%%%%%%%%%%%%%%%%%%%%%%%%%%%%%%%%%%%%%%%%%%%%%%%%%%%%%%%%%%%%%%%%%%%%%%%%%%%%%%%%%%%%%%%%%%%%%%%%%%%%%%%%%%%%%%%%%%%%%%%%%%%%%%%%%%%%%%
The following lemma is a strengthening of Lemma \ref{lemma_local_convexity}:
%%%%%%%%%%%%%%%%%%%%%%%%%%%%%%%%%%%%%%%%%%%%%%%%%%%%%%%%%%%%%%%%%%%%%%%%%%%%%%%%%%%%%%%%%%%%%%%%%%%%%%%%%%%%%%%%%%%%%%%%%%%%%%%%%%%%%%%%%%%%%%%%%%%%%%%%%%%%%%%%%%%%%%%%%%%%%%%%%%
\begin{lem}\label{transverse_halfspace_at_given_distance}
Let $X$ be a complete connected median space of rank $n$ and let $a,b\in X$. Then for any small $\epsilon>0$ which is smaller then $\frac{d(a,b)}{n}$, there exist a pairwise transverse halfspaces $\hh_1,...,\hh_k\in\HH(a,b)$, where $k\leq n$, such that for all $i\in\{1,...,k\}$ we have:
\begin{itemize}
\item $d(\hh_i^c,b)\geq\epsilon$.
\item $d(a,\displaystyle{\bigcap_{i=1}^k\hh_i})\geq d(a,b)-\frac{n(n+1)}{2}\epsilon$.
\end{itemize}
\end{lem}
%%%%%%%%%%%%%%%%%%%%%%%%%%%%%%%%%%%%%%%%%%%%%%%%%%%%%%%%%%%%%%%%%%%%%%%%%%%%%%%%%%%%%%%%%%%%%%%%%%%%%%%%%%%%%%%%%%%%%%%%%%%%%%%%%%%%%%%%%%%%%%%%%%%%%%%%%%%%%%%%%%%%%%%%%%%%%%%%%%
\begin{proof}
Let us proceed by induction on the rank of $X$. The lemma is trivially true for complete connected median space of rank $1$, that is, in the case of $\RR$-trees.\par
 Let us assume then that the lemma is true for complete connected median space of rank $n-1$. Let us fix $a,b\in X$ and $0<\epsilon\leq\frac{d(a,b)}{n}$. Let us consider $x\in[a,b]$ such that $d(x,b)=n\epsilon$, such point $x$ exists as the space $X$ is connected. By Lemma \ref{lemma_local_convexity}, there exists a halfspace $\hh\in\Delta(x,b)$ such that $d(\hh^c,b)\geq\frac{d(x,b)}{n}\geq \epsilon$ and $d(x,\hh)=0$.
For simplicity let us set $\tilde{a}:=\pi_{\hat{\hh}}(a)$. The hyperplane $\hat{\hh}$ being a median space of rank less than $n$, there exists then a family of pairwise transverse halfspaces $\hh_1,...,\hh_k\in\HH(\tilde{a},x)$, where $k\leq n-1$ and such that for any $i\in\{1,...,k\}$, we have $d(\hh_i^c,x)\geq\epsilon$ and $d(\tilde{a},\displaystyle{\bigcap_{i=1}^k\hh_i})\geq d(\tilde{a},x)-\frac{n(n-1)}{2}\epsilon$. Note that each halfspace $\hh_i$ is transverse to $\hh$ as it is, with its complement, the lift of a halfspace of $\hat{\hh}$ with non empty interior (see Proposition \ref{lift_and_transversality} below). As the point $x$ belongs to the interval $[a,b]$, any halfspace which separates $x$ from $\hh_i^c$ must also separates $b$ from $\hh_i^c$. Hence each $\hh_i^c$ is at distance greater than $\epsilon>0$ from $b$. It last to show that the intersection of $\hh$ with the $\hh_i$'s is at distance greater than $d(a,b)-\frac{n(n+1)}{2}\epsilon$ from $a$. 
For any $y\in\displaystyle{(\bigcap_{i=1}^k\hh_i)\cap\hh}$, the point $\tilde{a}=\pi_{\hat{\hh}}(a)=\pi_{\overline{\hh}}(a)$ belongs to the interval $[a,y]$. Thus we have:
\begin{equation}\label{inequality_transverse_halfspaces_at_given_distance}
d(a,y)= d(a,\tilde{a})+d(\tilde{a},y)\geq d(a,\tilde{a})+d(\tilde{a},x)-\frac{n(n-1)}{2}\epsilon
\end{equation}
The equality  $d(\tilde{a},x)+d(a,\tilde{a})=d(a,x)=d(a,b)-n\epsilon$ combined with Inequality (\ref{inequality_transverse_halfspaces_at_given_distance}) above yields
\[
d(a,y)\geq d(a,b)-n\epsilon-\frac{n(n-1)}{2}\epsilon=d(a,b)-\frac{n(n+1)}{2}\epsilon
\]
\end{proof}
%%%%%%%%%%%%%%%%%%%%%%%%%%%%%%%%%%%%%%%%%%%%%%%%%%%%%%%%%%%%%%%%%%%%%%%%%%%%%%%%%%%%%%%%%%%%%%%%%%%%%%%%%%%%%%%%%%%%%%%%%%%%%%%%%%%%%%%%%%%%%%%%%%%%%%%%%%%%%%%%%%%%%%%%%%%%%%%%%%
One chunk of the proof of Theorem \ref{local_compactness} is proved in the following technical lemma.
%%%%%%%%%%%%%%%%%%%%%%%%%%%%%%%%%%%%%%%%%%%%%%%%%%%%%%%%%%%%%%%%%%%%%%%%%%%%%%%%%%%%%%%%%%%%%%%%%%%%%%%%%%%%%%%%%%%%%%%%%%%%%%%%%%%%%%%%%%%%%%%%%%%%%%%%%%%%%%%%%%%%%%%%%%%%%%%%%%
\begin{lem}\label{depth_hyperplane}
Let $C$ be a complete connected median space of finite rank which is bounded. Let $\hh\subset X$ be a halfspace, then for any $\epsilon>0$ such that there is no two disjoint halfspaces of depth bigger than $\epsilon$ contained in $\hh$ and for any $a\in \hh$ such that $d(a,\hat{\hh})\geq depth_c(\hh)-\epsilon$, the convex hull $Conv(\{a\}\cup\hat{\hh})$ it at Hausdorff distance less than $(n(n+1)+1)\epsilon$ from $\hh$, where $n$ is the rank of $C$.
\end{lem}
%%%%%%%%%%%%%%%%%%%%%%%%%%%%%%%%%%%%%%%%%%%%%%%%%%%%%%%%%%%%%%%%%%%%%%%%%%%%%%%%%%%%%%%%%%%%%%%%%%%%%%%%%%%%%%%%%%%%%%%%%%%%%%%%%%%%%%%%%%%%%%%%%%%%%%%%%%%%%%%%%%%%%%%%%%%%%%%%%%
\begin{proof}
 Let us choose a point $a\in\hh\cap C$ such that $d(a,\hat{\hh})\geq depth_C(\hh)-\epsilon$. We set $C_{\hh}=Conv(\hat{\hh},a)$ and  take a point $x\in\hh\cap C$ lying outside $C_{\hh}$. We consider its projections into $C_{\hh}$ and $[a,\pi_{\hat{\hh}}(a)]$ that we denote by $x_{C_{\hh}}:=\pi_{C_{\hh}}(x)$ and $\tilde{x}:=m(x,a,\pi_{\hat{\hh}}(a))$ respectively (See Figure \ref{Figure_compactness_theorem}) .\\
\begin{figure}[h]
\includegraphics[scale=0.2]{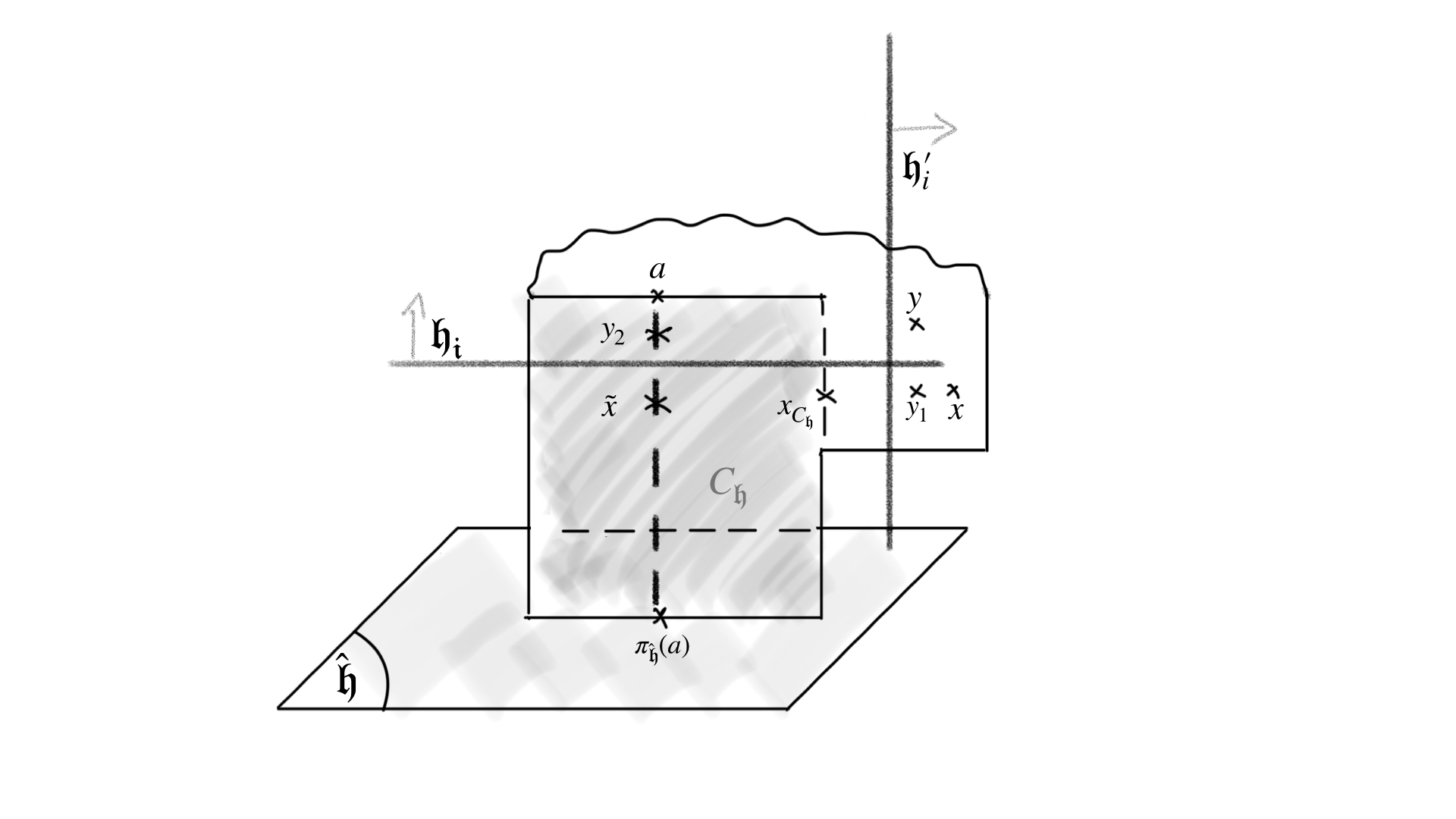}
\centering
\caption{The configuration arising in the second part of the proof of Lemma \ref{depth_hyperplane}}
\label{Figure_compactness_theorem}
\end{figure}
% By Theorem \ref{Intervals}, there exists $l_1,...,l_p,l_1',...,l_q'>0$, where $p,q\leq n$, such that the intervals $[\tilde{x},a]$ and $[x_{C_{\hh}},x]$ embed isometrically into $\displaystyle{\prod_{i=1}^p [0,l_i]}$ and $\displaystyle{\prod_{i=1}^p [0,l_i']} $ respectively, with $d(\tilde{x},a)=\displaystyle{\sum_{i=1}^p l_i}$ and $d(x_{C_{\hh}},x)=\displaystyle{\sum_{i=1}^q l_i'}$. Let $\hh_i$, respectively $\hh_j'$, be the lift of the intrinsic halfspace $[l_i-\epsilon,l_i]$, respectively $[l_i'-\epsilon,l_i']$, to the interval $[\tilde{x},a]$, respectively $[x_{C_{\hh}},x]$. 
 Let us first show that $d(\tilde{x},\hat{\hh})=d(x_{C_{\hh}},\hat{\hh})$. As the interval $[\pi_{\hat{\hh}}(a),a]$ lies in $C_{\hh}$, we have:
\[
 \tilde{x}=m(x,a,\pi_{\hat{\hh}}(a))=\pi_{[\pi_{\hat{\hh}}(a),a]}(x)=\pi_{[\pi_{\hat{\hh}}(a),a]}(\pi_{C_{\hh}}(x))=\pi_{[\pi_{\hat{\hh}}(a),a]}(x_{C_{\hh}})
\]
Hence, any halfspace separating $\tilde{x}$ from $\hat{\hh}$, separates $a$ from $\pi_{\hat{\hh}}(a)$, therefore it must separates $x_{C_{\hh}}$ from $\hat{\hh}$ as well by Lemma \ref{constraint_on_the_projection}. In the other hand, note that any halfspace separating any point in $C_{\hh}$ from $\hat{\hh}$ must separates the point $a$ from $\hat{\hh}$ as the convex subset $C_{\hh}$ is the convex hull of $\hat{\hh}\cup\{a\}$. Hence by Lemma \ref{constraint_on_the_projection}, any halfspace separating $x_{C_{\hh}}$ from $\hat{\hh}$ separates also $\tilde{x}$ from $\hat{\hh}$. Therefore, the two walls intervals $\mathcal{W} (\tilde{x},\hat{\hh})$ and $\mathcal{W} (x_{C_{\hh}},\hat{\hh})$ coincide, which implies the equality $d(\tilde{x},\hat{\hh})=d(x_{C_{\hh}},\hat{\hh})$. We deduce then the following:
\begin{eqnarray*}
d(x,x_{C_{\hh}})&=& d(x,\hat{\hh})-d(x_{C_{\hh}},\hat{\hh})\\
            &\leq & depth_C(H)-d(\tilde{x},\hat{\hh})\\
            &\leq & d(a,\hat{\hh})+\epsilon -d(\tilde{x},\hat{\hh})
\end{eqnarray*}
As $\tilde{x}$ lies in the interval $[a,\pi_{\hh}(a)]$, then its projection into $\hat{\hh}$ is precisely the point $\pi_\hh(a)$. Hence $d(a,\hat{\hh})=d(\tilde{x},\pi_{\hh}(a))$. Replacing the latter in the inequality above, we get:
\begin{equation}\label{distance_x_from_C_H}
d(x,x_{C_{\hh}})\leq d(a,\tilde{x})+\epsilon
\end{equation}
\par If the distance between $x$ and $C_{\hh}$ is less than $(n+1)\epsilon$, then there is nothing to show. Let us assume then that $d(x,C_{\hh})\geq (n+1)\epsilon$, which by Inequality (\ref{distance_x_from_C_H}) above, implies also that $d(a,\tilde{x})\geq n\epsilon$. Hence by Lemma \ref{transverse_halfspace_at_given_distance}, there exist two families of pairwise transverse halfspaces $\{\hh_1,...,\hh_p\}\subseteq \HH (\tilde{x},a)$ and $\{\hh_1',...,\hh_q'\}\subseteq \HH (x_{C_{\hh}},x)$ such that the halfspaces $\hh_i^c$ and $\hh_j'^c$ are of depth bigger than $\epsilon$ and verify the following:
\[
d(\tilde{x},\displaystyle{\bigcap_{i=1}^p\hh_i})\geq d(\tilde{x},a)-\frac{n(n+1)}{2}\epsilon \ \ \ \text{and}\ \ \ d(x_{C_{\hh}},\displaystyle{\bigcap_{i=1}^q\hh_i})\geq d(x_{C_{\hh}},x)-\frac{n(n+1)}{2}\epsilon.
\] 
By assumption, the halfspace $\hh$ does not contain two disjoint halfspace of depth bigger than $\epsilon$. Hence, the halfspaces $\hh_i^c$ and $\hh_j'^c$ are not disjoint for any $i\in\{1,...,p\}$ and $j\in\{1,...,q\}$. In the other hand, any halfspace $\hh$ in $\Delta(\hat{\hh},\tilde{x})=\Delta(\hat{\hh},x_{C_\hh})$ contains the points $a$ and $x$ (note that they do not not necessarily contain the halfspaces $\hh_i$ and $\hh_j$). Hence the intersection $\displaystyle{(\bigcap_{\hh'\in\Delta(\hat{\hh},\tilde{x})}\hh')}$ contains the interval $[a,x]$. Therefore by Helly's Theorem \ref{helly_theorem}, the intersection $\displaystyle{(\bigcap_{\hh'\in\Delta(\hat{\hh},\tilde{x})}\hh')\cap(\bigcap_{i=1}^p\hh_i^c)\cap(\bigcap_{j=1}^q\hh_j'^c) }$ is not empty. Let us consider a point $y$ in the latter intersection and let $y_1$ and $y_2$ be its projections into the interval $[x_{C_{\hh}},x]$ and $[\tilde{x},a]$ respectively. We claim the following:
\begin{equation}\label{second_inequality}
d(y,\hat{\hh})\geq d(x_{C_{\hh}},\hat{\hh}) +  d(x_{C_{\hh}},y_1) + d(\tilde{x},y_2) 
\end{equation}
Indeed, by construction we have the following inclusion:
\[
\mathcal{W}(y_1,x_{C_{\hh}})\cup\mathcal{W}(y_2,\tilde{x})\cup \mathcal{W}(x_{C_{\hh}},\hat{\hh})\subseteq \mathcal{W}(y,\hat{\hh})
\]
In the other hand, all the wall intervals arising on the left hand of the inclusion are disjoints, therefore we get:
\begin{eqnarray*}
\mu(\mathcal{W}(y_1,x_{C_{\hh}})\cup\mathcal{W}(y_2,\tilde{x})\cup \mathcal{W}(x_{C_{\hh}},\hat{\hh}))&=& \mu(\mathcal{W}(y_1,x_{C_{\hh}}))+\mu(\mathcal{W}(y_2,\tilde{x}))+ \mu(\mathcal{W}(x_{C_{\hh}}),\hat{\hh}))\\
         &=& d(y_1,x_{C_{\hh}})+d(y_2,\tilde{x})+d(x_{C_{\hh}},\hat{\hh})\\
         &\leq & \mu(\mathcal{W}(y,\hat{\hh}))=d(y,\hat{\hh})\\ 
\end{eqnarray*} 
\par Having the inequality \ref{second_inequality} in hand, we get:
\begin{eqnarray*}
depth_C(H)&\geq & d(y,\hat{\hh})\\
          &\geq & d(x_{C_{\hh}},\hat{\hh}) + d(y_1,x_{C_{\hh}})+d(y_2,\tilde{x})\\
          &\geq & d(x_{C_{\hh}},\hat{\hh}) + d(x_{C_{\hh}},x) -\frac{n(n+1)}{2}\epsilon + d(\tilde{x},a)-\frac{n(n+1)}{2}\epsilon \\
\end{eqnarray*}
As $d(x_{C_{\hh}},\hat{\hh})=d(\tilde{x},\hat{\hh})$, we get:
\[
depth_C(H) \geq d(\tilde{x},\hat{\hh})+d(\tilde{x},a)+d(x_{C_{\hh}},x)-n(n+1)\epsilon=d(\hat{\hh},a)+d(x_{C_{\hh}},x)-n(n+1)\epsilon
\]

We deduce then the following:
\[
d(x_{C_{\hh}},x)\leq depth_C(H)-d(\hat{\hh},a) +n(n+1)\epsilon  \leq (n(n+1)+1)\epsilon
\]
Which finishes the proof.
\end{proof}
%%%%%%%%%%%%%%%%%%%%%%%%%%%%%%%%%%%%%%%%%%%%%%%%%%%%%%%%%%%%%%%%%%%%%%%%%%%%%%%%%%%%%%%%%%%%%%%%%%%%%%%%%%%%%%%%%%%%%%%%%%%%%%%%%%%%%%%%%%%%%%%%%%%%%%%%%%%%%%%%%%%%%%%%%%%%%%%%%%
\begin{proof}[Proof of Theorem \ref{local_compactness}]Let us first remark that there is no loss of generality to assume that $C$ is a closed convex subset. Indeed, the complete space $X$ being of finite rank, the convex hull of $C$ is compact if and only $C$ is compact (see Lemma 13.2.11 \cite{Bowd_median-algebras}). In the other hand, Remark \ref{remark_Hagen} implies that if a halfspaces is of depth less than $\epsilon$ in $C$, then it is of depth less than $n.\epsilon$ in $Conv(C)$ where $n$ is the rank of the space $X$. Hence, the condition $(3)$ holds with respect to $C$ if and only if it holds with respect to $Conv(C)$.
\par
The implication $1 \Rightarrow 2$ is obvious.
%%%%%%%%%%%%%%%%%%%%%%%%%%%%%%%%%%%%%%%%%%%%%%%%%%%%%%%%%%%%%%%%%%%%%%%%%%%%%%%%%%%%%%%%%%%%%%%%%%%%%%%%%%%%%%%%%%%%%%%%%%%%%%%%%%%%%%%%%%%%%%%%%%%%%%%%%%%%%%%%%%%%%%%%%%%%%%%%%%
Let us first show the implication $2 \Rightarrow 3$. For a fixed $\epsilon>0$ there exist $x_0,x_1,...,x_{n_\epsilon}\in C$ such that the subset $C$ is at Hausdorff distance less than $\epsilon$ from $\displaystyle{\bigcup_{i=1}^{n_\epsilon}[x_0,x_i]}$. Let $\hh\in \HH (X)$ be a halfspace transverse to $C$. If $\hh$ does not contain any of the $x_i$, then it must be of depth less than $\epsilon$ in $C$. Thus any halfspace transverse to $C$ of depth bigger than $\epsilon$ must separate $x_0$ from some $x_i$. Therefore there is only finitely many pairwise disjoint halfspaces transverse to $C$ and of depth bigger than $\epsilon$.
 %%%%%%%%%%%%%%%%%%%%%%%%%%%%%%%%%%%%%%%%%%%%%%%%%%%%%%%%%%%%%%%%%%%%%%%%%%%%%%%%%%%%%%%%%%%%%%%%%%%%%%%%%%%%%%%%%%%%%%%%%%%%%%%%%%%%%%%%%%%%%%%%%%%%%%%%%%%%%%%%%%%%%%%%%%%%%%%%%%
\par We now prove the implication $3\Rightarrow 1$. By Lemma \ref{lemma_hausdorf_limit_compact}, it is enough to show that under the conditions of statement $(3)$, the set $C$ is the Hausdorff limit of some sequence of compact subsets. Let us fix $\epsilon$ and consider a family $\HH_\epsilon$ of maximal cardinal of pairwise disjoint halfspaces transverse to $C$ and of depth bigger than $\epsilon$ in $C$. The maximality implies that each halfspace $\hh\in\HH_\epsilon$ does not contain two disjoint halfspaces of depth bigger than $\epsilon$ in $C\cap \hh$. As we are considering only halfspaces which are transverse to the convex subset $C$, we may forget about the ambient space $X$ and assume that each halfspace is a halfspace of $C$, up to taking the intersection with the latter (see Propositions \ref{halfspace_gateconvex} and \ref{induced_hyperplane}). By assumption Condition $(3)$ holds, hence this family of halfspaces is finite. Let us set $C_{\epsilon}=Conv(\displaystyle{\bigcup_{\hh\in \HH_\epsilon} \hat{\hh}} )$ and first show that it is at Hausdorff distance less than $n\epsilon$ from $C \ \backslash \displaystyle{\bigcup_{\hh\in\HH_\epsilon}\hh}$. Let $x\in C$ be a point lying outside all of the halfspace $\hh\in\HH_\epsilon$. Note then that any halfspace separating $x$ from $C_\epsilon$ is disjoint from any halfspace in $\HH_\epsilon$. Hence by the maximality of the family $\HH_\epsilon$, any halfspace separating $x$ from $C_\epsilon$ is of depth less than $\epsilon$ in $C$. Therefore, we conclude by Lemma \ref{lemma_local_convexity} that the point $x$ is at distance at most $n\epsilon$ from $C_\epsilon$.

For each $\hh\in \HH_\epsilon$, we choose a point $a_\hh\in \hh$ such that $d(a_\hh,\hat{\hh})\geq depth_C(\hh)-\epsilon$. We set $C_\hh:=Conv(\{a_\hh\}\cup \hat{\hh})$ and use Lemma \ref{depth_hyperplane} to conclude that it is at Hausdorff distance less than $(n(n+1)+1)\epsilon$ from $\hh$.\par 
Let us set $\tilde{C}_\epsilon=\displaystyle{\bigcup_{\hh\in\HH_\epsilon}C_{\hh}}\cup C_\epsilon$. We have shown that $\tilde{C}_\epsilon$ is at Hausdorff distance less than $(n(n+1)+1)\epsilon$ from $C$. It last to show that it is compact. We proceed by induction on the rank of $C$. Note that when the rank of $C$ is $1$, then the hyperplane correspond to a point. Hence by Proposition \ref{convex_hull_compact_convex}, the subset $\tilde{C}$ is compact as it is a finite union of the convex hull of compact subsets. Let us assume now that the rank of $C$ is equal $n$ and that the implication $3\Rightarrow 1$ is true for median space of rank less or equal $n-1$. Since we have assumed Condition $(3)$ to be true, it is verified by each hyperplane $\hat{\hh}$. Therefore each hyperplane $\hat{\hh}$ is compact. We conclude by Proposition \ref{convex_hull_compact_convex} that $\tilde{C}_\epsilon$ is compact, which finishes the proof. 
%%%%%%%%%%%%%%%%%%%%%%%%%%%%%%%%%%%%%%%%%%%%%%%%%%%%%%%%%%%%%%%%%%%%%%%%%%%%%%%%%%%%%%%%%%%%%%%%%%%%%%%%%%%%%%%%%%%%%%%%%%%%%%%%%%%%%%%%%%%%%%%%%%%%%%%%%%%%%%%%%%%%%%%%%%%%%%%%%%

\end{proof}
%%%%%%%%%%%%%%%%%%%%%%%%%%%%%%%%%%%%%%%%%%%%%%%%%%%%%%%%%%%%%%%%%%%%%%%%%%%%%%%%%%%%%%%%%%%%%%%%%%%%%%%%%%%%%%%%%%%%%%%%%%%%%%%%%%%%%%%%%%%%%%%%%%%%%%%%%%%%%%%%%%%%%%%%%%%%%%%%%%
For the general case when the rank is infinite, it is harder to manipulate halfspaces as they may all be dense in the space and even if it is the case, one can no longer use an argument by induction on the rank of the space. 
%%%%%%%%%%%%%%%%%%%%%%%%%%%%%%%%%%%%%%%%%%%%%%%%%%%%%%%%%%%%%%%%%%%%%%%%%%%%%%%%%%%%%%%%%%%%%%%%%%%%%%%%%
%%%%%%%%%%%%%%%%%%%%%%%%%%%%%%%%%%%%%%%%%%%%%%%%%%%%%%%%%%%%%%%%%%%%%%%%%%%%%%%%%%%%%%%%%%%%%%%%%%%%%%%%%%%%%%%%%%%%%%%%%%%%%%%%%%%%%%%%%%%%%%%%%%%%%%%%%%%%%%%%%%%%%%%%%%%%%%%%%%
Let us give a criterion of local compactness in the infinite rank case:
%%%%%%%%%%%%%%%%%%%%%%%%%%%%%%%%%%%%%%%%%%%%%%%%%%%%%%%%%%%%%%%%%%%%%%%%%%%%%%%%%%%%%%%%%%%%%%%%%%%%%%%%%%%%%%%%%%%%%%%%%%%%%%%%%%%%%%%%%%%%%%%%%%%%%%%%%%%%%%%%%%%%%%%%%%%%%%%%%%
\begin{prop}\label{local_compactness_criterion_infinite_rank_case}
Let $X$ be a complete median space with compact intervals and let $K\subseteq X$ be a closed subset. If the outer measure of the set of transverse halfspaces to $K$ is finite,
\\ $i.e\ \bar{\mu}(\HH (K))<+\infty$, where $\mu$ is the canonical measure associated to $\HH(X)$, then $K$ is compact.
\end{prop}
Let us first make the following remarks:
%%%%%%%%%%%%%%%%%%%%%%%%%%%%%%%%%%%%%%%%%%%%%%%%%%%%%%%%%%%%%%%%%%%%%%%%%%%%%%%%%%%%%%%%%%%%%%%%%%%%%%%%%%%%%%%%%%%%%%%%%%%%%%%%%%%%%%%%%%%%%%%%%%%%%%%%%%%%%%%%%%%%%%%%%%%%%%%%%%
\begin{rmk}\label{finite_measure_convex_hull}
\begin{itemize}
\item If $X$ be a complete median space and $K\subseteq X$ a subset such that $\bar{\mu}(\HH (K))<+\infty$, then the convex hull of $K$ is bounded. This is first due to the fact that the set of halfspaces which are transverse to $K$ is the same as the set of halfspaces which are transverse to the convex hull of $K$. In the other hand, having a sequence of points which is unbounded give rise to a sequence of wall interval with an arbitrarily big measure.
\item The converse of Proposition \ref{finite_measure_convex_hull} is false, even in the finite rank case. One may consider a star like tree obtained from the concatenation of the intervals $[0,\frac{1}{n}]$ at $\{0\}$.
\end{itemize}
\end{rmk} 
%%%%%%%%%%%%%%%%%%%%%%%%%%%%%%%%%%%%%%%%%%%%%%%%%%%%%%%%%%%%%%%%%%%%%%%%%%%%%%%%%%%%%%%%%%%%%%%%%%%%%%%%%%%%%%%%%%%%%%%%%%%%%%%%%%%%%%%%%%%%%%%%%%%%%%%%%%%%%%%%%%%%%%%%%%%%%%%%%%
\begin{proof}[Proof of Proposition \ref{local_compactness_criterion_infinite_rank_case}]
Let $K\subseteq X$ be a closed subset such that $\bar{\mu}(\HH (K))=M$. By Remark \ref{finite_measure_convex_hull}, there is no loss of generality if we consider the closure of the convex hull of $K$. Let us first remark that for any $x,y\in X$, the set of halfspaces which separates $x$ from $y$ is exactly the same as the set of halfspaces separating $\pi_{K}(x)$ from $\pi_{K}(y)$. Hence, for any $\epsilon >0$, there exist $x_1,y_1,...,x_{n_\epsilon},y_{n_\epsilon}\in K$ such that $\mathcal{W}(x_i,y_i)$ and $\mathcal{W}(x_j,y_j)$ are disjoint for any $i\neq j$ and :
\[
\mu(\displaystyle{\bigcup_{i=1}^{n_\epsilon} \mathcal{W} (x_i,y_i)})=\mu(\displaystyle{\sum_{i=1}^{n_\epsilon}d(x_i,y_i)})\geq M - \epsilon
\]
Let consider the convex hull of all the point $C_\epsilon=Conv(\{x_1,y_1,...,x_{n_\epsilon},y_{n_\epsilon}\})$ and a point $x\in K$. Then any halfspace separates $x$ from $C_\epsilon$ if and only if it separates $x$ from all the points $x_i$ and $y_i$. Thus, we must have $\mu(\mathcal{W} (x,C_\epsilon)) \leq \epsilon$. Let us then consider a sequence $C_{\frac{1}{n}}$ defined as above. We may assume that the sequence $(C_{\frac{1}{n}})_{n\in\NN}$ is ascending with respect to the inclusion. By Proposition \ref{convex_hull_finite_subset_compact}, each subset $C_\frac{1}{n}$ is compact. In the other hand, the sequence $(C_{\frac{1}{n}})_{n\in\NN}$ converges with respect to the Hausdorff metric to $C=\displaystyle{\bigcup_{n\in\NN}^i C_\frac{1}{n}}$. As each $C_\frac{1}{n}$ is compact, the subset $C$ is totally bounded. Thus its closure is a compact which contains the closed subset $K$. We conclude that the subset $K$ is also compact.
\end{proof}
%%%%%%%%%%%%%%%%%%%%%%%%%%%%%%%%%%%%%%%%%%%%%%%%%%%%%%%%%%%%%%%%%%%%%%%%%%%%%%%%%%%%%%%%%%%%%%%%%%%%%%%%%%%%%%%%%%%%%%%%%%%%%%%%%%%%%%%%%%%%%%%%%%%%%%%%%%%%%%%%%%%%%%%%%%%%%%%%%%%%%%%%%%%%%%%%%%%%%%%%%%%%%%%%%%%%%%%%%%%%%%%%%%%%%%%%%%%%%%%%%%%%%%%%%%%%%%%%%%%%%%%%%%%%%%%%%%%%%%%%%%%%%%%%%%%%%%%%%%%%%%%%%%%%%%%%%%%%%%%%%%%%%%%%%%%%%%%%%%%%%%%%%%%%%%%%%%%%
\section{Transitive actions on median spaces of finite rank and local compactness}%%%%%%%%%%%%%%%%%%%%%%%%%%%%%%%%%%%%%%%%%%%%%%%%%%%%%%%%%%%%%%%%%%%%%%%%%%%%%%%%%%%%%%%%%%%%%%%%%%%%%%%%%%%%%%%%%%%%%%%%%%%%%%%%%%%%%%%%%%%%%%%%%%%%%%%%%%%%%%%%%%%%%%%%%%%%%%%%%%%%%%%%%%%%%%%%%%%%%%%%%%%%%%%%%%%%%%%%%%%%%%%%%%%%%%%%%%%%%%%%%%%%%%%%%%%%%%%%%%%%%%%%%%%%%%%%%%%%%%%%%%%%%%%%%%%%%%%%%%%%%%%%%%%%%%%%%%%%%%%%%%%%%%%%%%%%%%%%%%%%%%%%%%%%%%%%%%%%
The strategy to prove Theorem \ref{rigidity} consists of considering the set of halfspaces which are ``branched" at an arbitrary point.  The set $\HH_x(X)$ of halfspaces branched at a point $x\in X$ in a complete median space of finite rank is the set of halfspaces $\hh \subset X$ such that $x\in \bar{\hh}\cap \bar{\hh^c}$. The set $\HH_x(X)$ can be seen as the extension of the notion of the valency at a point $x$ in an $\RR$-tree to the case of complete median space of finite rank.
 Theorem \ref{rigidity} is obtained then as a consequence of the following results:

\begin{thm}\label{thm_principal}
Let $X$ be a complete connected median space of rank $n$ which admits a transitive action. If for some or equivalently any $x\in X$:
\begin{enumerate}
\item The set $\HH_x$ contains no triple of pairwise disjoint halfspaces then the space $X$ is isomorphic to $(\RR^n,l^1)$.
\item The set $\HH_x$ contains three halfspaces which are pairwise disjoint then the space $X$ is not locally compact.
\end{enumerate}
\end{thm}
The second part of the theorem above is the easiest to prove as the image of the facing triple will generate infinitely many pairwise disjoint halfspaces inside a bounded neighborhood. It is the first part of the theorem that demands a little bit of technicality. Under its assumption, we show that the space is locally modelled on $(\RR^n,\ell^1)$.
\par In order to do so, we first look, in Subsection \ref{subsection_trace} below, at the trace of the halfspaces on the closed convex $n$-cubes of the spaces, which are isometrically embedded convex pieces of $(\RR^n,\ell^1)$. We see these pieces as the ``branching locus" for our space.

%%%%%%%%%%%%%%%%%%%%%%%%%%%%%%%%%%%%%%%%%%%%%%%%%%%%%%%%%%%%%%%%%%%%%%%%%%%%%%%%%%%%%%%%%%%%%%%%%%%%%%%%%%%%%%%%%%%%%%%%%%%%%%%%%%%%%%%%%%%%%%%%%%%%%%%%%%%%%%%%%%%%%%%%%%%%%%%%%%%%%%%%%%%%%%%%%%%%%%%%%%%%%%%%%%%%%%%%%%%%%%%%%%%%%%%%%%%%%%%%%%%%%%%%%%%%%%%%%%%%%%%%%%%%%%%%%%%%%%%%%%%%%%%%%%%%%%%%%%%%%%%%%%%%%%%%%%%%%%%%%%%%%%%%%%%%%%%%%%%%%%%%%%%%%%%%%%%%
\subsection{Trace of halfspaces on convex sets}\label{subsection_trace}
%%%%%%%%%%%%%%%%%%%%%%%%%%%%%%%%%%%%%%%%%%%%%%%%%%%%%%%%%%%%%%%%%%%%%%%%%%%%%%%%%%%%%%%%%%%%%%%%%%%%%%%%%%%%%%%%%%%%%%%%%%%%%%%%%%%%%%%%%%%%%%%%%%%%%%%%%%%%%%%%%%%%%%%%%%%%%%%%%%%%%%%%%%%%%%%%%%%%%%%%%%%%%%%%%%%%%%%%%%%%%%%%%%%%%%%%%%%%%%%%%%%%%%%%%%%%%%%%%%%%%%%%%%%%%%%%%%%%%%%%%%%%%%%%%%%%%%%%%%%%%%%%%%%%%%%%%%%%%%%%%%%%%%%%%%%%%%%%%%%%%%%%%%%%%%%%%%%%
Throughout this section $X$ is a complete connected median space of finite rank.
\par Let us denote the set of hyperplanes of $X$ by $\hat{\HH}(X)$ and deduce from Proposition \ref{halfspace_gateconvex} that any hyperplane in a closed convex subset is induced from a hyperplane of the ambient space as stated in the following:
\begin{prop}\label{induced_hyperplane}
Let $C\subseteq X$ be a closed convex subset. We have then:
\[
\hat{\HH} (C)=\{\hat{\hh}\cap C \ |\ \hat{\hh}\in\hat{\HH} (X) \ \text{and separates two points of} \ C\}
\] 
\end{prop}
%%%%%%%%%%%%%%%%%%%%%%%%%%%%%%%%%%%%%%%%%%%%%%%%%%%%%%%%%%%%%%%%%%%%%%%%%%%%%%%%%%%%%%%%%%%%%%%%%%%%%%%%%%%%%%%%%%%%%%%%%%%%%%%%%%%%%%%%%%%%%%%%%%%%%%%%%%%%%%%%%%%%%%%%%%%%%%%%%%
\begin{proof}
By Proposition \ref{halfspace_gateconvex}, it is enough to show that for any halfspace $\hh$ which is transverse to $C$ we have $\overline{\hh\cap C}=\bar{\hh}\cap C$. The inclusion $\overline{\hh\cap C}\subseteq\bar{\hh}\cap C$ is obvious. For the converse, let us take a sequence $(x_n)_{n\in\NN}\subseteq\hh$ which converges to a point $x\in C$. We denote by $y_n:=\pi_{C}(x_n)$ their projection into the closed convex set $C$. As $\hh\cap C$ is not empty, we get that $(y_n)_{n\in\NN}\subseteq\hh\cap C$. Due to the continuity of the projection, the sequence $(y_n)_{n\in\NN}\subseteq\hh$ converges to $\pi_{C}(x)=x$. 
\end{proof}
%%%%%%%%%%%%%%%%%%%%%%%%%%%%%%%%%%%%%%%%%%%%%%%%%%%%%%%%%%%%%%%%%%%%%%%%%%%%%%%%%%%%%%%%%%%%%%%%%%%%%%%%%%%%%%%%%%%%%%%%%%%%%%%%%%%%%%%%%%%%%%%%%%%%%%%%%%%%%%%%%%%%%%%%%%%%%%%%%%
Let $C\subseteq X$ be a convex subset and $\hh\subseteq X$ a halfspace. We call the \textbf{\textit{trace}} of the hyperplane $\hat{\hh}$ on $C$ the intersection $\hat{\hh}\cap C$. 
\begin{prop}\label{lift_and_transversality}
Let $C\subseteq X$ be a closed convex subset and $\hh\in\HH(X)$ be a halfspace. Then the lift of any halfspace $T$ of $\hat{\hh}\cap C$ to the ambient space $X$ is a halfspace transverse to $\hh$ assuming that $T$ and its complement are of non empty interior inside $\hat{\hh}\cap C$. 
\end{prop}
%%%%%%%%%%%%%%%%%%%%%%%%%%%%%%%%%%%%%%%%%%%%%%%%%%%%%%%%%%%%%%%%%%%%%%%%%%%%%%%%%%%%%%%%%%%%%%%%%%%%%%%%%%%%%%%%%%%%%%%%%%%%%%%%%%%%%%%%%%%%%%%%%%%%%%%%%%%%%%%%%%%%%%%%%%%%%%%%%%
\begin{proof}
Let $T\in\HH(\hat{\hh}\cap C)$ be an halfpsace of the trace of the hyperplane $\hat{\hh}$ on the convex subset $C$. To show that the lift of $T$ to $X$ is transverse to $\hh$, it is enough to show that there exist points in $\hh$ and $\hh^c$ which projects into $T$ and another ones which projects into $T^c\cap C\cap \hat{\hh}$. The halfspace $T$ being inside the trace of the hyperplane $\hat{\hh}$, for any point $x$ in the interior of $T$ or its complement inside $\hat{\hh}\cap C$, there exists a sequence $(x_n)_{n\in\NN}$ inside $\hh\cap C$ which converge to $x$. Taking the index $n$ big enough, the point $x_n$ projects inside a small neighborhood of $x$ in $T$.
\end{proof}
%%%%%%%%%%%%%%%%%%%%%%%%%%%%%%%%%%%%%%%%%%%%%%%%%%%%%%%%%%%%%%%%%%%%%%%%%%%%%%%%%%%%%%%%%%%%%%%%%%%%%%%%%%%%%%%%%%%%%%%%%%%%%%%%%%%%%%%%%%%%%%%%%%%%%%%%%%%%%%%%%%%%%%%%%%%%%%%%%%
\begin{rmk}
The proposition will no longer be true if we drop the assumption on $T$ and its complement being both of non empty interior. One may consider for instance the following subspace of $(\RR^2,\ell^1)$:
\[
X=\{(x,y)\in \RR^2 \ |\  x\leq 0\}\cup \{(x,y)\in \RR^2 \ |\  y-x\geq 0\}
\]
We take the convex subset $C$ to be the half line $x=0$ and $T$, a halfspace of $C$, to be the trace of the halfspace $\hh:=\{(x,y)\in X\ | \ y>0\}$. The convex subset $C$ is also the hyperplane which bounds the halfspace defined by the inequality $x>0$. The lift of $(0,0)$, the complement of $T$ inside $C$, is not transverse to $\hh$ as it is disjoint from it.
\end{rmk}
%%%%%%%%%%%%%%%%%%%%%%%%%%%%%%%%%%%%%%%%%%%%%%%%%%%%%%%%%%%%%%%%%%%%%%%%%%%%%%%%%%%%%%%%%%%%%%%%%%%%%%%%%%%%%%%%%%%%%%%%%%%%%%%%%%%%%%%%%%%%%%%%%%%%%%%%%%%%%%%%%%%%%%%%%%%%%%%%%%
\begin{prop}\label{trace_and_facing_triple}
Let $C\subseteq X$ be a convex subset which isometric to an $n$-cube $([-\epsilon,\epsilon],\ell^1)$ where $n$ is the rank of X. Let  $\hh\in \HH(X)$ be a halfspace in $X$ which is disjoint from $C$. If its hyperplane $\hat{\hh}$ intersects the interior of $C$, then there exist $\hh_1,\hh_2\in \HH(X)$ such that they constitute with the halfspace $\hh$ a facing triple in $X$ and $\hat{\hh}\cap \hat{\hh_1}\cap \hat{\hh_2}\neq \emptyset$.
\end{prop}
%%%%%%%%%%%%%%%%%%%%%%%%%%%%%%%%%%%%%%%%%%%%%%%%%%%%%%%%%%%%%%%%%%%%%%%%%%%%%%%%%%%%%%%%%%%%%%%%%%%%%%%%%%%%%%%%%%%%%%%%%%%%%%%%%%%%%%%%%%%%%%%%%%%%%%%%%%%%%%%%%%%%%%%%%%%%%%%%%%
\begin{proof}
Let us identify $C$ with $[-\epsilon,\epsilon]^n$. By Proposition \ref{lift_and_transversality}, the rank of $\hat{\hh}\cap C$ is smaller than $n$. Hence the latter is contained in a hyperplane of $C$, let us say the hyperplane given by the equation $x_1=0$. Then the lift to the ambient space of the two halfspaces $\mathfrak{ H_1}=\{(x_1,...,x_n)\in C \ |\  x_1>0 \}$ and $\mathfrak{ H_2}=\{(x_1,...,x_n)\in C \ |\  x_1<0 \}$ give us the desired halfspaces. Indeed, Lemma \ref{constraint_on_the_projection} tell us that $\overline{\hh}$ projects into $\overline{\hh}\cap C=\hat{\hh}\cap C$, which is outside $\mathfrak{ H_1}$ and $\mathfrak{  H_2}$.
\end{proof}
%%%%%%%%%%%%%%%%%%%%%%%%%%%%%%%%%%%%%%%%%%%%%%%%%%%%%%%%%%%%%%%%%%%%%%%%%%%%%%%%%%%%%%%%%%%%%%%%%%%%%%%%%%%%%%%%%%%%%%%%%%%%%%%%%%%%%%%%%%%%%%%%%%%%%%%%%%%%%%%%%%%%%%%%%%%%%%%%%%
Finally we deduce the following corollary in the case where the median space admits a transitive action:
%%%%%%%%%%%%%%%%%%%%%%%%%%%%%%%%%%%%%%%%%%%%%%%%%%%%%%%%%%%%%%%%%%%%%%%%%%%%%%%%%%%%%%%%%%%%%%%%%%%%%%%%%%%%%%%%%%%%%%%%%%%%%%%%%%%%%%%%%%%%%%%%%%%%%%%%%%%%%%%%%%%%%%%%%%%%%%%%%%
\begin{cor}\label{no_facing_triple}
Let $X$ be a complete connected median space of rank $n$ and let $([-\epsilon,\epsilon]^n,\ell^1)\cong C\subseteq X$ be an isometrically embedded $n$-cube centred at a point $x\in X$. If $\HH_{x}(X)$ does not contain a facing triple, then it coincides with $\HH_x(C)$.
\end{cor}
%%%%%%%%%%%%%%%%%%%%%%%%%%%%%%%%%%%%%%%%%%%%%%%%%%%%%%%%%%%%%%%%%%%%%%%%%%%%%%%%%%%%%%%%%%%%%%%%%%%%%%%%%%%%%%%%%%%%%%%%%%%%%%%%%%%%%%%%%%%%%%%%%%%%%%%%%%%%%%%%%%%%%%%%%%%%%%%%%%
\begin{proof}
By Proposition \ref{halfspace_gateconvex}, it is enough to show that any halfspace in $\HH_x(X)$ is transverse to the $n$-cube $C$. Let us consider a halfspace $\hh\in \HH_x(X)$ branched at $x$. The set $\HH_x(X)$ does not contain a facing triple and the intersection $\hat{\hh}\cap C$ is not empty as both contain the point $x$, hence Proposition \ref{trace_and_facing_triple} implies then that the halfspace $\hh$ is necessarily transverse to $C$. 
\end{proof}
%%%%%%%%%%%%%%%%%%%%%%%%%%%%%%%%%%%%%%%%%%%%%%%%%%%%%%%%%%%%%%%%%%%%%%%%%%%%%%%%%%%%%%%%%%%%%%%%%%%%%%%%%%%%%%%%%%%%%%%%%%%%%%%%%%%%%%%%%%%%%%%%%%%%%%%%%%%%%%%%%%%%%%%%%%%%%%%%%%%%%%%%%%%%%%%%%%%%%%%%%%%%%%%%%%%%%%%%%%%%%%%%%%%%%%%%%%%%%%%%%%%%%%%%%%%%%%%%%%%%%%%%%%%%%%%%%%%%%%%%%%%%%%%%%%%%%%%%%%%%%%%%%%%%%%%%%%%%%%%%%%%%%%%%%%%%%%%%%%%%%%%%%%%%%%%%%%%%
\subsection{Proof of Theorem \ref{thm_principal}}\label{section_theorem_principal}%%%%%%%%%%%%%%%%%%%%%%%%%%%%%%%%%%%%%%%%%%%%%%%%%%%%%%%%%%%%%%%%%%%%%%%%%%%%%%%%%%%%%%%%%%%%%%%%%%%%%%%%%%%%%%%%%%%%%%%%%%%%%%%%%%%%%%%%%%%%%%%%%%%%%%%%%%%%%%%%%%%%%%%%%%%%%%%%%%%%%%%%%%%%%%%%%%%%%%%%%%%%%%%%%%%%%%%%%%%%%%%%%%%%%%%%%%%%%%%%%%%%%%%%%%%%%%%%%%%%%%%%%%%%%%%%%%%%%%%%%%%%%%%%%%%%%%%%%%%%%%%%%%%%%%%%%%%%%%%%%%%%%%%%%%%%%%%%%%%%%%%%%%%%%%%%%%%%
\paragraph{Proof of the first claim of Theorem \ref{thm_principal} :}
Let $X$ be a complete connected median space of finite rank which admits a transitive action.
Let us fix a maximal pairwise transverse family of halfspaces $\HH=\{\hh_1,...,\hh_n\}$ in $\HH(X)$. Let us set the following
\begin{equation}\label{equality_lines}
D_i:=\bigcap_{j\neq i}\hat{\hh}_j
\end{equation}
 Let us show that each $D_i$ is a strongly convex isometric embedding of an $\RR$-tree.
%%%%%%%%%%%%%%%%%%%%%%%%%%%%%%%%%%%%%%%%%%%%%%%%%%%%%%%%%%%%%%%%%%%%%%%%%%%%%%%%%%%%%%%%%%%%%%%%%%%%%%%%%%%%%%%%%%%%%%%%%%%%%%%%%%%%%%%%%%%%%%%%%%%%%%%%%%%%%%%%%%%%%%%%%%%%%%%%%%
\begin{prop}\label{Di_rank_1}
Each $D_i$, endowed with the induced metric of $X$, is a complete connected median space of rank $1$.
\end{prop}
%%%%%%%%%%%%%%%%%%%%%%%%%%%%%%%%%%%%%%%%%%%%%%%%%%%%%%%%%%%%%%%%%%%%%%%%%%%%%%%%%%%%%%%%%%%%%%%%%%%%%%%%%%%%%%%%%%%%%%%%%%%%%%%%%%%%%%%%%%%%%%%%%%%%%%%%%%%%%%%%%%%%%%%%%%%%%%%%%%
Let us first show the following lemma:
%%%%%%%%%%%%%%%%%%%%%%%%%%%%%%%%%%%%%%%%%%%%%%%%%%%%%%%%%%%%%%%%%%%%%%%%%%%%%%%%%%%%%%%%%%%%%%%%%%%%%%%%%%%%%%%%%%%%%%%%%%%%%%%%%%%%%%%%%%%%%%%%%%%%%%%%%%%%%%%%%%%%%%%%%%%%%%%%%%
\begin{lem}\label{halfspace_non_empty_interior}
Let $X$ be a complete connected median space of finite rank. We assume that there exist two transverse halfspaces $\hh_1,\hh_2\subseteq X$. Then there exist two transverse halfspaces such that both them and their complements are of non empty interior inside $X$. 
\end{lem}
%%%%%%%%%%%%%%%%%%%%%%%%%%%%%%%%%%%%%%%%%%%%%%%%%%%%%%%%%%%%%%%%%%%%%%%%%%%%%%%%%%%%%%%%%%%%%%%%%%%%%%%%%%%%%%%%%%%%%%%%%%%%%%%%%%%%%%%%%%%%%%%%%%%%%%%%%%%%%%%%%%%%%%%%%%%%%%%%%%
\begin{proof}
Let us consider two transverse halfspaces $\hh_1,\hh_2\subseteq X$. By Remark \ref{closed_branched_halfspaces} and up to considering the complement, we may assume that they are both open in $X$. Let us consider a point $x\in \hh_1\cap \hh_2$. Let us set the following $x_i:=\pi_{\hh_i^c}(x)$ and $x_0:=\pi_{\hh_1^c\cap\hh_2^c}(x)$. Note that $x_0=\pi_{\hh_1^c}(x_2)=\pi_{\hh_2^c}(x_1)$ by Lemma \ref{constraint_on_the_projection}. The halfspaces $\hh_1$ and $\hh_2$ being both open, the points $x_1$ and $x_2$ are distinct from $x_0$. Let us set $C:=Conv([x_0,x_1]\cup[x_0,x_2])$ and set $\tilde{x}:=\pi_C(x)$. Note that $\pi_{\hh_i^c}(\tilde{x})=x_i$. As $[x_0,x_1]\cap[x_0,x_2]=\{x_0\}$, Proposition \ref{ncube} implies that the interval $[x,x_0]$ is isometric to the $\ell^1$-product $[x,x_1]\times[x,x_2]$. Hence, the lift to $X$ of any halfspaces $\mathfrak{H}_1\in \HH([x,x_1]$ and $\mathfrak{H}_2\in \HH([x,x_2]$, such that $\mathfrak{H}_i$ and $\mathfrak{H}_i^c$ are of non empty interior in $[x,x_i]$, yields two transverse halfspaces such that both them and their complements are of non empty interior in $X$. 
\end{proof}
%%%%%%%%%%%%%%%%%%%%%%%%%%%%%%%%%%%%%%%%%%%%%%%%%%%%%%%%%%%%%%%%%%%%%%%%%%%%%%%%%%%%%%%%%%%%%%%%%%%%%%%%%%%%%%%%%%%%%%%%%%%%%%%%%%%%%%%%%%%%%%%%%%%%%%%%%%%%%%%%%%%%%%%%%%%%%%%%%%
\begin{proof}[Proof of Proposition \ref{Di_rank_1}]
By Helly's Theorem \ref{helly_theorem}, each $D_i$ is a non empty closed convex subset of $X$ and it intersects both $\hh_i$ and $\hh_i^c$. Hence, each $D_i$ is of rank bigger than or equal one. It is left to show that it is of rank smaller then two. If there exist two  transverse halfspaces in $D_i$, Lemma \ref{halfspace_non_empty_interior} ensures that there is no loss of generality if we assume them to be with their complement inside $D_i$ of non empty interior. In the other hand, Proposition \ref{lift_and_transversality} implies that the lift of such halfspaces to the ambient space $X$ yields halfspaces which are transverse to each $\hh_j$ where $j\neq i$, which would contradict the maximality of the family $(\hh_i)$. %For a justification, see the justification tex file
\end{proof}
%%%%%%%%%%%%%%%%%%%%%%%%%%%%%%%%%%%%%%%%%%%%%%%%%%%%%%%%%%%%%%%%%%%%%%%%%%%%%%%%%%%%%%%%%%%%%%%%%%%%%%%%%%%%%%%%%%%%%%%%%%%%%%%%%%%%%%%%%%%%%%%%%%%%%%%%%%%%%%%%%%%%%%%%%%%%%%%%%%
Following the same argument of the proof of Proposition \ref{Di_rank_1}, we note that for $i\neq j$, we have $D_i\cap D_j=\displaystyle{\bigcap_{k=1}^n \hat{\hh}_k}=\{a_0\}$ for some $a_0\in X$.\par
In the following, we show that any point is the center of an isometrically embedded $n$-cube.
%%%%%%%%%%%%%%%%%%%%%%%%%%%%%%%%%%%%%%%%%%%%%%%%%%%%%%%%%%%%%%%%%%%%%%%%%%%%%%%%%%%%%%%%%%%%%%%%%%%%%%%%%%%%%%%%%%%%%%%%%%%%%%%%%%%%%%%%%%%%%%%%%%%%%%%%%%%%%%%%%%%%%%%%%%%%%%%%%%
\begin{lem}\label{embedded_cube}
Let $X$ be a complete connected median space of rank $n$. There exist then a point $x\in X$ and $\epsilon>0$ such that $x\in X$ is the center of an isometrically embedded convex $n$-cube $([-\epsilon,\epsilon]^n,\ell^1)$. In particular, if $X$ admits a transitive action, any point $x\in X$ lies in the center of some embedded convex $n$-cube.
\end{lem}
%%%%%%%%%%%%%%%%%%%%%%%%%%%%%%%%%%%%%%%%%%%%%%%%%%%%%%%%%%%%%%%%%%%%%%%%%%%%%%%%%%%%%%%%%%%%%%%%%%%%%%%%%%%%%%%%%%%%%%%%%%%%%%%%%%%%%%%%%%%%%%%%%%%%%%%%%%%%%%%%%%%%%%%%%%%%%%%%%%
\begin{proof}
By the transitivity assumption of the isometry group of $X$, it is enough to show the existence of an isometrically embedded $n$-cube in $X$. By Helly's Theorem \ref{helly_theorem}, the intersection $\displaystyle{\bigcap_{i=1}^n\hh_i}$ is not empty. Let us consider a point $a$ in the latter intersection. Again by Helly's Theorem, the intersections $\hh_i\cap D_i$ are not empty and do not contain $a_0$, hence the projection of $a$ into $D_i$ avoid $a_0$. Let us set $a_i:=\pi_{D_i}(a)$. By Proposition \ref{Di_rank_1}, each interval $[a_0,a_i]$ is isometric to closed interval of $\RR$. Let us denote by $C$ the convex hull between the $\RR$-trees $D_i$. By Lemma \ref{ncube}, the interval $[a_0,\pi_C(a)]$ is isometric to the $\ell^1$-product of the intervals $[a_0,a_i]$, thus completing the argument.
\end{proof}
%%%%%%%%%%%%%%%%%%%%%%%%%%%%%%%%%%%%%%%%%%%%%%%%%%%%%%%%%%%%%%%%%%%%%%%%%%%%%%%%%%%%%%%%%%%%%%%%%%%%%%%%%%%%%%%%%%%%%%%%%%%%%%%%%%%%%%%%%%%%%%%%%%%%%%%%%%%%%%%%%%%%%%%%%%%%%%%%%%
Let us assume now that the set $\HH_{a_0}(X)$ does not contain a pairwise disjoint triple of halfspaces. Remark that the transitivity assumption implies then that for any $x\in X$ the set $\HH_x(X)$ does not contain such triple.
%%%%%%%%%%%%%%%%%%%%%%%%%%%%%%%%%%%%%%%%%%%%%%%%%%%%%%%%%%%%%%%%%%%%%%%%%%%%%%%%%%%%%%%%%%%%%%%%%%%%%%%%%%%%%%%%%%%%%%%%%%%%%%%%%%%%%%%%%%%%%%%%%%%%%%%%%%%%%%%%%%%%%%%%%%%%%%%%%%
\begin{prop}
Each $D_i$, as defined in the beginning of Subsction \ref{section_theorem_principal} Equality (\ref{equality_lines}), is isometric to the real line.
\end{prop}
%%%%%%%%%%%%%%%%%%%%%%%%%%%%%%%%%%%%%%%%%%%%%%%%%%%%%%%%%%%%%%%%%%%%%%%%%%%%%%%%%%%%%%%%%%%%%%%%%%%%%%%%%%%%%%%%%%%%%%%%%%%%%%%%%%%%%%%%%%%%%%%%%%%%%%%%%%%%%%%%%%%%%%%%%%%%%%%%%%
\begin{proof}
By Proposition \ref{Di_rank_1}, we already know that each $D_i$ is a closed convex subset of rank $1$, hence it is isometric to a complete connected $\RR$-tree. Under the assumption that the set $\HH_{a_0}(X)$ does not contain any facing triple, Proposition \ref{halfspace_gateconvex} and Helly's Theorem \ref{helly_theorem} imply that no $D_i$ can contain a facing triple. Hence, each $D_i$ is isometric to an interval of the real line. To conclude that it is isometric to the real line, it is enough to show that any point in $D_i$ lies in the interior of an interval in $D_i$. Let us consider a point $x\in D_i$. By Lemma \ref{embedded_cube}, there exists an embedded $n$-cube $C\cong ]-\epsilon,\epsilon[^n$ in $X$ centred at the point $x$. Under the assumption that there is no facing triple in $\HH_{x}(X)$, Corollary \ref{no_facing_triple} implies that the halfspace $\hh_j$, for any $j\neq i$, are transverse to the $n$-cube. We have then $\displaystyle{(\bigcap_{j\neq i}\hat{\hh}_j)\cap C}\cong ]-\epsilon,\epsilon[ \subseteq D_i$. As $D_i$ identifies with an interval of $\RR$, the latter intersection intersection is open in $D_i$, which finish the argument.
\end{proof}

 Now, we have all the ingredients needed to prove the first part of Theorem \ref{thm_principal}:
 %%%%%%%%%%%%%%%%%%%%%%%%%%%%%%%%%%%%%%%%%%%%%%%%%%%%%%%%%%%%%%%%%%%%%%%%%%%%%%%%%%%%%%%%%%%%%%%%%%%%%%%%%%%%%%%%%%%%%%%%%%%%%%%%%%%%%%%%%%%%%%%%%%%%%%%%%%%%%%%%%%%%%%%%%%%%%%%%%%
\begin{prop}\label{open_convex_hull}
The convex hull of the lines $D_i$ contains $X$ and is isometric to $(\RR^n,\ell^1)$. 
\end{prop}
%%%%%%%%%%%%%%%%%%%%%%%%%%%%%%%%%%%%%%%%%%%%%%%%%%%%%%%%%%%%%%%%%%%%%%%%%%%%%%%%%%%%%%%%%%%%%%%%%%%%%%%%%%%%%%%%%%%%%%%%%%%%%%%%%%%%%%%%%%%%%%%%%%%%%%%%%%%%%%%%%%%%%%%%%%%%%%%%%%
\begin{proof}
We set $C:=Conv(D_1\cup ...\cup D_n)$. The set $C$ is a closed subset of the complete median space $X$, as it is the convex hull of finitely many closed convex subsets, hence it is also complete. By Proposition \ref{embed_lemm}, it embeds isometrically, through the projections onto the $D_i$'s, as a closed subset of $\displaystyle{\prod_{i=1}^n D_i \cong \RR^n}$. It is enough to show that the embedding is open to conclude that it is surjective. In our way proving that, we prove also that $C$ contains $X$. 
Let us take a point $x\in X$ and consider the family $\hh_{i,l}:=\pi_{D_i}^{-1}(]-\infty,\pi_{D_i}(\widetilde{x})[),\hh_{i,r}:=\pi_{D_i}^{-1}([\pi_{D_i}(\widetilde{x}),+\infty[)$, where $\widetilde{x}:=\pi_C(x)$ and each $D_i$ is identified with $\RR$. By Proposition \ref{halfspace_gateconvex}, each $\hh_{i,l}$ and $\hh_{i,r}$ is a halfspace of $X$. By Lemma \ref{embedded_cube}, there exists an $n$-cube centered at $\tilde{x}$. Thus by Corollary \ref{no_facing_triple}, the family of halfspaces $\{\hh_{1,l},\hh_{1,r},...,\hh_{n,l},\hh_{n,r}\} $ and their complements in $X$ constitutes all the elements of $\HH_{\tilde{x}}$. In one hand, this implies that the projection map $(\pi_{D_1},...,\pi_{D_n})$ is open. In the other hand, by Lemma \ref{strenghtening_separation_theorem} we get :
\[
\displaystyle{\bigcap_{i=1}^n (\hh_{i,l}\cap\hh_{i,r})}\ =\ \{\widetilde{x}\}
\]
As the projections onto each $D_i$ factor through the projection onto $C$, that is $\pi_{D_i}(x)=\pi_{D_i}(\pi_C(x))=\pi_{D_i}(\tilde{x})$, the point $x$ lie in $(\hh_{i,l}\cap\hh_{i,r})$ for any $i\in\{1,...,n\}$. Hence, we get $x=\tilde{x}$, which proves that $C=X$ and complete the proof.
\end{proof}

%%%%%%%%%%%%%%%%%%%%%%%%%%%%%%%%%%%%%%%%%%%%%%%%%%%%%%%%%%%%%%%%%%%%%%%%%%%%%%%%%%%%%%%%%%%%%%%%%%%%%%%%%%%%%%%%%%%%%%%%%%%%%%%%%%%%%%%%%%%%%%%%%%%%%%%%%%%%%%%%%%%%%%%%%%%%%%%%%%%%%%%%%%%%%%%%%%%%%%%%%%%%%%%%%%%%%%%%%%%%%%%%%%%%%%%%%%%%%%%%%%%%%%%%%%%%%%%%%%%%%%%%%%%%%%%%%%%%%%%%%%%%%%%%%%%%%%%%%%%%%%%%%%%%%%%%%%%%%%%%%%%%%%%%%%%%%%%%%%%%%%%%%%%%%%%%%%%%
\paragraph{Proof of the second part of Theorem \ref{thm_principal}}
%%%%%%%%%%%%%%%%%%%%%%%%%%%%%%%%%%%%%%%%%%%%%%%%%%%%%%%%%%%%%%%%%%%%%%%%%%%%%%%%%%%%%%%%%%%%%%%%%%%%%%%%%%%%%%%%%%%%%%%%%%%%%%%%%%%%%%%%%%%%%%%%%%%%%%%%%%%%%%%%%%%%%%%%%%%%%%%%%%%%%%%%%%%%%%%%%%%%%%%%%%%%%%%%%%%%%%%%%%%%%%%%%%%%%%%%%%%%%%%%%%%%%%%%%%%%%%%%%%%%%%%%%%%%%%%%%%%%%%%%%%%%%%%%%%%%%%%%%%%%%%%%%%%%%%%%%%%%%%%%%%%%%%%%%%%%%%%%%%%%%%%%%%%%%%%%%%%%
The idea of the proof is to show that under the assumption of the existence of a facing triple in $\HH_{a_0}$ and a transitive action on $X$, there exist infinitely many pairwise disjoint halfspaces with depth uniformly bounded below inside any ball centred at $a_0$.
Let us first show that any halfpaces in $\HH_{a_0}$ is of positive depth inside any ball centred at $a_0$.
%%%%%%%%%%%%%%%%%%%%%%%%%%%%%%%%%%%%%%%%%%%%%%%%%%%%%%%%%%%%%%%%%%%%%%%%%%%%%%%%%%%%%%%%%%%%%%%%%%%%%%%%%%%%%%%%%%%%%%%%%%%%%%%%%%%%%%%%%%%%%%%%%%%%%%%%%%%%%%%%%%%%%%%%%%%%%%%%%%
\begin{prop}\label{depth_transitive_assumption}
Let $X$ be a complete connected median space of finite rank which admits a transitive action. Then for any halfspace $\hh\in\HH_{a_0}$ and $r>0$, we have $depth_{B(a_0,r)}(\hh)>0$.
\end{prop}
%%%%%%%%%%%%%%%%%%%%%%%%%%%%%%%%%%%%%%%%%%%%%%%%%%%%%%%%%%%%%%%%%%%%%%%%%%%%%%%%%%%%%%%%%%%%%%%%%%%%%%%%%%%%%%%%%%%%%%%%%%%%%%%%%%%%%%%%%%%%%%%%%%%%%%%%%%%%%%%%%%%%%%%%%%%%%%%%%%
\begin{proof}
Let $\hh\in\HH_{a_0}$ be a halfspace containing the point $a_0$ in its hyperplane. If the halfspace $\hh$ is open, its depth inside any ball centered at $a_0$ is positive. Let us assume then that the halfspace $\hh$ is closed. By Proposition \ref{ncube}, there exists an isometrically embedded $n$-cube $C\cong[-\epsilon,\epsilon]^n$ where $n$ is the rank of the space $X$. The action being transitive, we can assume that the $n$-cube is centred at $a_0$. If the halfspace $\hh$ is transverse to the $n$-cube $C$, then it gives rise to a halfspace of $C$ which contains $a_0$ and which is of positive depth inside $C$. If $\hh$ is not transverse to $C$, then it will contain it. Let us consider then the trace of the hyperplane $\hat{\hh}$ which bounds the halfspace $\hh$, on the $n$-cube, that is, its intersection with the latter. Let us denote it by $\hat{C}$. It is a convex subset which contains $a_0$. The rank of $X$ being $n$, Proposition \ref{lift_and_transversality} implies that the convex subset $\hat{C}$ is of rank less than $n-1$. Again, there exist then points inside $C$ which are at positive distance from $\hat{C}$, hence from $\hat{\hh}$. 
\end{proof}
%%%%%%%%%%%%%%%%%%%%%%%%%%%%%%%%%%%%%%%%%%%%%%%%%%%%%%%%%%%%%%%%%%%%%%%%%%%%%%%%%%%%%%%%%%%%%%%%%%%%%%%%%%%%%%%%%%%%%%%%%%%%%%%%%%%%%%%%%%%%%%%%%%%%%%%%%%%%%%%%%%%%%%%%%%%%%%%%%%
%\begin{lem}\label{nesting_halfspace}
%Let $X$ be a complete connected median space of finite rank with a transitive action. Then for any halfspace $\hh\in\HH_{x_0}$ and $\epsilon>0$, there exist an isometry $g\in Isom(X)$ such that $g(\hat{\hh})\cap\hat{\hh}=\emptyset$ and $(d(x_0,g(x_0))\leq \epsilon$.
%\end{lem}
%%%%%%%%%%%%%%%%%%%%%%%%%%%%%%%%%%%%%%%%%%%%%%%%%%%%%%%%%%%%%%%%%%%%%%%%%%%%%%%%%%%%%%%%%%%%%%%%%%%%%%%%%%%%%%%%%%%%%%%%%%%%%%%%%%%%%%%%%%%%%%%%%%%%%%%%%%%%%%%%%%%%%%%%%%%%%%%%%%
\begin{proof}[Proof of the second part of Theorem \ref{thm_principal}]
 %Let $\hh_1,\hh_2,\hh_3\in\HH_{x_0}$ be a facing triple branched at point $x_0$, that is, a pariwise disjoint triple of halfspaces in $\HH_{x_0}$. It is enough to show that for any $\epsilon$, there exist two disjoint halfspace $\hh_2',\hh_3'\in\HH_{x}$ contained in $\hh_1$ and $d(x,x_0)\leq \epsilon$. By Lemma \ref{nesting_halfspace}, there exist an isometry $g\in Isom(X)$ such that $g(\hat{\hh})\cap\hat{\hh}=\emptyset$ and $(d(x_0,g(x_0))\leq \frac{\epsilon}{2}$. If $\hh_1\subset g(H_1)$, then the halfspaces $g(H_2)$ and $g(H_3)$ are disjoint halfspace contained in $\hh_1$ and at $g(x_0)$, which proves the claim.

Let us fix a point $a_0\in X$ and show that under the assumptions of Theorem \ref{thm_principal}, any neighbourhood of $a_0$ contains infinitely many disjoint halfspaces of depth bigger than some $\epsilon>0$ inside the latter neighbourhood. We will conclude then by Theorem \ref{local_compactness} that the space is not locally compact.

%It is enough to show that for any halfspace $\hh\in\HH_{a_0}$ and any $\epsilon$, there exist a point $x\in H$ an two disjoint halfspace $\hh_1,\hh_2 \in\HH_{x}$ contained in $\hh$ and $d(x,a_0)\leq \epsilon$. Let us assume that the rank of $X$ is $n$.

The action of $Isom(X)$ being transitive, by Lemma \ref{embedded_cube}, there exists an isometrically embedded $n$-cube $C\cong([-\eta,\eta]^n,\ell^1)$ centred at $a_0$. Let us parametrize the $n$-cube by $x_1,...,x_n$ and identify $a_0$ with $(0,...,0)$. Let $\hh_1,\hh_2,\hh_3\in\HH_{a_0}$ be a facing triple. By Proposition \ref{depth_transitive_assumption}, each of the $\hh_i$ is of positive depth inside any ball centred at $x_0$. For any point $x$ inside the $n$-cube, there exists an isometry $g\in Isom(X)$ which maps $a_0$ to $x$. At least, the image of one of the $\hh_i$'s by the isometry $g$ is disjoint from the $n$-cube. Hence, for any $r\in]0,\eta[$ there exists a halfspace $\hh_r\in\HH_{(r,...,r)}$ which is disjoint from the $n$-cube and of depth bigger than some uniform $\epsilon$. The trace of the hyperplane $\hat{\hh}_r$ on the $n$-cube $C$ is a convex subset of rank less than $n$ containing the point $(r,...,r)$. Hence, it is contained in a hyperplane of the $n$-cube $C$ given by an equation of the form $x_{i_r}=r$. Thus, there exists an infinite subset $I\subseteq ]0,\eta[$ such that for any $r_1,r_2\in I$, the trace of the hyperplane $\hat{\hh}_{r_1}$ on the $n$-cube is disjoint from the trace of the hyperplane $\hat{\hh}_{r_2}$. By Lemma \ref{constraint_on_the_projection}, if the halfspaces $\hh_{r_1}$ and $\hh_{r_2}$ intersect, then projection of their intersection into the $n$-cube $C$ lies inside $\hat{\hh}_{r_1}\cap \hat{\hh}_{r_2}$. Hence, for any such $r_1$ and $r_2$, the halfspaces $\hh_{r_1}$ and $\hh_{r_2}$ are disjoint. Therefore, the set $(H_r)_{r\in I}$ give us the desired infinite family of pairwise disjoint halfspaces of depth bigger than some uniform $\epsilon >0$.

%. Let us assume then that it is contained in the hyperplane given by the equation $x_1=0$, where the variables $x_1,...,x_n$ parametrize the $n$-cube. Let us fix $0<\eta'<\frac{\eta}{2n}$ and set $a_i=(i\eta',i\eta',....,i\eta')$ where $i\in\{1,....,n\}$. Let $g_1,....,g_n\in Isom(X)$ where each $g_i$ maps $a_0$ into $a_i$. The trace of the image of the hyperplane by the isometry $g_i$ on the $n$-cube is a convex subset of rank less than $n$ containing the point $a_i$ and contained in some intrinsic hyperplane given by an equation of the form $x_{j_i}=a_i$. Let us denote the trace of each $g_i(\hat{\hh})$ on the $n$-cube by $\hat{\hh}_i$. there exists then $i_1,i_2$ such that $\hat{\hh}_{i_1}$ and $\hat{\hh}_{i_2}$ are contained in disjoint parallel hyperplane of the $n-$cube. By Helly Theorem, we deduce that $g_{i_1}(\hat{\hh})$ and $g_{i_2}(\hat{\hh})$ are disjoint. Therefore, the trace of the hyperplane $g_{i_1}^{-1}\circ g_{i_2}(\hat{\hh})$ on the $n$-cube is disjoint from  is a closed convex subset disjoint from the trace 

\end{proof}

%%%%%%%%%%%%%%%%%%%%%%%%%%%%%%%%%%%%%%%%%%%%%%%%%%%%%%%%%%%%%%%%%%%%%%%%%%%%%%%%%%%%%%%%%%%%%%%%%%%%%%%%%%%%%%%%%%%%%%%%%%%%%%%%%%%%%%%%%%%%%%%%%%%%%%%%%%%%%%%%%%%%%%%%%%%%%%%%%%%%%%%%%%%%%%%%%%%%%%%%%%%%%%%%%%%%%%%%%%%%%%%%%%%%%%%%%%%%%%%%%%%%%%%%%%%%%%%%%%%%%%%%%%%%%%%%%%%%%%%%%%%%%%%%%%%%%%%%%%%%%%%%%%%%%%%%%%%%%%%%%%%%%%%%%%%%%%%%%%%%%%%%%%%%%%%%%%%%
\subsection{A comment on a weaker assumption in Theorem \ref{thm_principal}}
%%%%%%%%%%%%%%%%%%%%%%%%%%%%%%%%%%%%%%%%%%%%%%%%%%%%%%%%%%%%%%%%%%%%%%%%%%%%%%%%%%%%%%%%%%%%%%%%%%%%%%%%%%%%%%%%%%%%%%%%%%%%%%%%%%%%%%%%%%%%%%%%%%%%%%%%%%%%%%%%%%%%%%%%%%%%%%%%%%%%%%%%%%%%%%%%%%%%%%%%%%%%%%%%%%%%%%%%%%%%%%%%%%%%%%%%%%%%%%%%%%%%%%%%%%%%%%%%%%%%%%%%%%%%%%%%%%%%%%%%%%%%%%%%%%%%%%%%%%%%%%%%%%%%%%%%%%%%%%%%%%%%%%%%%%%%%%%%%%%%%%%%%%%%%%%%%%%%
One may weaken the assumption in Theorem \ref{thm_principal} and assume the action of $Isom(X)$ to be topologically transitive instead of transitive, that is, it admits a dense orbit. The same strategy works, we divide the theorem into condition on the existence of a facing triple in the neighbourhood of a points. Theorem \ref{thm_principal} adapts into the following:
%%%%%%%%%%%%%%%%%%%%%%%%%%%%%%%%%%%%%%%%%%%%%%%%%%%%%%%%%%%%%%%%%%%%%%%%%%%%%%%%%%%%%%%%%%%%%%%%%%%%%%%%%%%%%%%%%%%%%%%%%%%%%%%%%%%%%%%%%%%%%%%%%%%%%%%%%%%%%%%%%%%%%%
\begin{thm}\label{thm_principal_bis}
Let $X$ be a complete connected median space of rank $n$ which admits a topologically transitive action.
If there exist $x\in X$ and $\epsilon>0$ such that the set of halfspaces $\HH(B(x,\epsilon))$ contains no facing triple then the space $X$ is isomorphic to $(\RR^n,\ell^1)$.
\end{thm}
%%%%%%%%%%%%%%%%%%%%%%%%%%%%%%%%%%%%%%%%%%%%%%%%%%%%%%%%%%%%%%%%%%%%%%%%%%%%%%%%%%%%%%%%%%%%%%%%%%%%%%%%%%%%%%%%%%%%%%%%%%%%%%%%%%%%%%%%%%%%%%%%%%%%%%%%%%%%%%%%%%%%%%
The second part reformulates into the following:
%%%%%%%%%%%%%%%%%%%%%%%%%%%%%%%%%%%%%%%%%%%%%%%%%%%%%%%%%%%%%%%%%%%%%%%%%%%%%%%%%%%%%%%%%%%%%%%%%%%%%%%%%%%%%%%%%%%%%%%%%%%%%%%%%%%%%%%%%%%%%%%%%%%%%%%%%%%%%%%%%%%%%%
\begin{prop}\label{prop1_bis}
Let $X$ be a complete connected median space of rank $n$ which admits a topologically transitive action.
If there exists $x\in X$ such that for any $\epsilon>0$, the set of halfspaces $\HH(B(x,\epsilon))$ contains a facing triple, then the space $X$ is not locally compact.
\end{prop}
%%%%%%%%%%%%%%%%%%%%%%%%%%%%%%%%%%%%%%%%%%%%%%%%%%%%%%%%%%%%%%%%%%%%%%%%%%%%%%%%%%%%%%%%%%%%%%%%%%%%%%%%%%%%%%%%%%%%%%%%%%%%%%%%%%%%%%%%%%%%%%%%%%%%%%%%%%%%%%%%%%%%%%
The proof of Proposition \ref{prop1_bis} follows exactly the path of the proof of the second part of Theorem \ref{thm_principal}. Regarding the proof of Theorem \ref{thm_principal_bis}. One shows that under its assumption, we obtain the same result of Corollary \ref{no_facing_triple}, that is, the space $X$ is locally isometrically modelled on $(\RR^n,\ell^1)$. The rest of the arguments follow more or less the same path.
%%%%%%%%%%%%%%%%%%%%%%%%%%%%%%%%%%%%%%%%%%%%%%%%%%%%%%%%%%%%%%%%%%%%%%%%%%%%%%%%%%%%%%%%%%%%%%%%%%%%%%%%%%%%%%%%%%%%%%%%%%%%%%%%%%%%%%%%%%%%%%%%%%%%%%%%%%%%%%%%%%%%%%%%%%%%%%%%%%%%%%%%%%%%%%%%%%%%%%%%%%%%%%%%%%%%%%%%%%%%%%%%%%%%%%%%%%%%%%%%%%%%%%%%%%%%%%%%%%%%%%%%%%%%%%%%%%%%%%%%%%%%%%%%%%%%%%%%%%%%%%%%%%%%%%%%%%%%%%%%%%%%%%%%%%%%%%%%%%%%%%%%%%%%%%%%%%%%
\section{Actions with discrete orbit}
%%%%%%%%%%%%%%%%%%%%%%%%%%%%%%%%%%%%%%%%%%%%%%%%%%%%%%%%%%%%%%%%%%%%%%%%%%%%%%%%%%%%%%%%%%%%%%%%%%%%%%%%%%%%%%%%%%%%%%%%%%%%%%%%%%%%%%%%%%%%%%%%%%%%%%%%%%%%%%%%%%%%%%%%%%%%%%%%%%%%%%%%%%%%%%%%%%%%%%%%%%%%%%%%%%%%%%%%%%%%%%%%%%%%%%%%%%%%%%%%%%%%%%%%%%%%%%%%%%%%%%%%%%%%%%%%%%%%%%%%%%%%%%%%%%%%%%%%%%%%%%%%%%%%%%%%%%%%%%%%%%%%%%%%%%%%%%%%%%%%%%%%%%%%%%%%%%%%
The aim of this section is to show that an isometric action which is Roller non elementary, Roller minimal and minimal on a complete locally compact median space of finite rank has discrete orbits. Let us first recall in the following subsection the nomenclature concerning isometric actions on the Roller compactification of a median space and some properties that we will be needing for the proof of Theorem \ref{Theorem_discrete_orbit}.
%%%%%%%%%%%%%%%%%%%%%%%%%%%%%%%%%%%%%%%%%%%%%%%%%%%%%%%%%%%%%%%%%%%%%%%%%%%%%%%%%%%%%%%%%%%%%%%%%%%%%%%%%%%%%%%%%%%%%%%%%%%%%%%%%%%%%%%%%%%%%%%%%%%%%%%%%%%%%%%%%%%%%%%%%%%%%%%%%%%%%%%%%%%%%%%%%%%%%%%%%%%%%%%%%%%%%%%%%%%%%%%%%%%%%%%%%%%%%%
\subsection{Roller Boundary}\label{Roller_boundary}
%%%%%%%%%%%%%%%%%%%%%%%%%%%%%%%%%%%%%%%%%%%%%%%%%%%%%%%%%%%%%%%%%%%%%%%%%%%%%%%%%%%%%%%%%%%%%%%%%%%%%%%%%%%%%%%%%%%%%%%%%%%%%%%%%%%%%%%%%%%%%%%%%%%%%%%%%%%%%%%%%%%%%%%%%%%%%%%%%%%%%%%%%%%%%%%%%%%%%%%%%%%%%%%%%%%%
There is a natural way to define a boundary in the case of median spaces by the means of halfspaces. Roughly speaking, directions towards infinity are characterized by maximal family of pairwise intersecting halfspaces, with an empty intersection in the space $X$. This was introduced in \cite{Roller} particularly in the case of discrete median algebra then extended and studied further in \cite{Fior_median_property} in the case of median spaces. When the median space $X$ is complete, locally convex and has compact interval, the latter boundary compactifies $X$. We call it \textbf{\textit{the Roller compactification}} and we denote it by $\overline{X}$. We denote the \textbf{\textit{Roller Boundary}} by $\partial X:=\overline{X}\backslash X$. The metric on $X$ induces a stratifications of $\overline{X}$ into components, one component corresponds to the metric completion of $X$ and the remaining ones are contained in the Roller Boundary. Each component is endowed with a median metric given by the measure on the set of halfspaces of $X$. Any action of a group $G$ on $X$ extends to an action on $\overline{X}$. We say that the action is :
\begin{itemize}
\item \textbf{\textit{Roller Elementary}} if it has a finite orbit in $\overline{X}$
\item \textbf{\textit{Roller minimal}} if it does not fix a component of $\partial X$ nor a closed convex subset of $X$.
\end{itemize}
%%%%%%%%%%%%%%%%%%%%%%%%%%%%%%%%%%%%%%%%%%%%%%%%%%%%%%%%%%%%%%%%%%%%%%%%%%%%%%%%%%%%%%%%%%%%%%%%%%%%%%%%%%%%%%%%%%%%%%%%%%%%%%%%%%%%%%%%%%%%%%%%%%%%%%%%%%%%%%%%%%%%%%%%%%%%%%%%%%%%%%%%%%%%%%%%%%%%%%%%%%%%%%
We say that an action is \textbf{\textit{minimal}} if the convex hull of any orbit is the whole space. Note that a Roller minimal action does not imply that the action is minimal (see Example \ref{example_non_minimal_action} below or the example given in the paragraph following Theorem 5.1 in \cite{Fior_tits}). The converse does not hold neither as one may consider the action of the stabilizer of point at infinity of a regular simplicial tree.
We remark that under the minimality assumption, every halfspace is \textbf{\textit{thick}}, that is, the halfspace and its complement have non-empty interior. 
%%%%%%%%%%%%%%%%%%%%%%%%%%%%%%%%%%%%%%%%%%%%%%%%%%%%%%%%%%%%%%%%%%%%%%%%%%%%%%%%%%%%%%%%%%%%%%%%%%%%%%%%%%%%%%%%%%%%%%%%%%%%%%%%%%%%%%%%%%%%%%%%%%%%%%%%%%%%%%%%%%%%%%%%%%%%%%%%%%
\begin{prop}\label{deep_halfspaces}
Let $X$ be a complete median space of finite rank which admits a minimal action. Then any halfspace is thick.
\end{prop}
%%%%%%%%%%%%%%%%%%%%%%%%%%%%%%%%%%%%%%%%%%%%%%%%%%%%%%%%%%%%%%%%%%%%%%%%%%%%%%%%%%%%%%%%%%%%%%%%%%%%%%%%%%%%%%%%%%%%%%%%%%%%%%%%%%%%%%%%%%%%%%%%%%%%%%%%%%%%%%%%%%%%%%%%%%%%%%%%%%
\begin{proof}[Proof]
Let us assume that $X$ is of rank $n$, there exist then a $n$-cube isometrically embedded into $X$ (see Lemma \ref{embedded_cube}). Let us denote by $a_0$ its center. Let $\hh\in \HH(X)$ be a halfspace in $X$. Under the assumption of the existence of a minimal action, there exist an isometry which maps the center of the cube $a_0$ into $\hh$. We obtains embedded $n$-cube which has it center inside $\hh$. The rank of $X$ being $n$, at least a lift of one canonical halfspace of the $n$-cube which contains the center, contains the halfspace $\hh^c$. We consider then a point inside the $n$-cube which is at positive distance from the corresponding halfspace. the latter point is necessarily at positive distance from $\hh^c$.
\end{proof}
%%%%%%%%%%%%%%%%%%%%%%%%%%%%%%%%%%%%%%%%%%%%%%%%%%%%%%%%%%%%%%%%%%%%%%%%%%%%%%%%%%%%%%%%%%%%%%%%%%%%%%%%%%%%%%%%%%%%%%%%%%%%%%%%%%%%%%%%%%%%%%%%%%%%%%%%%%%%%%%%%%%%%%%%%%%%%%%%%%
\par We say that a median space is \textbf{\textit{irreducible}} if it does not split as the $\ell^1$-product of two median spaces. We say that two halfspaces are strongly separated if they are disjoint and there is no halfspace which is transverse to both.\par The following lemma is borrowed from \cite{Fior_superrigidity} (Proposition 2.10, Proposition 2.12 and Lemma 4.1 therein)
%%%%%%%%%%%%%%%%%%%%%%%%%%%%%%%%%%%%%%%%%%%%%%%%%%%%%%%%%%%%%%%%%%%%%%%%%%%%%%%%%%%%%%%%%%%%%%%%%%%%%%%%%%%%%%
\begin{lem}\label{Strongly_separated_facing_triple}
Let $X$ be an irreducible median space of finite rank. Let us assume that the action of $Isom(X)$ is Roller non elementary and Roller minimal. Then there exists a strongly separated facing triple $\hh_1,\hh_2,\hh_3\in \HH(X)$ of positive depth, and a point $c\in X$ such that for any $x_i\in \hh_i$, we have $m(x_1,x_2,x_3)=c$. Note that the point $c$ lies in $\hh_1^c\cap\hh_2^c\cap\hh_3^c$.
\end{lem}

%%%%%%%%%%%%%%%%%%%%%%%%%%%%%%%%%%%%%%%%%%%%%%%%%%%%%%%%%%%%%%%%%%%%%%%%%%%%%%%%%%%%%%%%%%%%%%%%%%%%%%%%%%%%%%
%%%%%%%%%%%%%%%%%%%%%%%%%%%%%%%%%%%%%%%%%%%%%%%%%%%%%%%%%%%%%%%%%%%%%%%%%%%%%%%%%%%%%%%%%%%%%%%%%%%%%%%%%%%%%%%%%%%%%%%%%%%%%%%%%%%%%%%%%%%%%%%%%%%%%%%%%%%%%%%%%%%%%%%%%%%%%%%%%%%%%%%%%%%%%%%%%%%%%%%%%%%%%%%%%%%%%%%%%%%%%%%%%%%%%%%%%%%%%%
\subsection{Discussion on the conditions of Theorem \ref{Theorem_discrete_orbit}}\label{Subsection_discussion_condition}
%%%%%%%%%%%%%%%%%%%%%%%%%%%%%%%%%%%%%%%%%%%%%%%%%%%%%%%%%%%%%%%%%%%%%%%%%%%%%%%%%%%%%%%%%%%%%%%%%%%%%%%%%%%%%%%%%%%%%%%%%%%%%%%%%%%%%%%%%%%%%%%%%%%%%%%%%%%%%%%%%%%%%%%%%%%%%%%%%%%%%%%%%%%%%%%%%%%%%%%%%%%%%%%%%%%%
The simplest examples of locally compact median spaces having a non discrete orbit are given by $(\RR^n,\ell^1)$ and some of its median subspaces. In these cases, the orbits are even connected and are generated by hyperbolic elements. Another example is given by considering a homogeneous rooted tree where the length of the edges tends to $0$ as the levels go down. More explicitly, let us consider the following Real tree. Consider a segment $[x_0,x]$ of length $1$ and a sequence $(x_n)_{n\in\NN}\in[x_0,x]$ such that $d(x_n,x)=\frac{1}{2^n}$. Let $T$ be the real $T$ containing the segment geodesic $[x_0,x]$ such that $T\backslash [x_0,x_1]$ is a homogeneous rooted tree of valency $2$ and the branching points of the geodesic $[x_0,x]$ in $T$ are exactly the points $(x_n)_{n\in\NN^*}$ (see Figure \ref{rooted-tree} below). Then the orbit of $x$ under the group of isometries of $T$, which consists only of elliptic elements as they all stabilize the point $x_0$, is a Cantor set.\par
\begin{figure}[H]
\begin{center}
\includegraphics[scale=0.23]{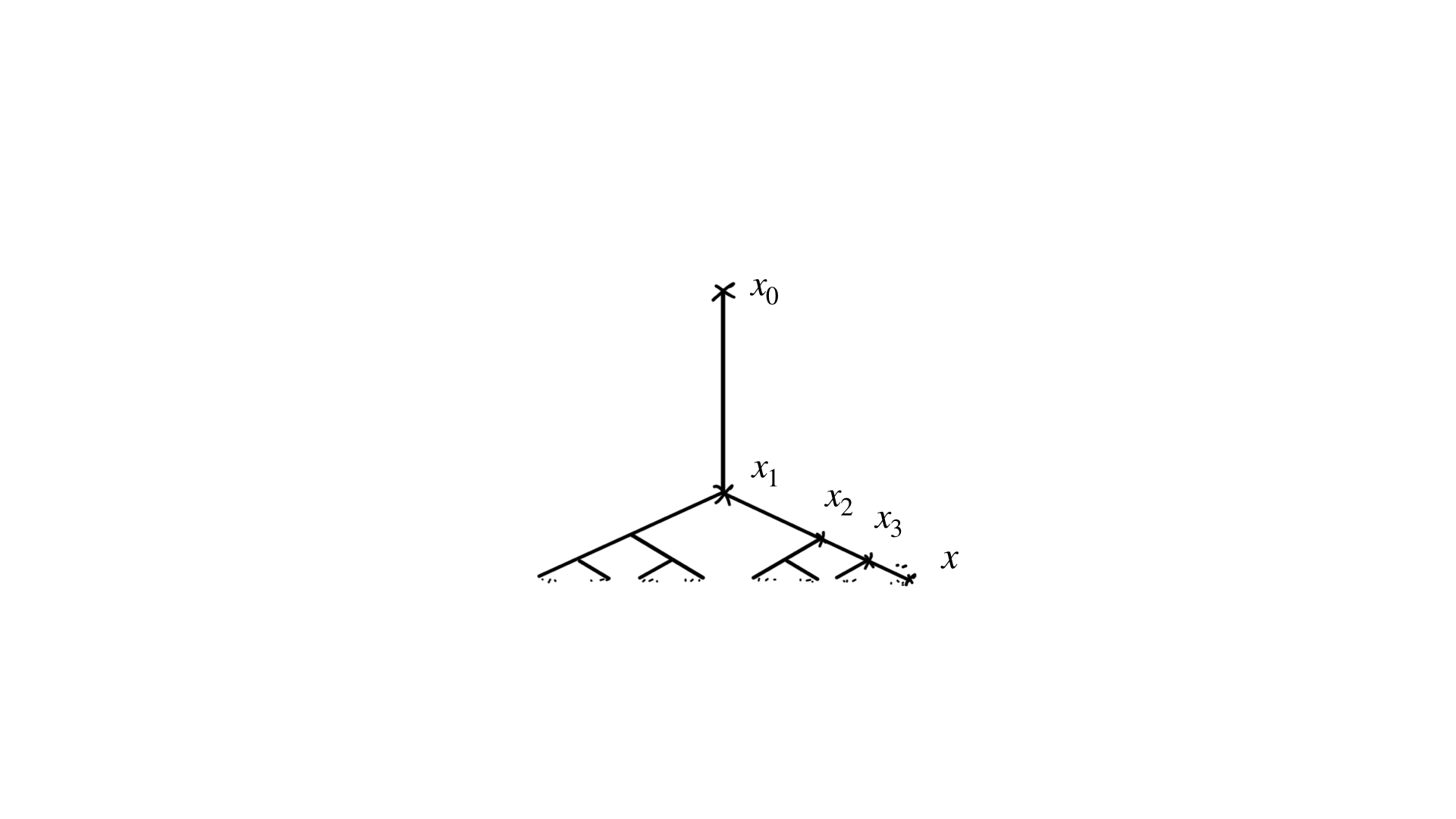}

\centering

\caption{The segment $[x_0,x_1]$ is fixed by the action of the group of isometries of the real tree.}

\label{rooted-tree}
\end{center}
\end{figure}
The above examples can be avoided by assuming that the action is Roller non elementary and Roller minimal. Note that the Roller minimality condition is also needed as one may consider a homogeneous simplicial tree of finite valency and we glue at each vertex the rooted simplicial tree seen above. Remark also that the action of the isometry group in the example above is not minimal as the singleton ${x}$ constitutes a halfspaces of empty interior. In fact it is precisely the existence of such halfspaces of empty interior which gave space to the existence of a non discrete orbit outside a ``flat", and without contradicting the local compactness. One can show that in the case of rank $1$, that is in the case where the space is an $\RR$-tree, if an action is Roller non-elementary and Roller minimal then it is minimal, assuming that the real tree is locally compact. This holds no longer in the higher rank case as shown in the following example:
%%%%%%%%%%%%%%%%%%%%%%%%%%%%%%%%%%%%%%%%%%%%%%%%%%%%%%%%%%%%%%%%%%%%%%%%%%%%%%%%%%%%%%%%%%%%%%%%%%%%%%%%%%%%%%%%%%%%%%%%%%%%%%%%%
\begin{exm}\label{example_non_minimal_action}
Let us consider the $\ell^1$-product $C:=T\times [0,+\infty[$ where $T$ is the real tree constructed above (See Figure \ref{product-rooted-tree} below). We glue then a copy of $C$ through a segment on each of the points $\tilde x_{n,g}:=(g(x_n),n)$, where $g\in Isom(T)$, and in a way that the new median space admits a reflection on each segment (See Figure \ref{tree-like}). The resulting space, that we denote by $X$, is a locally compact median space and the action of its group of isometries is Roller non elementary and Roller minimal. The closure of any orbit is the whole space. However, the action is not minimal as the fiber above $x$ inside a copy of $C$ constitutes a halfspace of empty interior and the orbit of $x$ is not discrete.

\begin{figure}[H]
\begin{center}
\includegraphics[scale=0.15]{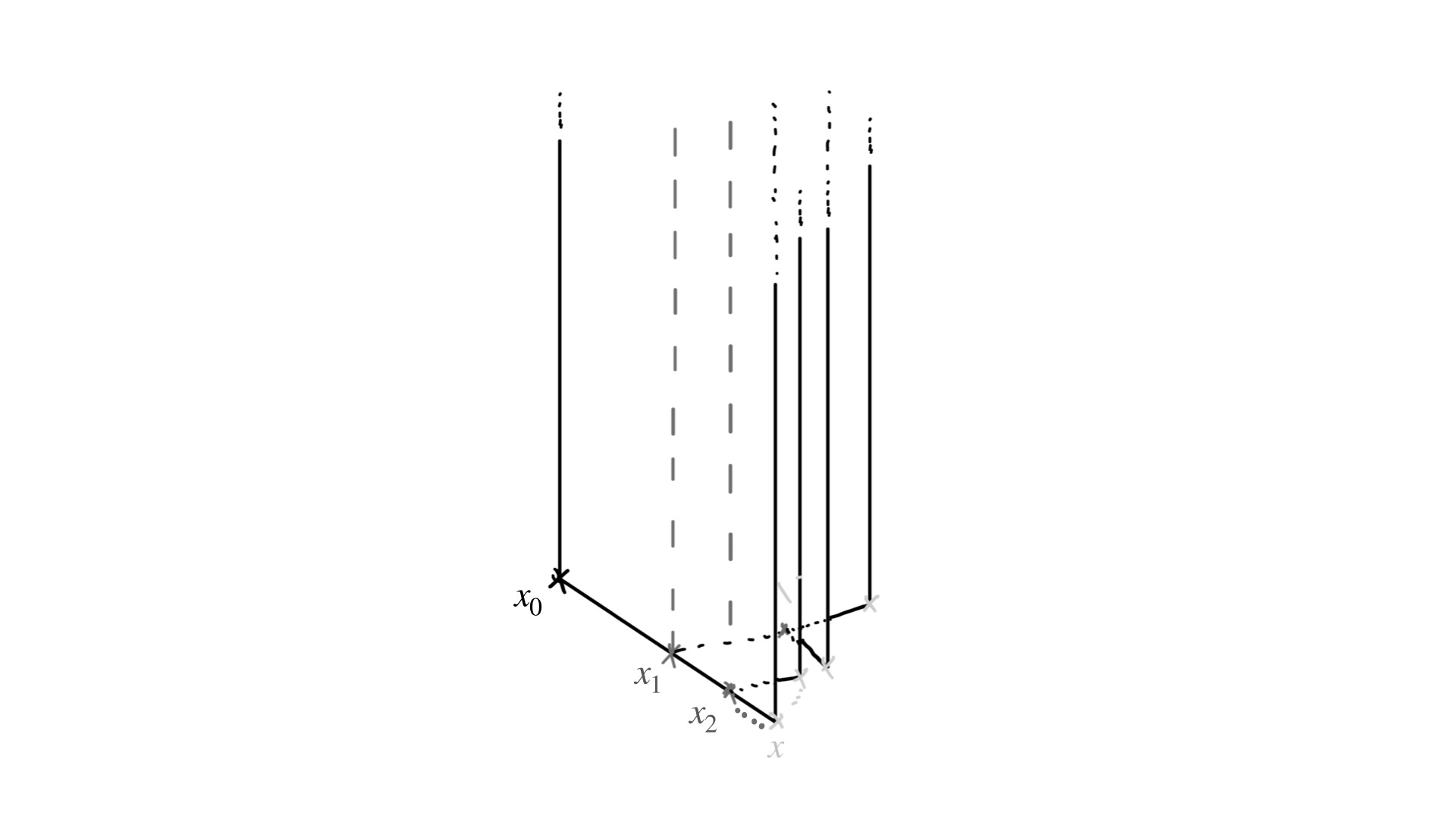}

\centering

\caption{ The orbit of $x$ in $C=T\times [0,+\infty[$ is non discrete but still the action of $G$ on $C$ is neither Roller non-elemenatary nor Roller minimal (nor minimal).}

\label{product-rooted-tree}
\end{center}
\end{figure}

\begin{figure}[H]
\begin{center}

\includegraphics[scale=0.15]{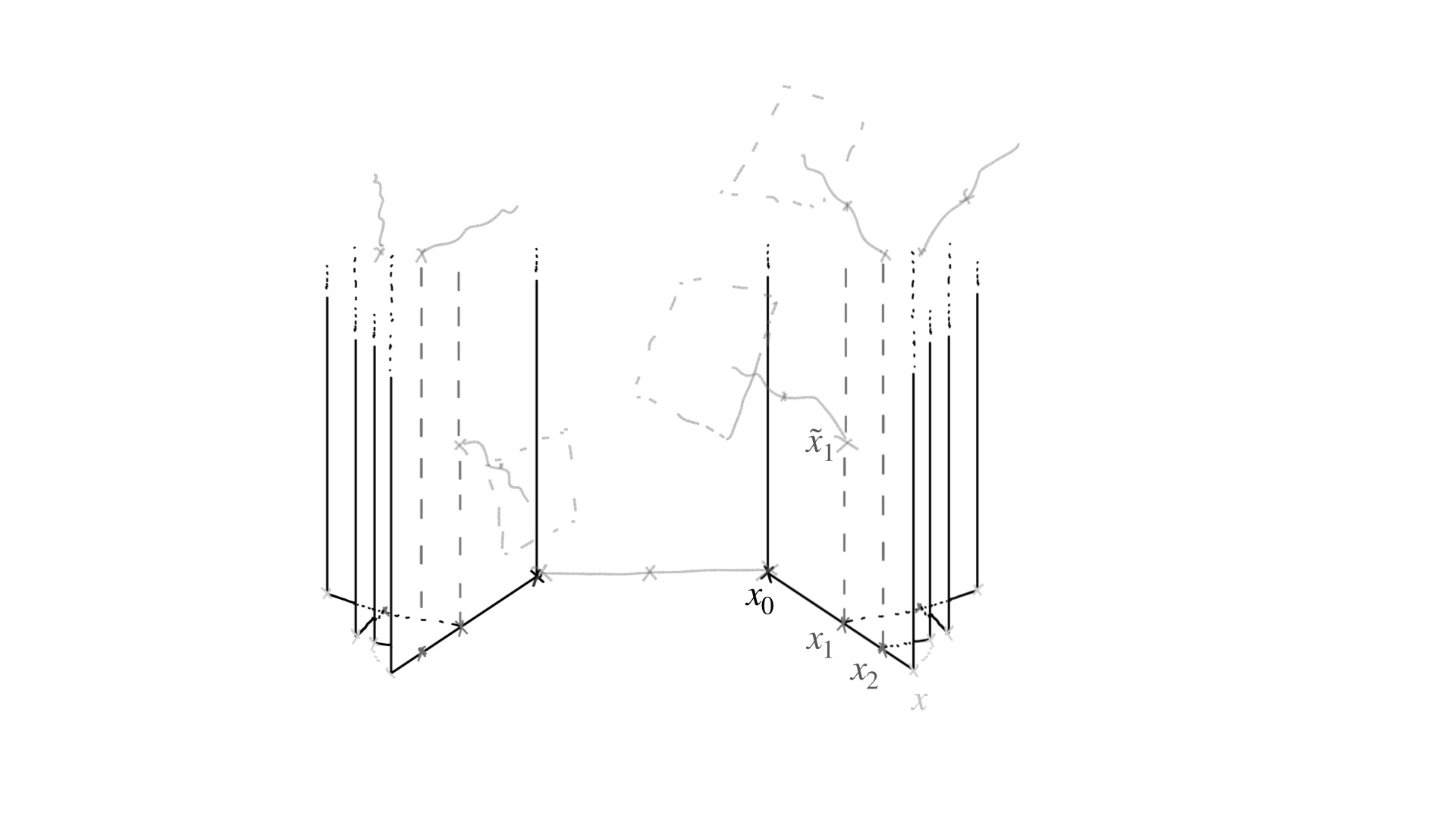}

\centering

\caption{The subgroup generated by the inversions along the edges of the space $X$ acts Roller non elementarily and Roller minimally on the latter.}

\label{tree-like}
\end{center}
\end{figure}
\end{exm}
%%%%%%%%%%%%%%%%%%%%%%%%%%%%%%%%%%%%%%%%%%%%%%%%%%%%%%%%%%%%%%%%%%%%%%%%%%%%%%%%%%%%%%%%%%%%%%%%%%%%%%%%%%%%%%
%%%%%%%%%%%%%%%%%%%%%%%%%%%%%%%%%%%%%%%%%%%%%%%%%%%%%%%%%%%%%%%%%%%%%%%%%%%%%%%%%%%%%%%%%%%%%%%%%%%%%%%%%%%%%%%%%%%%%%%%%%%%%%%%%%%%%%%%%%%%%%%%%%%%%%%%%%%%%%%%%%%%%%%%%%%%%%%%%%%%%%%%%%%%%%%%%%%%%%%%%%%%%%%%%%%%%%%%%%%%%%%%%%%%%%%%%%%%%%
\subsection{Proof of Theorem \ref{Theorem_discrete_orbit}}
%%%%%%%%%%%%%%%%%%%%%%%%%%%%%%%%%%%%%%%%%%%%%%%%%%%%%%%%%%%%%%%%%%%%%%%%%%%%%%%%%%%%%%%%%%%%%%%%%%%%%%%%%%%%%%%%%%%%%%%%%%%%%%%%%%%%%%%%%%%%%%%%%%%%%%%%%%%%%%%%%%%%%%%%%%%%%%%%%%%%%%%%%%%%%%%%%%%%%%%%%%%%%%%%%%%%
Any complete connected median space $X$ is geodesic (see Lemma 13.3.2 \cite{Bowd_median-algebras}), hence by Hopf-Rinow Theorem, showing that a median space $X$ is not locally compact is equivalent to finding a closed ball which is not compact.\par 
Let us first prove the following combinatorial lemma:
\begin{lem}\label{infinite_halfpsace}
Let $X$ be a median algebra of finite rank and let $\HH \subseteq \HH(X)$ be an infinite subset of halfspaces such that any $\hh_1,\hh_2\in \HH$ are either transverse or disjoint. Then there exists an infinite subset $\HH' \subseteq \HH$ of pairwise disjoint halfspaces.
\end{lem}
%%%%%%%%%%%%%%%%%%%%%%%%%%%%%%%%%%%%%%%%%%%%%%%%%%%%%%%%%%%%%%%%%%%%%%%%%%%%%%%%%%%%%%%%%%%%%%%%%%%%%%%%%%%%%%%%%%%%%%%%%%%%%%%%%%%%%%%%%%%%%%%%%%%%%%%%%%%%%%%%%%%%%%%%%%%%%%%%%%%%%%%%%%%%%%%%%%%%%%%%%%%%%%
\begin{proof}
Let us consider the dual graph $\Gamma$ of $\HH$, that is, the non oriented graph whose vertices are the halfspaces of $\HH$ and two vertices are joined by and edge if the halfspaces labelling the vertices are transverse, i.e. $\Gamma:=(V,E)$ such that $V=\HH$ and $(\hh_1,\hh_2)\in E$ if and only if $\hh_1$ and $\hh_2$ are transverse. Thus, finding an infinite family of pairwise disjoint halfspaces in $\HH$ translates into finding an infinite subset $A$ of the graph $\Gamma$ consisting of vertices which are pairwise non adjacent. As the rank of the space $X$ is finite, the graph $\Gamma$ corresponds to the $1$-skeleton of finite-dimensional simplicial complex. The set of vertices being infinite, the graph is unbounded with regard to its combinatorial metric. Therefore such subset $A$ exists.
\end{proof}
%%%%%%%%%%%%%%%%%%%%%%%%%%%%%%%%%%%%%%%%%%%%%%%%%%%%%%%%%%%%%%%%%%%%%%%%%%%%%%%%%%%%%%%%%%%%%%%%%%%%%%%%%%%%%%%%%%%%%%%%%%%%%%%%%%%%%%%%%%%%%%%%%%%%%%%%%%%%%%%%%%%%%%%%%%%%%%%%%%%%%%%%%%%%%%%%%%%%%%%%%%%%%%
Let $G$ be a group acting by isometries on a median space $X$. We denote by $Stab_G(x)$ the subgroup of $G$ consisting of isometries which stabilize the point $x$. If $G=Isom(X)$ we simply write $Stab(x)$. We have the following proposition:
%%%%%%%%%%%%%%%%%%%%%%%%%%%%%%%%%%%%%%%%%%%%%%%%%%%%%%%%%%%%%%%%%%%%%%%%%%%%%%%%%%%%%%%%%%%%%%%%%%%%%%%%%%%%%%%%%%%%%%%%%%%%%%%%%%%%%%%%%%%%%%%%%%%%%%%%%%%%%%%%%%%%%%%%%%%%%%%%%%%%%%%%%%%%%%%%%%%%%%%%%%%%%%
\begin{prop}\label{finite_orbit_elleptic_subgroup}
Let $X$ be a complete connected locally compact median space of finite rank. Let us assume that $Isom(X)$ acts minimally on $X$. Then for any $x_0,x\in X$ the orbit $Stab(x_0).x$ is finite.
\end{prop}
%%%%%%%%%%%%%%%%%%%%%%%%%%%%%%%%%%%%%%%%%%%%%%%%%%%%%%%%%%%%%%%%%%%%%%%%%%%%%%%%%%%%%%%%%%%%%%%%%%%%%%%%%%%%%%%%%%%%%%%%%%%%%%%%%%%%%%%%%%%%%%%%%%%%%%%%%%%%%%%%%%%%%%%%%%%%%%%%%%%%%%%%%%%%%%%%%%%%%%%%%%%%%%
Before proving the proposition, we will be needing some results. We have the following lemma which states that any point $x_0\in X$ is determined by the couple $x\in X$ and $\HH_{x}\cap \HH(x_0,x)$: 
\begin{lem}\label{uniquely_determiner_by_branched_halfspace}
Let $X$ be a complete connected median space of finite rank. Let us consider $x,x_0\in X$ and set $C:=\displaystyle{\bigcap_{\hh\in\HH_{x}\cap \HH(x_0,x)}\hh}$. We have then $\pi_C(x_0):=x$.
\par Note that the convex subset $C$ is closed by Remark \ref{closed_branched_halfspaces}, hence the nearest point projection onto $C$ exists. 
\end{lem}
\begin{proof}
By Proposition \ref{halfspace_gateconvex} and Proposition \ref{strenghtening_separation_theorem}, we have :
\[
C\cap [x_0,x]=(\bigcap_{\hh\in\HH_{x}\cap \HH(x_0,x)}\hh)\cap [x_0,x]= \{x\}
\]
We conclude by Lemma \ref{constraint_on_the_projection} that $\pi_C(x_0)=x$.
%As $x\in C$, we have $\pi_C(x_0)\in[x,x_0]$. Any halfspaces which separates a point in $[x,x_0]$ (in particular $\pi_C(x)$) from $x$ must also separates $x_0$ from $x$. In the other hand, there is no halfspace $\hh\in\HH_{x}$ which separates $x$ from $\pi_C(x_0)$ as the latter point lies in the intersection of such halfspaces. By Lemma \ref{strenghtening_separation_theorem}, we conclude that $\pi_C(x_0)=x$.
\end{proof}
%%%%%%%%%%%%%%%%%%%%%%%%%%%%%%%%%%%%%%%%%%%%%%%%%%%%%%%%%%%%%%%%%%%%%%%%%%%%%%%%%%%%%%%%%%%%%%%%%%%%%%%%%%%%%%%%%%%%%%%%%%%%%%%%%%%%%%%%%%%%%%%%%%%%%%%%%%%%%%%%%%%%%%%%%%%%%%%%%%%%%%%%%%%%%%%%%%%%%%%%%%%%%%
\begin{lem}\label{stabilizer_of_points_does_not_nest_halfspaces}
Let $X$ be a complete connected median space of finite rank and let $x\in X$. Then for any isometry $g\in Stab(x_0)$ and a closed halfspace $\hh\in\HH(X)$ such that $x_0\in\hh^c$, the halfspaces $\hh$ and $g.\hh$ are either transverse or disjoint. 
\end{lem}
%%%%%%%%%%%%%%%%%%%%%%%%%%%%%%%%%%%%%%%%%%%%%%%%%%%%%%%%%%%%%%%%%%%%%%%%%%%%%%%%%%%%%%%%%%%%%%%%%%%%%%%%%%%%%%%%%%%%%%%%%%%%%%%%%%%%%%%%%%%%%%%%%%%%%%%%%%%%%%%%%%%%%%%%%%%%%%%%%%%%%%%%%%%%%%%%%%%%%%%%%%%%%%
\begin{proof}
Let us consider $g\in Stab(x_0)$ and a closed halfspace $\hh\in \HH'$. As both $\hh^c$ and $(g.\hh)^c$ contains $x_0$, it is enough to show that we have $g.\hh\subseteq \hh$ if and only if $g.\hh=\hh$. Note that the same conclusion will yield with regards to the case when $\hh\subseteq g.\hh$ as we have $g.\hh\subseteq \hh$ if and only if $\hh\subseteq g^{-1}.\hh$. Let us assume then that $g.\hh\subseteq \hh$. We set $\tilde{x}_0:=\pi_\hh(x_0)$ and first show that $\pi_{g.\hh}(x_0)=\tilde{x}_0$. In one hand, We have $\tilde{x}_0\in[\pi_{g.\hh}(x_0),x_0]$, which implies that :
\[
d(\pi_{g.\hh}(x_0),x_0)=d(\pi_{g.\hh}(x_0),\tilde{x}_0)+d(\tilde{x}_0,x)
\]
 In the other hand, we have 
\[
d(x_0,\tilde{x}_0)=d(g.x_0,g.\tilde{x}_0)=d(x_0,\pi_{g.\hh}(g(x_0)))=d(x_0,\pi_{g.\hh}(x_0))
\]
Hence, we have $d(\tilde{x}_0,\pi_{g.\hh}(x_0))=0$.\par
Finally, we deduce that for any point $a\in\hh$ we have $\tilde{x}_0\in[a,x_0]$. As $x_0\in(g.\hh)^c$ and $\tilde{x}_0\in(g.\hh)$, the point $a$ cannot lie outside $g.\hh$. Therefore, we do have $g.\hh=\hh$.
\end{proof}
%%%%%%%%%%%%%%%%%%%%%%%%%%%%%%%%%%%%%%%%%%%%%%%%%%%%%%%%%%%%%%%%%%%%%%%%%%%%%%%%%%%%%%%%%%%%%%%%%%%%%%%%%%%%%%%%%%%%%%%%%%%%%%%%%%%%%%%%%%%%%%%%%%%%%%%%%%%%%%%%%%%%%%%%%%%%%%%%%%%%%%%%%%%%%%%%%%%%%%%%%%%%%%
\begin{proof}[Proof of Proposition \ref{finite_orbit_elleptic_subgroup}]
Let us consider $x_0,x\in X$ and $g\in Stab(x_0)$. We denote by $\HH'_x$ the set of minimal halfspaces in $\HH_x\cap \HH(x_0,x)$. By Lemma \ref{uniquely_determiner_by_branched_halfspace}, any point $x\in X$ is determined by the point $x_0$ and the set $\HH'_x$. Hence, it is enough to show that the orbit of any halfspace by $Stab(x_0)$ is finite. By Lemma \ref{stabilizer_of_points_does_not_nest_halfspaces}, the union of the orbit of each halfspace in $\HH'_x$ under $Stab(x_0)$ constitutes a family of halfspaces which are either transverse or disjoint. The space $X$ being assumed to be locally compact, the finiteness of the latter family is ensured by Lemma \ref{infinite_halfpsace} and Theorem \ref{local_compactness}.
\end{proof}
%%%%%%%%%%%%%%%%%%%%%%%%%%%%%%%%%%%%%%%%%%%%%%%%%%%%%%%%%%%%%%%%%%%%%%%%%%%%%%%%%%%%%%%%%%%%%%%%%%%%%%%%%%%
\begin{proof}[Proof of Theorem \ref{Theorem_discrete_orbit}]
Let us show that if a median space $X$ admits an action which is Roller minimal and Roller non elementary with a non discrete orbit, then it is not locally compact. Let us set $G:=Isom(X)$ and let $x_0\in X$ such that $G.x_0$ is non discrete. By Lemma \ref{Strongly_separated_facing_triple}, there exists a facing triple $\hh_1,\hh_2,\hh_3\in\HH(X)$ which are uniquely determined by a point $x\in \hh_1^c\cap\hh_2^c\cap\hh_3^c$ in the sense that for any $x_i\in\hh_i$, we have $m(x_1,x_2,x_3)=x$. Let us fix $R>0$ and let $K\subseteq Isom(X)$ such that $d(x_0,g_i.x_0)\leq R$ for any $g\in K$. As the orbit of $x_0$ under $G$ is not discrete, the subset $ K$ is infinite. If $ K.x$ is finite, this implies that the orbit of $x_0$ under $Stab(x)$ is infinite. Hence, by Proposition \ref{finite_orbit_elleptic_subgroup} the space $X$ would not be locally compact. Let us assume then that $K.x$ is infinite and show that there is a closed ball centred at $x$ which is not compact. By Proposition \ref{convex_hull_ball} and Theorem \ref{local_compactness}, it is enough to find an infinite subset in $\HH :=K.\hh_1\cup K.\hh_2\cup K.\hh_3$ which consists of halfspaces which are pairwise disjoint. If $\HH$ does not contain an infinite chain then by considering the minimal element of each maximal chain, one obtain a subfamily of halfspaces which are either transverse of disjoint. Hence, by Lemma \ref{infinite_halfpsace}, there exist an infinite subfamily of pairwise disjoint halfspaces. Let us assume then that there exists an infinite countable chain $ \HH_1\subseteq  \HH$. As $\hh_1$, $\hh_2$ and $\hh_3$ are disjoints, the chain $\HH_1$ is given by $(g_{1,n}.\hh_{i_n})_{n\in\NN}$ where $g_{1,n}\in K$ and $i_n\in\{1,2,3\}$. Let us set $\HH':=\displaystyle{\bigcup_{i\in\NN}(g_{1,n}.\hh_1\cup g_{1,n}.\hh_2\cup g_{1,n}.\hh_3 )}$ and note that $\HH'\backslash \HH_1$ is infinite. Again, If the subset $\HH'\backslash \HH_1$ does not contain an infinite chain then we are done. Let us assume then that $\HH'\backslash \HH_1$ contains an infinite chain $\HH_2$. Such chain is given by $(g_{2,n}.\hh_{j_n})_{n\in\NN}$ where $g_{2,n}\in K$ and $j_n\in\{1,2,3\}$. 
Note that as the halfspaces $\hh_1$, $\hh_2$ and $\hh_3$ are pairwise strongly separated, for any isometry $g\in G$, the halfspace $g.\hh_i$ cannot intersect two halfspaces in $\{\hh_1,\hh_2,\hh_3\}$. Hence, if for each $i\in\{1,2,3\}$ such that $g.\hh_i$ intersects a halfspace in $\{\hh_1,\hh_2,\hh_3\}$, then there exists a permutation $\sigma\in S_3$ such that $g.\hh_i$ intersects only the halfspace $\hh_{\sigma (i)}$. In the latter case, we necessarily have $g(x)=x$ as for any $x_i\in\hh_i\cap \hh_{\sigma (i)}$ we have $m(x_1,x_2,x_3)=x$ and $m(x_1,x_2,x_3)=g(x)$.
%As $\hh_i$ and $\hh_j$ are strongly separated, for any isometry $g\in Isom(X)$ such that $g(x)\neq x$, at least one of the $\hh_i$'s verifies $g.\hh_i\cap (\hh_1\cup\hh_2\cup\hh_3)=\emptyset$.
 As we are considering the isometries $g\in K$ such that $g(x)\neq x$, we conclude that the infinite subset $ \displaystyle{\bigcup_{i\in\NN}(g_{2,n}.\hh_1\cup g_{2,n}.\hh_2\cup g_{2,n}.\hh_3 ) \backslash (\HH_1\cup\HH_2)}$ consists of pairwise disjoint halfspaces, which completes the proof.
\end{proof}
\newpage
\bibliographystyle{plain}
%\nocite{*}
\bibliography{bibliography}

\end{document}